\documentclass[leqno]{amsart}
\usepackage{amsmath,amsthm,amssymb,latexsym,graphicx, enumerate, mathtools,mathrsfs,cite,bigints}

\newcommand{\NN}{{\mathbb N}}

\newcommand{\RR}{{\mathbb R}}

\newcommand{\ZZ}{{\mathbb Z}}

\newcommand{\abs}[1]{ \left| #1 \right|}

\newcommand{\dc}{\on{DC}}
\newcommand{\del}{\partial}
\newcommand{\diam}[1]{\on{diam}\left(#1\right)}

\newcommand{\dist}{\on{dist}}
\newcommand{\eps}{\varepsilon}
\newcommand{\floor}[1]{\lfloor #1 \rfloor}
\newcommand{\Id}{\mathrm{Id}}

\newcommand{\loc}{\mathrm{loc}}

\newcommand{\nor}[2]{\left\|#1\right\|_{#2}}
\newcommand{\odd}{\mathrm{odd}}
\newcommand{\oline}[1]{\overline{#1}}
\newcommand{\oo}{\infty}
\newcommand{\osc}{\on{osc}}
\newcommand{\pars}[1]{\left(#1\right)}
\newcommand{\rad}{\mathrm{rad}}

\newcommand{\supp}{\on{supp }}

\newcommand{\Var}{\mbf{Var}}

\newcommand{\mcl}{\mathcal}

\newcommand{\mbf}{\mathbf}
\newcommand{\on}{\operatorname}

\newtheorem{lemma}{Lemma}
\newtheorem{proposition}{Proposition}
\newtheorem{theorem}{Theorem}
\newtheorem{question}{Question}

\newtheorem{definition}{Definition}

\numberwithin{equation}{section}
\numberwithin{lemma}{section}
\numberwithin{proposition}{section}
\numberwithin{theorem}{section}
\numberwithin{corollary}{section}
\numberwithin{definition}{section}

\begin{document}
\title{Interpolation results for pathwise Hamilton-Jacobi equations}
\author[P.-L. Lions, B. Seeger, P. Souganidis]{Pierre-Louis Lions$^{1,3,6}$, Benjamin Seeger$^{1,4,7}$, Panagiotis Souganidis$^{2,5,8}$}
\address{$^1$Universit\'e Paris-Dauphine \& Coll\`ege de France \\ Place du Mar\'echal de Lattre de Tassigny \\ 75016 Paris, France}
\address{$^2$University of Chicago \\ 5734 S. University Ave. \\ Chicago, IL 60637}
\email{$^3$lions@ceremade.dauphine.fr, $^4$seeger@ceremade.dauphine.fr, $^5$souganidis@math.uchicago.edu}

\thanks{$^6$Partially supported by the Air Force Office for Scientific Research  grant FA9550-18-1-0494\\
$^7$Partially supported by the National Science Foundation Mathematical Sciences Postdoctoral Research Fellowship under Grant Number DMS-1902658\\
$^8$Partially supported by the National Science Foundation grants DMS-1600129 and DMS-1900599, the Office for Naval Research grant N000141712095 and the Air Force Office for Scientific Research grant FA9550-18-1-0494}

\subjclass[2010]{60H15, 46B70}
\keywords{Pathwise Hamilton-Jacobi equations, interpolation}

\date{\today}

\maketitle

\begin{abstract}
	We study the interplay between the regularity of paths and Hamiltonians in the theory of pathwise Hamilton-Jacobi equations with the use of interpolation methods. The regularity of the paths is measured with respect to Sobolev, Besov, H\"older, and variation norms, and criteria for the Hamiltonians are presented in terms of both regularity and structure. We also explore various properties of functions that are representable as the difference of convex functions, the largest space of Hamiltonians for which the equation is well-posed for all continuous paths. Finally, we discuss some open problems and conjectures.
\end{abstract}

\section{Introduction}

The objective of this paper is to study the well-posedness of pathwise viscosity solutions for the initial value problem
\begin{equation}\label{E:main}
	du = \sum_{i=1}^m H^i(Du) \cdot dW^i \quad \text{in } \RR^d \times (0,\oo) \quad \text{and} \quad u(\cdot,0) = u_0 \quad \text{on } \RR^d,
\end{equation}
where $H = (H^1,H^2,\ldots, H^m) \in C(\RR^d,\RR^m)$, $W = (W^1,W^2,\ldots, W^m) \in C([0,\oo),\RR^m)$, and $u_0 \in UC(\RR^d)$, the space of uniformly continuous functions on $\RR^d$. In particular, we aim to expand the understanding of \eqref{E:main} by analyzing the interplay between the properties of the Hamiltonian $H$ and the path $W$.

To date, the theory of solutions of \eqref{E:main} falls broadly into two categories, depending on the assumed regularity of the path $W$.

In the first case, which is classical, the path $W$ is continuously differentiable and the Hamiltonian $H$ is continuous, and \eqref{E:main} is understood using the Crandall-Lions theory of viscosity solutions (see \cite{CIL}). In this setting, $dW$ stands for the continuous function $\frac{d}{dt} W(t) = \dot W(t)$, and ``$\cdot$'' denotes multiplication. As a consequence of the evolution structure of the equation, the theory also extends to paths with $\dot W \in L^1$ or paths of bounded variation; see Ishii \cite{I} or Lions and Perthame \cite{LP}. 

The second class of problems was studied by Lions and Souganidis \cite{LS1, LS2, LS3, LS4, Snotes}, who introduced the notion of pathwise viscosity solutions of \eqref{E:main} for arbitrary continuous paths $W$. In these works, appropriately defined sub- and super-solutions are shown to satisfy a comparison principle, and, hence, the uniqueness of solutions is proved. Moreover, the equation is stable with respect to the driving paths in the topology of uniform convergence. That is, the solution $u$ of \eqref{E:main} can be identified as the unique function such that, if
\begin{equation}\label{approxpaths}
	(W_n)_{n \in \NN} \subset W^{1,1}([0,T],\RR^m), \quad \lim_{n \to \oo} W_n = W \quad \text{uniformly,}
\end{equation}
and $(u_n)_{n \in \NN} \subset UC(\RR^d \times [0,T]$ are the classical viscosity solutions of
\begin{equation}\label{E:introapprox}
	u_{n,t} = \sum_{i=1}^m H^i(Du_n) \cdot \dot W^i_n \quad \text{in } \RR^d \times (0,\oo) \quad \text{and} \quad u(\cdot,0) = u_0 \quad \text{on } \RR^d,
\end{equation}
then, as $n \to \oo$, $u_n$ converges uniformly on $\RR^d \times [0,T]$ to $u$.

In \cite{LS1}, the well-posedness of \eqref{E:main} is established for $W \in C([0,T],\RR^m)$ under the condition that $H \in C^2(\RR^d)$. This is extended to less regular Hamiltonians in \cite{LS2}, where it is proved that \eqref{E:main} is well-posed for all continuous paths and all choices of initial data if and only if, for every $i = 1,2,\ldots,m$,
\[
	H^i \in \dc(\RR^d) := \{ H \in C(\RR^d) : H = H_1 - H_2 \text{ for some convex functions } H_1 \text{ and } H_2 \}.
\]

The condition that each $H^i$ be equal to a difference of convex functions is much weaker than $H^i \in C^2(\RR^d)$, and covers a variety of interesting examples. For instance, the results of \cite{LS2} allow for the study of the geometric equation
\[
	du = |Du| \cdot dW,
\]
which models interface motion with the prescribed normal velocity $dW$. 

Nevertheless, $\dc$-Hamiltonians satisfy a variety of restrictions not shared by generic continuous functions. Indeed, they are locally Lipschitz, as well as twice-differentiable almost everywhere. Thus, for example, if $0 < \gamma < 1$, the space $C^{1,\gamma}(\RR^d)$ is not contained in $\dc(\RR^d)$. Hence, according to \cite{LS2}, for any $H \in C^{1,\gamma}(\RR^d) \backslash \dc(\RR^d)$, there exist continuous paths $W$ and approximations as in \eqref{approxpaths} such that the corresponding solutions of \eqref{E:introapprox} can have multiple limits or exhibit blow-up.

On the other hand, the motivation for studying the equation \eqref{E:main} comes from applications in which $W$ is, say, the sample path of a stochastic process, such as a Brownian motion. Such paths are nowhere differentiable and of unbounded variation on any time interval. However, they possess many properties not shared by generic continuous paths, like, for example, H\"older, Sobolev, or Besov regularity, or finite $p$-variation for some $p > 1$. It is natural to expect that the well-posedness of \eqref{E:main} can be established for more regular paths and Hamiltonians not belonging to $\dc$, and, in particular, the solution of \eqref{E:main} can still be identified as the limit of solutions of \eqref{E:introapprox} for appropriate approximating sequences $(W_n)_{n \in \NN}$.

We accomplish this by interpolating between the two regimes described above. For various examples of spaces $\mathscr H \subset C(\RR^d)$ that contain functions not belonging to $\dc(\RR^d)$, we identify spaces $\mathscr P \supset W^{1,1}([0,T],\RR^m)$ with the property that, given $u_0 \in UC(\RR^d)$, $H \in \mathscr H$, and $W \in \mathscr P$, there exists a unique function $u \in UC(\RR^d \times [0,T])$ such that, if $(W_n)_{n \in \NN} \subset W^{1,1}([0,T])$ is a sequence satisfying
\[
	\lim_{n \to \oo} \sup_{t \in [0,T]} |W_n(t) - W(t)| = 0 \quad \text{and} \quad \sup_{n \in \NN} \nor{W_n}{\mathscr P} < \oo,
\]
then, as $n \to \oo$, the solution $u_n$ of \eqref{E:introapprox} converges uniformly in $\RR^d \times [0,T]$ to $u$.

The interplay between the regularity of $H$ and $W$ naturally imposes some restrictions on the possibilities for $\mathscr H$ and $\mathscr P$. Formal interpolation arguments indicate that, if the path space $\mathscr P$ measures regularity of the paths of degree $\alpha \in (0,1)$, in some sense, then the space $\mathscr H$ should contain Hamiltonians with regularity of order $2(1-\alpha)$, and the results we prove support this hypothesis.

Throughout the paper, we consider Hamiltonians that depend only on the gradient. Different methods are required if $H$ depends on $u$, as described in \cite{LS3,Snotes}. In fact, there is not a satisfactory theory for Hamiltonians depending on both $Du$ and $u$ unless the dependence on one is linear. 

When the Hamiltonian depends on the space variable $x$, the question of well-posedness becomes more complicated. Indeed, more regularity and structural requirements are needed for the Hamiltonian, as is described in more detail in \cite{Snotes}. Some particular and instructive examples are explored in the works of Friz, Gassiat, Lions, and Souganidis \cite{FGLS}, Lions and Souganidis \cite{LSpreprint}, and Seeger \cite{Se}. If $H$ is linear in $Du$, more general spatial dependence can be treated using either stochastic calculus or the theory of rough paths, as in Caruana, Friz, and Oberhauser \cite{CFO} and Diehl, Friz, and Oberhauser \cite{DFO}. In any of these settings, the question of using interpolation between existing regimes of well-posedness remains completely open.

\subsection{Some representative results} \label{SS:examples}
In order to give a flavor of the results to follow later in the paper, we discuss several examples of spaces of Hamiltonians and paths for which the above program can be carried out. These are consequences of the main theorems, which involve real interpolation spaces (see Theorem \ref{T:intromain}, Theorem \ref{T:intromainweak}, and Section \ref{S:mainresult}). 

Throughout the rest of the introduction, to simplify the presentation, we take $m = 1$, that is, $H$ and $W$ are both scalar valued. There is no loss of generality in doing so, as all the results continue to hold for $m > 1$.

Below, for $\alpha \in (0,1)$, $C^{0,\alpha}$ is the space of $\alpha$-H\"older continuous paths; for $p \in [1,\oo)$, $V_p$ denotes the space of paths of finite $p$-variation (see subsection \ref{SS:variation}); and $B^s_{pq}$ is the Besov space of parameters $s > 0$ and $1 \le p,q \le \oo$ (see subsection \ref{SS:besov}). Recall that, if $R > 0$ and $f \in B^s_{pq}(B_R)$, then $f \in L^p(B_R)$ and
\[
	\left[\int_0^1\pars{ \frac{ \sup_{|y| \le t} \nor{ f(\cdot + y) + f(\cdot - y) - 2f}{L^p(B_R)} }{t^s} }^q \frac{dt}{t} \right]^{1/q} < \oo.
\]
For more definitions and notation, see subsection \ref{SS:notation} and Section \ref{S:spaces} below.

\begin{theorem}\label{T:introBesovresult}
	Fix $\alpha \in (0,1)$ and assume that, for all $R > 0$, $H \in B^{2(1-\alpha)}_{\oo,1}(B_R)$. 
	\begin{enumerate}[(a)]
	\item If $W \in C^{0,\alpha}([0,T],\RR)$, then there exists a unique $u \in UC(\RR^d \times [0,T])$ such that, if $(W_n)_{n \in \NN} \subset W^{1,1}([0,T],\RR)$, $\lim_{n \to \oo} W_n = W$ uniformly, and 
	\begin{equation}\label{Holdercontrol}
		\sup_{n \in \NN} \nor{W_n}{C^{0,\alpha}} < \oo,
	\end{equation}
	then, as $n \to \oo$, the solution $u_n$ of \eqref{E:introapprox} converges uniformly in $\RR^d \times [0,T]$ to $u$.
	\item The same result is true if $W \in V_{1/\alpha}([0,T],\RR^m)$ and, instead of \eqref{Holdercontrol}, the approximating sequence satisfies
	\[
		\sup_{n \in \NN} \nor{W_n}{V_{0,1/\alpha}} < \oo,
	\]
	or if $W \in B^\alpha_{p,1}([0,T])$ for some $p > 1/\alpha$ and the approximating sequence satisfies
	\[
		\sup_{n \in \NN} \nor{W_n}{B^\alpha_{p,1}} < \oo.
	\]
	\end{enumerate}
\end{theorem}

We note the condition that $H \in B^{2(1-\alpha)}_{\oo,1}$ is satisfied if, for example,
\[
	H \in C^{0,\beta}(\RR^d) \text{ for } 2(1-\alpha) < \beta < 1  \quad \text{or} \quad H \in C^{1,\beta-1}(\RR^d)  \text{ for } 1 \le 2(1-\alpha) < \beta.
\]

We discuss now what Theorem \ref{T:introBesovresult} says when $W$ is a Brownian motion. It is well known (see Stroock and Varadhan \cite{SV} and Friz and Victoir \cite{FV}) that, with probability one, Brownian paths belong to $C^{0,\alpha}$ and $V_{1/\alpha}$ for $\alpha < 1/2$, and fail to belong to the same spaces for $\alpha \ge 1/2$. It is also true (see Ciesielski, Kerkyacharian, and Roynette \cite{CKR} and Roynette \cite{R}) that Brownian paths belong to $B^\alpha_{p,q}$ for any $\alpha < 1/2$, or if $\alpha = 1/2$ and $q = \oo$, and fail to belong to Besov spaces of any other parameters. Theorem \ref{T:introBesovresult} thus allows for an extension of the equation to Brownian paths as long as $H$ belongs to the Besov space $B^\beta_{\oo,1}$ for $\beta > 1$, with the approximating paths being, say, a standard mollification or a piecewise linear interpolant of the sample path. 

The next result explains that Brownian paths have properties that allow this to be pushed further, that is, we may take $H \in B^1_{\oo,1}$ for particular approximating families.

\begin{theorem}\label{T:introBrownian}
	Let $W: [0,T] \times \Omega \to \RR$ be a standard Brownian motion defined on a probability space $(\Omega,\mbf F, \mbf P)$, and assume that, for all $R > 0$, $H \in B^1_{\oo,1}(B_R)$. Then there exists a unique random variable $u : \Omega \to UC(\RR^d \times [0,T])$ such that, if $(\delta_n)_{n \in \NN}$ and $(W_n)_{n\in \NN}$ satisfy $\delta_n \xrightarrow{ n \to \oo} 0$ and either
	\[
		\left\{
		\begin{split}
		&W_n \text{ is piecewise affine on a partition $\mcl P_n = \{0 = t^n_0 < t^n_1 < \cdots < t^n_{N_n} = T\}$ of $[0,T]$ such that}\\
		&|W_n(t_i) - W_n(t_{i-1})| = \delta_n \text{ for all } n \in \NN \text{ and } i = 1,2,\ldots, N_n,
		\end{split}
		\right.
	\]
	or
	\[
		W_n(t) = \delta_n \zeta\pars{ \frac{t}{\delta_n^2}} \text{ for a linearly-interpolated simple random walk $\zeta$},
	\]
	then, as $n \to \oo$, the solution $u_n$ of \eqref{E:introapprox} converges uniformly in $\RR^d \times [0,T]$ to $u$, almost surely in the first case and in distribution in the second case.
\end{theorem}

We next present some further refinements of the above results. 

When $d = 1$, then $\dc(\RR)$ can be exactly characterized as the space of functions with first derivative of bounded variation. This is used to prove the next theorem:

\begin{theorem}\label{T:intro1d}
	If $\alpha \in (0,1)$ and $d = 1$, then the conclusions of Theorem \ref{T:introBesovresult} remain true if, for some $r > \frac{1}{1-\alpha}$ and for all $R > 0$, $H \in B^{2(1-\alpha)}_{r,1}(B_R)$. If $d = 1$, then the conclusions of Theorem \ref{T:introBrownian} remain true if, for some $r > 2$ and for all $R > 0$, $H \in B^1_{r,1}(B_R)$.
\end{theorem}

Theorem \ref{T:intro1d} implies that, in one dimension and for the path spaces specified in Theorem \ref{T:introBesovresult}, it is possible to take $H$ belonging to the Sobolev-Slobodeckij space $W^{\beta,r}$ for $\beta > 2(1-\alpha)$ and $r > \frac{1}{1-\alpha}$. In fact, the Hamiltonian $H$ may even belong to the Besov-Lorentz space $B^{2(1-\alpha)}_1(L^{\frac{1}{1-\alpha},1})_\loc$ (see Proposition \ref{P:1dbesov}, and see subsection \ref{SS:besov} for definitions). 

Note that Theorem \ref{T:introBrownian}, and Theorem \ref{T:introBesovresult} when $\alpha = 1/2$, give the criterion $H \in B^1_{\oo;1}$, which is strictly contained in $C^1$. However, if $d = 1$, then, by Theorem \ref{T:intro1d}, $H$ may belong to $H \in B^1_{r,1;\loc}$ for $r > 2$, or even $B^1_1(L^{2,1})_\loc$, and such functions are, in general, not even Lipschitz continuous.

The final result gives further examples of Hamiltonians for which Theorem \ref{T:introBesovresult} still holds. These are obtained by taking advantage of properties of $\dc$-functions having to do with structure rather than regularity.

\begin{theorem}\label{T:introstructure}
	Let $\alpha \in (0,1)$. Then the conclusions of Theorem \ref{T:introBesovresult}, and Theorem \ref{T:introBrownian} with $\alpha = 1/2$, hold if
	\[
		H \text{ is radial and, for some } r > \frac{d}{1-\alpha} \text{ and all $R > 0$, $H \in B^{2(1-\alpha)}_{r,1}(B_R)$},
	\]
	or
	\[
		\left\{
		\begin{split}
		&\text{for some $(d-1)$-dimensional hyperplane $\Gamma$, there exist } H_1,H_2 \in B^{2(1-\alpha)}_{\oo,1;\loc} \text{ such that}\\[1.2mm]
		&\text{$H = H_1$ or $H = H_2$ on either side of $\Gamma$,}
		\end{split}
		\right.
	\]
	or
	\[
		H(p) = a\pars{ \frac{p}{|p|}}|p| \quad \text{for some } a \in B^{2(1-\alpha)}_{\oo,1}(S^{d-1}).
	\]
\end{theorem}

The third example in Theorem \ref{T:introstructure} is important in the theory of front propagation.  Indeed, when $H$ takes such a form, then the level sets of the solution $u$ of \eqref{E:main} evolve according to the normal velocity $a(n)dW$.

\subsection{The main result: interpolation spaces}

The previous theorems follow from the main results of the paper, which are described next. 

For $\alpha \in [0,1]$ and $p \in [1,\oo]$, we define
\[
	\mathscr H_{\alpha,p} := (\dc(\RR^d), C(\RR^d) \cap L^\oo(\RR^d))_{\alpha,p;\loc} \quad \text{and} \quad \mathscr P_{\alpha,p} := (C_0([0,T],\RR^m),W_0^{1,1}([0,T],\RR^m) )_{\alpha,p}.
\]
Here, for two compatible normed spaces $X$ and $Y$, that is, both $X$ and $Y$ belong to a common Hausdorff topological space, $(X,Y)_{\alpha,p}$ denotes the real interpolation space of Lions and Peetre \cite{LP} of parameters $\alpha \in [0,1]$ and $p \in [1,\oo]$; see Bergh and L\"ofstr\"om \cite{BL} for more details. The notation $C_0$ and $W^{1,1}_0$ indicates the appropriate space of paths that satisfy $W(0) = 0$.

To formulate the results, it is convenient to introduce the solution map
\begin{equation}\label{introsolutionmap}
	(H,W) \mapsto u := S(H,W) \in UC(\RR^d \times [0,T]),
\end{equation}
where, for a fixed initial datum $u_0 \in UC(\RR^d)$, $u$ is the solution of \eqref{E:main}. The classical and pathwise viscosity solution theories then give that $S$ is a well-defined and continuous map on, respectively, $C(\RR^d) \times W_0^{1,1}([0,T],\RR^m)$ and $\dc_\loc(\RR^d) \times C_0([0,T],\RR^m)$.

The main results of the paper are that, for $\alpha \in (0,1)$ and $p \in [1,\oo]$, the meaning of $S$ can be extended, in an appropriate way, to a well-defined map on $\mathscr H_{\alpha,p} \times \mathscr P_{\alpha,p'}$, where $p'$ denotes the conjugate exponent of $p \in [1,\oo]$, that is, $\frac{1}{p} + \frac{1}{p'} = 1$. The first theorem deals with the case where $1 < p < \oo$, which, in particular, implies that $\dc_\loc(\RR^d)$ is dense in $\mathscr H_{\alpha,p}$ and $W^{1,1}_0([0,T],\RR^m)$ is dense in $\mathscr P_{\alpha,p'}$.

\begin{theorem}\label{T:intromain}
	Let $\alpha \in (0,1)$ and $p \in (1,\oo)$. Then the map \eqref{introsolutionmap} extends continuously to $\mathscr H_{\alpha,p} \times \mathscr P_{\alpha,p'}$. 
\end{theorem}

If $X_0$ and $X_1$ are two normed spaces such that $X_0$ embeds continuously into $X_1$, then, for $\alpha \in (0,1)$, $X_0$ is not dense in $(X_0,X_1)_{\alpha,\oo}$ or $(X_1,X_0)_{\alpha,\oo}$. In the present context, $X_0 = \dc_\loc$ or $X_0 = W^{1,1}_0$ are spaces of sufficiently ``smooth'' Hamiltonians or paths that are not dense in, respectively, $\mathscr H_{\alpha,\oo}$ or $\mathscr P_{\alpha,\oo}$ for any $\alpha \in (0,1)$. Thus, in order to make sense of $S(H,W)$ for arbitrary $H$ or $W$ belonging to either $\mathscr H_{\alpha,\oo}$ or $\mathscr P_{\alpha,\oo}$, it is necessary to allow for approximating sequences that do not converge in the full topology of these spaces. This is achieved with the next result.

\begin{theorem}\label{T:intromainweak}
Let $\alpha \in (0,1)$, $p \in [1,\oo]$, and $(H,W) \in \mathscr H_{\alpha,p} \times \mathscr P_{\alpha,p'}$. Then there exists a unique $S(H,W) \in UC(\RR^d \times [0,T])$ such that:
\begin{enumerate}[(a)] 
\item If $p < \oo$, $(W_n)_{n=1}^\oo \subset \mathscr P_{\alpha,p'}$, and
\begin{equation}\label{boundedpaths}
	\lim_{n \to \oo} \nor{W_n - W}{\oo,[0,T]}  = 0 \quad \text{and} \quad \sup_{n \in \NN} \nor{W_n}{\mathscr P_{\alpha,p'} } < \oo,
\end{equation}
then
\[
	\lim_{n \to \oo} \nor{ S(H, W_n) - S(H,W)}{\oo,\RR^d \times [0,T]} = 0.
\]
\item If $p' < \oo$, $(H^n)_{n=1}^\oo \subset \mathscr H_{\alpha,p}$, and
\begin{equation}\label{boundedHs}
	\lim_{n \to \oo} H_n = H \text{ locally uniformly and } \sup_{n \in \NN} \nor{H_n}{\mathscr H_{\alpha,p} } < \oo,
\end{equation}
then
\[
	\lim_{n \to \oo} \nor{S(H_n, W) - S(H,W)}{\oo,\RR^d \times [0,T]} = 0.
\]
\end{enumerate}
\end{theorem}

The next result demonstrates that, in general, the assumptions of Theorem \ref{T:intromainweak} cannot be relaxed. 
	
\begin{theorem}\label{T:sharp}
	For $\beta \in (0,1)$ and $p \in \RR^d$, define $H_\beta(p) := |p|^\beta$. Then $H_\beta \in \mathscr H_{\alpha,1}$ if and only if $\alpha + \beta > 1$. Moreover, if $\alpha \in (0,1)$ and $u_0(x) = |x|$ for $x \in \RR^d$, then the following hold:
	\begin{enumerate}[(a)]
	\item If $\alpha + \beta < 1$, then there exists a sequence of paths $(W_n)_{n \in \NN} \subset W^{1,1}_0([0,T],\RR)$ such that
	\[
		\lim_{n \to \oo} W_n = 0 \text{ uniformly,} \quad \sup_{n \in \NN} \nor{W_n}{\mathscr P_{\alpha,\oo}} < \oo, \quad \text{and} \quad \lim_{n \to \oo} S(H_\beta,W_n) = +\oo.
	\]
	\item For any $c_0 > 0$, there exists a sequence of paths $(W_n)_{n \in \NN} \subset W^{1,1}([0,T])$ such that
	\[
		\lim_{n \to \oo} W_n = 0 \text{ uniformly,} \quad \sup_{n \in \NN} \nor{W_n}{\mathscr P_{\alpha,\oo}} < \oo,
	\]
	and
	\[
		\lim_{n \to \oo} \sup_{(x,t) \in \RR^d \times [0,T]} |S(H_{1-\alpha},W_n)(x,t) - (|x| \vee c_0 t^\alpha)| = 0.
	\]
	\end{enumerate}
\end{theorem}

The classical viscosity solution theory says that, for any $u_0 \in UC(\RR^d)$, we have $S(H,0) = u_0$, that is, the solution $u$ of \eqref{E:main} with $W \equiv 0$ must satisfy $u(x,t) = u_0(x)$ for $(x,t) \in \RR^d \times [0,T]$. Theorem \ref{T:sharp}(b) then implies that the map $W \mapsto S(H_{1-\alpha},W)$ cannot be extended uniquely to $\mathscr P_{\alpha,\oo}$ by taking approximating sequences that are bounded in $\mathscr P_{\alpha,\oo}$.

\subsection{Some open questions}
Although Theorems \ref{T:intromain} and \ref{T:intromainweak} give an extensive description of the initial value problem \eqref{E:main}, a number of questions still remain, some of which we outline next.

Unless $d = 1$, there is no analytic characterization of the space $\dc(\RR^d)$. We do have, however, the continuous inclusions (see Propositions \ref{P:C11dc} and \ref{P:1ddc})
\[
	W^{2,\oo}(\RR^d) \subset \dc(\RR^d) \quad \text{and} \quad W^{2,1}(\RR) \subset \dc(\RR).
\]
This shows that some form of second-order regularity can be used as a criterion for belonging to $\dc$, and, hence, leads us to formulate the following question:

\begin{question}\label{Q:betterdc}
Does there exist $q = q_d  \in [1,\oo)$ such that $W^{2,q}(\RR^d) \subset \dc(\RR^d)$ for all $q > q_d$?
\end{question}

A partial result (see Proposition \ref{P:radialfdc}) is that
\begin{equation}\label{partialresult}
	\text{if } f \in W^{2,q}(\RR^d) \text{ with } q > d \text{ and $f$ is radial, then } f \in \dc(\RR^d).
\end{equation}
Indeed, this is what allows for a proof of the first statement of Theorem \ref{T:introstructure}. 

On the other hand, while finishing the preparation of this work, a counterexample was communicated to us by Terence Tao \cite{Texample}, which shows that the answer to Question \ref{Q:betterdc} is negative in general if $d > 1$; see Proposition \ref{P:Texample} for more details.

The next question, which is about certain interpolation spaces, is motivated by the definition of $\mathscr P_{\alpha,p}$, and, when $d = 1$, the definition of $\mathscr H_{\alpha,p}$.

\begin{question}\label{Q:interpolationspace}
What is the characterization of the interpolation space $(W^{m,1}(\RR), C(\RR))_{\alpha,p}$ for $\alpha \in (0,1)$, $p \in [1,\oo]$, and $m = 1,2,\ldots$? 
\end{question}

The space $(W^{1,1}(\RR), C(\RR))$ is related to $\mathscr P_{\alpha,p}$, which, as we prove in Section \ref{S:P}, contains examples of H\"older, variation, and Besov-Lorentz spaces. The case $m = 2$ is treated in Propositions \ref{P:1dbesov}, which shows that Besov-Lorentz regularity is a sufficient criterion for belonging to $(W^{2,1}(\RR),C(\RR))_{\alpha,p}$. Similar questions were studied by Kruglyak \cite{K} using Calder\'on-Zygmund decomposition methods, but for ranges of exponents that fall out of the scope of Question \ref{Q:interpolationspace}.

The last question is motivated by the example given in Theorem \ref{T:sharp}, and by analogous observations from the theory of rough differential equations. 

Let $\mcl T$ denote the map
\begin{equation}\label{roughpathmap}
	C^1([0,T],\RR^m) \ni Y \mapsto \mcl TY := X \in C([0,T],\RR^n),
\end{equation}
where, for some smooth function $f: \RR^n \to \RR^{n \times m}$ and $x \in \RR^n$, $X = \mcl TY$ is given by the solution of the initial value problem
\begin{equation}\label{E:RDE}
	\dot X(t) = f(X(t)) \dot Y(t) \quad \text{for } t \in [0,T] \quad \text{and} \quad X(0) = 0.
\end{equation}
When $\alpha > \frac{1}{2}$, it turns out that the solution operator $\mcl T$ for \eqref{E:RDE} extends continuously to $Y \in C^{0,\alpha}([0,T],\RR^m)$, which is the so-called Young regime. However, if $\alpha \le 1/2$, and if $Y \in C^{0,\alpha}([0,T],\RR^m)$ and $(Y_n)_{n \in \NN} \subset C^1([0,T],\RR^m)$ are such that
\begin{equation}\label{goodYseq}
	\lim_{n \to \oo} Y_n = Y \text{ uniformly and } \sup_{n \in \NN} \nor{Y_n}{C^{0,\alpha}} < \oo,
\end{equation}
then the sequence of solutions $(X_n)_{n \in \NN}$ of \eqref{E:RDE} corresponding to $(Y_n)_{n \in \NN}$ can fail to converge, or can have different limits for different approximating sequences. 

This lack of convergence and uniqueness is tied directly to certain iterated integrals of $Y$. For example, if $1/3 < \alpha \le 1/2$, $Y \in C^{0,\alpha}([0,T],\RR^m)$, $(Y_n)_{n \in \NN}$ and $(\tilde Y_n)_{n \in \NN}$ are two sequences of smooth paths satisfying \eqref{goodYseq}, and if, as $n \to \oo$, the antisymmetric matrix-valued paths $\mathbb Y_n$ and $\tilde {\mathbb Y_n}$ defined, for $i,j = 1,2,\ldots,m$ and $t \in [0,T]$, by
\[
	\mathbb Y^{ij}_n(t) := \int_0^t  Y^i_n(s)\dot Y^j_n(s)ds - \int_0^t  Y^j_n(s)\dot Y^i_n(s)ds \quad \text{and} \quad
	\tilde{\mathbb Y}^{ij}_n(t) := \int_0^t  \tilde Y^i_n(s)\dot {\tilde Y}^j_n(s)ds - \int_0^t  \tilde Y^j_n(s)\dot {\tilde Y}^i_n(s)ds
\]
converge uniformly to, respectively, $\mathbb Y$ and $\tilde{\mathbb Y}$, then, as $n \to \oo$, the corresponding solutions $X_n$ and $\tilde X_n$ of \eqref{E:RDE} have different uniform limits unless $\mathbb Y = \tilde{\mathbb Y}$.

The theory of rough paths put forward by Lyons \cite{Ly} (see also Friz and Hairer \cite{FH} and Friz and Victoir \cite{FV} for many more details and extensions) makes this connection systematic by introducing suitably augmented versions of the H\"older spaces that ``record'' the information about the iterated integrals. The modified solution map \eqref{roughpathmap} is then continuous with respect to the appropriate topology, and this allows for the equation \eqref{E:RDE} to have an analytic solution theory for paths with regularity below the Young regime.

In view of the analogous phenomena demonstrated by Theorem \ref{T:sharp}, we are led to the following question.

\begin{question}\label{Q:closable}
	Suppose that $\alpha,\beta \in (0,1)$ satisfy $\alpha + \beta < 1$. Can the topology on the space $\mathscr H_{\alpha,1} \times \mathscr P_{\beta,\oo}$ be ``augmented'' with certain quantities in such a way that the solution operator \eqref{introsolutionmap} can be extended with respect to the new topology?
\end{question}

The question of borrowing ideas and techniques from the theory of rough paths to make sense of singular partial differential equations is not new. In the seminal work of Hairer \cite{Hrfirst,Hrrs}, the theory of regularity structures is introduced in order to provide an analytic framework for several nonlinear, singular partial differential equations. A different but related approach is that of paracontrolled distributions, as in the work of Gubinelli, Imkeller, and Perkowski \cite{GIP}. A more direct analogy with rough paths is seen in the work of Otto and Weber \cite{OW}. 

\subsection{Organization of the paper}
Section \ref{S:spaces} contains an overview of several of the spaces that are referred to throughout the paper, especially those that are not often used in the theory of Hamilton-Jacobi equations, that is, Lorentz, Sobolev, Besov, and real interpolation spaces. In Section \ref{S:mainresult}, we prove the main results, that is, Theorems \ref{T:intromain} and \ref{T:intromainweak}. Section \ref{S:H} and Section \ref{S:P} then analyze several examples of spaces belonging to respectively $\mathscr H_{\alpha,p}$ and $\mathscr P_{\alpha,p'}$, thus giving proofs of the results described above in subsection \ref{SS:examples}. The sharpness result, Theorem \ref{T:sharp}, is discussed in Section \ref{S:sharp}. In Appendix \ref{S:LS}, we give an overview of the fact that the solution operator \eqref{introsolutionmap} extends to continuous paths and $\dc$-Hamiltonians. Finally, in Appendix \ref{S:dc}, we present many examples of functions that are representable as a difference of convex functions.

\subsection{Notation}\label{SS:notation} Given a space of paths $X([0,T],\RR^m)$,
\[
	X_0 = X_0([0,T],\RR^m) := \left\{ W \in X([0,T],\RR^m) : W(0) = 0 \right\}.
\]
This notation will be used with $C^{0,\alpha}$, $W^{s,p}$, $B^s_{pq}$, and others in place of $X$, in which case the usual norms can be replaced with equivalent semi-norms, which will be explained in the various contexts.

For a space of functions on $\RR^d$ denoted by $Y$, we say that $f \in Y_\loc(\RR^d)$ if, for every open and bounded set $U \subset \RR^d$, there exists $\tilde f \in Y(\RR^d)$ such that $f = \tilde f$ on $U$. In most situations, it will hold that, if $f \in Y_\loc(\RR^d)$ and $\eta \in C^\oo(\RR^d)$ has compact support, then $\eta \cdot f \in Y(\RR^d)$. We also write
\[
	Y_\rad(\RR^d) := \left\{ f : f \text{ is radial on } \RR^d \right\} \quad \text{and} \quad Y_\odd(\RR^d) := \left\{ f : f \text{ is odd on } \RR^d\right\}.
\]

We write $\nor{f}{\oo}$ for the supremum norm of a function $f$. At times, we also use the notation $\nor{f}{\oo,U} := \sup_{x \in U} |f(x)|$ to distinguish the domain. We denote by $(B)UC(U)$ the space of (bounded and) uniformly continuous functions on $U$. If $U$ is bounded, this space is equipped with the supremum-norm $\nor{\cdot}{\oo,U}$, and, otherwise, $f \in UC(U)$ is equivalent to
\[
	\nor{f}{UC(U)} := \sum_{n =1}^\oo \max \left\{ 2^{-n} , \nor{f}{\oo,B_n \cap U} \right\} < \oo.
\]

We denote by $\mathscr S = \mathscr S(\RR^d)$ the space of Schwartz functions on $\RR^d$, that is,
\[
	\mathscr S(\RR^d) := \left\{ f \in C^\oo(\RR^d) : \sup_{x \in \RR^d}  |x|^m |D^n f(x)| < \oo \right\},
\]
and $\mathscr S'(\RR^d) = \mathscr S'$ its dual, the space of tempered distributions. 

If $(X,\mu)$ is a measure space and $p \in [1,\oo]$, then
\[
	L^p(X,\mu) := \left\{ f:X \to \RR, \; \nor{f}{L^p(X,\mu)} := \left[ \int_X f(x)^p \mu(dx) \right]^{1/p} < \oo \right\}.
\]
At times, we suppress the dependence on $X$ or $\mu$ when this does not cause confusion. For $p \in [1,\oo]$,
\[
	p' := \frac{p}{p-1} \in [1,\oo].
\]

For an open set $U \subset \RR^d$ and $f: U \to \RR$, $\supp f$ is the closure of the set on which $f = 0$. If $k = 0,1,2,\ldots$, $C^k_c(U)$ is the set of functions with up to $k$ continuous derivatives and such that $\supp f$ is compact. $BV(U)$ denotes the space of functions of bounded variation on $U$.

For a probability space $(\Omega,\mbf F, \mbf P)$ and a $\mbf F$-measurable random variable $X: \Omega \to \RR$, we write
\[
	\mbf E[X] := \int_{\Omega} X(\omega) \mbf P(d\omega) \quad \text{and} \quad \mbf {Var}[X] := \mbf E| X - \mbf E X|^2.
\]

When $f: \RR^d \to\RR$, we write the Legendre transform of $f$ as $f^*$, that is,
\[
	f^*(p) := \sup_{x \in \RR^d} \left\{ p \cdot x - f(x) \right\} \quad \text{for } p \in \RR^d.
\]

If $x \in \RR^d$ and $R > 0$, $B_R(x) := \{ y \in \RR^d: |y-x| < R \}$ and $B_R := B_R(0)$. For $d = 1,2,3,\ldots$, $S^{d-1}$ denotes the $(d-1)$-dimensional unit sphere in $\RR^d$. For $\alpha \in \RR$, $\lfloor \alpha \rfloor \in \ZZ$ and $\lceil \alpha \rceil \in \ZZ$ denote respectively the floor and ceiling of $\alpha$, and, for $\beta \in \RR$,
\[
	\alpha \vee \beta := \max(\alpha,\beta), \quad \alpha \wedge \beta = \min(\alpha,\beta), \quad \alpha_+ := \alpha \vee 0, \quad \text{and} \quad \alpha_- := - (\alpha \wedge 0).
\]

\section{Function spaces}\label{S:spaces}

This section contains a brief overview of various function spaces used throughout the paper. Many more details can be found in the appropriate references, listed below. 

\subsection{The space $\dc$}

Functions that are representable as a difference of convex functions made an appearance in the context of pathwise viscosity solutions in \cite{LS2}. In order to use the space of such functions in the interpolation theory, it is necessary to equip it with an appropriate norm.

\begin{definition}\label{D:dc}
Let $U \subset \RR^d$ be an open domain and let $f: U \to \RR$. Then $f \in \dc(U)$ if there exist convex functions $f_1$ and $f_2$ on $U$ such that
\[
	f = f_1 - f_2.
\]
If $U$ is bounded, $\dc(U)$ is equipped with the norm
\[
	\nor{f}{\dc(U)} := \inf\left\{ \nor{f_1}{\oo,U} + \nor{f_2}{\oo,U} : f = f_1 - f_2, \; f_1,f_2 \text{ convex}  \right\}.
\]
A function $f$ is said to belong to $\dc_\loc(U)$ if $f \in \dc(V)$ for all bounded $V \subset U$, or equivalently,
\[
	\nor{f}{\dc_\loc} := \sum_{n=1}^\oo \max(2^{-n}, \nor{f}{\dc(B_n)} ) < \oo.
\]
When $U = \RR^d$, we write $\dc = \dc(\RR^d)$ and $\dc_\loc = \dc_\loc(\RR^d)$.
\end{definition}

We note that the quantity $\nor{\cdot}{\dc_\loc}$ is not a true norm, but the Fr\'echet space $\dc_\loc$ is a complete metric space with the metric $(f,g) \mapsto \nor{f - g}{\dc_\loc}$.

In terms of regularity, $\dc$-functions share essentially the same properties as convex functions; namely, they are locally Lipschitz, but not $C^1$ in general, and they are almost everywhere twice-differentiable. Examples include $C^{1,1}$-functions and, when $d = 1$, functions whose first derivative belongs to $BV$. For more details and examples, see Appendix \ref{S:dc}.

\subsection{Lebesgue and Lorentz spaces}\label{SS:LL}

Let $1 < p < \oo$. Recall that, given a measure space $(X,\mu)$ and $f \in L^p(X,\mu)$, we have
\[
	\nor{f}{L^p(X,\mu)}^p = p \int_0^\oo \sigma^{p-1} \mu(\{ x \in X : |f(x)| > \sigma\}|)d\sigma.
\]
Denote by $\oline{f}:[0,\oo) \to [0,\oo)$ the nonincreasing rearrangement of $|f|$, that is,
\[
	\oline{f}(t) := \inf\{ \sigma : \mu(\{ x \in X: |f(x)| > \sigma\}) \le t \},
\]
which satisfies
\[
	\int_0^\oo |\oline{f}(t)|^p dt = \nor{f}{L^p(X,\mu)}^p.
\]
Now let $q \in [1,\oo]$. The Lorentz space $L^{p,q}(X,\mu)$ is defined as those functions with finite $\nor{\cdot}{L^{p,q}(X,\mu)}$-norm, where
\[
	\nor{f}{L^{p,q}(X,\mu)} := \left[ \int_0^\oo \pars{ t^{1/p} \oline{f}(t)}^q \frac{dt}{t} \right]^{1/q},
\]
with the corresponding analogue when $q = \oo$. Observe that $L^{p,p}(X,\mu) = L^p(X,\mu)$ for all $1 < p < \oo$.

The $L^{p,q}(X,\mu)$-norm is equivalent to the quantity, 
\[
	N_{p,q}(f) := \left[\sum_{k\in \ZZ} (2^k \mu(\{ x \in X: |f(x)| > 2^k\})^{1/p})^q\right]^{1/q},
\]
which is not itself a norm, since it fails to satisfy the triangle inequality property. 

If $\mu(X) < \oo$, then $N_{p,q}$ is equivalent to
\[
	\tilde N_{p,q}(f) := \left[\sum_{k=1}^\oo (2^k \mu(\{ x \in X: |f(x)| > 2^k\})^{1/p})^q\right]^{1/q},
\]
and therefore, in this case, for all $r > p$, $L^r(X,\mu) \subset L^{p,q}(X,\mu)$.

Additional difficulties arise when $p = 1$ or $p = \oo$. For instance, $L^{1,q}$ is not a Banach space for any $q>1$. In this paper, we only consider the cases $1 < p < \oo$ or $p = q$. For more details on Lorentz spaces, see \cite{BL,Hu}.

\subsection{Spaces of vector-value sequences}\label{SS:sequence}

Given $p \in [1,\oo]$ and a normed spaced $X$, $\ell^p(X)$ is the space of sequences $(a_n)_{n \in \NN} \subset X$ such that
\[
	\nor{(a_n)}{\ell^p(X)} := \pars{ \sum_{n=1}^\oo \nor{a_n}{X}^p }^{1/p} < \oo,
\]
and, for $s \in \RR$, $\ell^{s,p}(X)$ is the space of sequences $(a_n)_{n \in \NN}$ that satisfy
\[
	\nor{(a_n)}{\ell^{s,p}(X)} := \left[ \sum_{n=1}^\oo \pars{ 2^{ns} \nor{a_n}{X} }^p \right]^{1/p} < \oo,
\]
with the correct analogue when $p = \oo$ in both cases.

\subsection{Spaces of Besov type}\label{SS:besov} We now list various equivalent definitions of Besov spaces that are used throughout the paper. For more details, characterizations, and properties, see \cite{BL,T}.

We first give the standard definition in terms of Fourier analysis. We denote by $\mathscr F : \mathscr S \to \mathscr S$ the Fourier transform
\[
	\mathscr F f(\xi) = \hat{f}(\xi) := \int_{\RR^d} e^{-ix \cdot \xi} f(x)\;dx \quad \text{for }\xi \in \RR^d,
\]
and its inverse $\mathscr F^{-1}$ is
\[
	\mathscr F^{-1} f(x) = \check{f}(x) := \frac{1}{(2\pi)^d} \int_{\RR^d} e^{ix \cdot \xi} f(\xi)\;d\xi \quad \text{for } x \in \RR^d.
\]
Both $\mathscr F$ and $\mathscr F^{-1}$ extend by duality to $\mathscr S'$.

It is a standard fact \cite{BL,T} that there exists a function $\phi \in C^\oo_0(\RR^d)$ such that
\begin{equation}\label{pou}
	\phi \ge 0, \quad \supp \phi = \left\{ \xi : \frac{1}{2} \le |\xi| \le 2 \right\}, \quad \text{and} \quad \sum_{k=-\oo}^{\oo} \phi(2^{-k} \xi) = 1.
\end{equation}
We then define $(\phi_k)_{k \in \ZZ} \subset \mathscr S$ and $\psi \in \mathscr S$ by
\[
	\mathscr F \psi(\xi) = 1 - \sum_{k=1}^\oo \phi(2^{-k}\xi) \quad \text{and} \quad \mathscr F \phi_k(\xi) = \phi(2^{-k}\xi) \quad \text{for } \xi \in \RR^d.
\]

For $s \in \RR$ and $1 \le p,q \le \oo$, the Besov space $B^s_{pq}(\RR^d)$ is given by
\[
	B^s_{pq}(\RR^d) = B^s_{pq} := \left\{ f \in \mathscr S' : \nor{f}{B^s_{pq}} < \oo \right\},
\]
where
\begin{equation}\label{besovnorm}
	\nor{f}{B^s_{pq}} := \nor{\psi * f}{L^p(\RR^d)} + \left[ \sum_{k=1}^\oo (2^{sk} \nor{\phi_k *f}{L^p(\RR^d)} )^q \right]^{1/q}.
\end{equation}
The linear maps defined, for $f \in \mathscr S'$, by
\begin{equation}\label{LP}
	L_0 f := \psi * f \quad \text{and} \quad L_k f := \phi_k * f \quad \text{for } k = 1,2,\ldots
\end{equation}
are known as the Littlewood-Paley projections, and it is clear that $f \in B^s_{pq}$ if and only if
\[
	\pars{ L_k f}_{k=0}^\oo \subset \ell^{s,q}(L^p(\RR^d)).
\]
At times, is convenient to use the notation
\[
	B^s_{pq} = B^s_q(L^p)
\]
to emphasize the role of the underlying $L^p$-metric. These spaces can also be generalized to allow for choices other than $L^p$. As a particular example, we consider spaces of Besov-Lorentz type (see \cite{ST} for more details), that is, the spaces defined, for $s \in \RR$ and $1 \le p,q,r \le \oo$, by
\[
	B^s_q(L^{p,r}) := \left\{ f \in \mathscr S' : (L_k f)_{k = 0}^\oo \in \ell^{s,q}(L^{p,r}(\RR^d)) \right\}
\]
with the norm
\[
	\nor{f}{B^s_q(L^{p,r})} := \nor{\pars{ L_k f}_{k=0}^\oo }{\ell^{s,q}(L^{p,r}(\RR^d))}.
\]

In some cases, the Besov norm of a function can be equivalently defined in terms of its modulus of continuity. For $h \in \RR^d$ and $f: \RR^d \to \RR$, define
\[
	\Delta^2_h f(x) := f(x+h) + f(x-h) - 2f(x) \quad \text{for } x \in \RR^d.
\]
Then, for $0 < s < 2$ and $1 \le p,q \le \oo$, we define a norm equivalent to \eqref{besovnorm} by
\begin{equation}\label{equivalentbesov}
	f \mapsto \nor{f}{L^p(\RR^d)} + \left[ \int_0^\oo \pars{\frac{ \sup_{|h| \le t} \nor{\Delta^2_h f}{L^p(\RR^d)} }{t^s} }^q \frac{dt}{t} \right]^{1/q},
\end{equation}
with the $L^q$-norm replaced with the correct analogue when $q = \oo$. This definition can be extended to the case $s \ge 2$ in various ways, but we do not pursue this here, since the range $0 < s < 2$ is the one relevant to this paper.

If $f \in B^s_{pq,\loc}$, then, for any $\eta \in C_c^\oo(\RR^d)$, $\tilde f := f \cdot \eta \in B^s_{pq}$, and for any open set $U$ containing support of $\eta$, the norm \eqref{equivalentbesov} for $\tilde f$ is also equivalent to
\begin{equation}\label{equivalentbesov2}
	\tilde f \mapsto \nor{\tilde f}{L^p(U)} + \left[ \int_0^1 \pars{\frac{ \sup_{|h| \le t} \nor{\Delta^2_h \tilde f}{L^p(U)} }{t^s} }^q \frac{dt}{t} \right]^{1/q}.
\end{equation}

Both \eqref{equivalentbesov} and \eqref{equivalentbesov2} have extensions to the Besov-Lorentz spaces; that is, the norm $\nor{\cdot}{B^s_{q}(L^{p,r})}$ can be replaced with the correct analogues of \eqref{equivalentbesov} or \eqref{equivalentbesov2} with the $L^p$-norm replaced with the $L^{p,r}$-norm. In particular, for all $\tilde p > p$, we have the continuous embedding
\[
	B^s_{\tilde p, q;\loc} \subset B^s_q(L^{p,r})_\loc.
\]

When $s > 0$ is not an integer and $1 \le p = q < \oo$, then $B^s_{pp}$ is equal to the Sobolev-Slobodeckij space $W^{s,p} = W^{s,p}(\RR^d)$, where, if $s = n + \alpha$ with $n = 0,1,2,\ldots$ and $\alpha \in (0,1)$, equipped with the norm
\[
	\nor{f}{W^{s,p}} := \sum_{k=1}^n \nor{D^k f}{L^p} + \left[ \iint_{\RR^d \times \RR^d} \frac{|f(x) - f(y)|^p}{|x-y|^{\alpha p + d}} dxdy \right]^{1/p}.
\]
When $p = q = \oo$ and $s > 0$, $B^s_{\oo \oo}$ is also called the Zygmund space $\mathscr C^s(\RR^d)$. For $s$ not equal to an integer, $\mathscr C^s(\RR^d)$ is equal to the space $C^{n,\alpha}(\RR^d)$, where $s = n + \alpha$ with $n =0,1,2,\ldots$ and $\alpha \in (0,1)$. In general, for a nonnegative integer $n$, we have the strict inclusions
\[
	C^n(\RR^d) \subset C^{n-1,1}(\RR^d) \subset \mathscr C^n(\RR^d).
\]

The Besov spaces used throughout this paper have indices in the range $s > d/p$, and therefore are subspaces of continuous functions, in view of the continuous embeddings $B^s_{pq} \subset \mathscr C^{s-d/p} \subset C$.

\subsection{Variation spaces}\label{SS:variation}
For $p \in [1,\oo)$, define
\[
	V_{p,0} := \left\{ W \in C_0 : \nor{W}{V_p} := \sup_{ \mcl P} \pars{ \sum_{i=1}^n |W(t_i) - W(t_{i-1})|^p}^{1/p} < \oo \right\},
\]
where the supremum is taken over all partitions $\mcl P := \left\{ 0 = t_0 < t_1 < \cdots < t_N = T \right\}$ of $[0,T]$. The space $V_{p,0}$ becomes a Banach space under the norm $\nor{\cdot}{V_p}$. Observe that $V_{1,0}$ is the space of continuous paths $W$ of bounded variation that satisfy $W(0) = 0$.

We have the (strict) continuous embedding
\[
	C_0^{0,1/p} \subset V_{p,0} \quad \text{for all } p \in [1,\oo).
\]

\subsection{Real interpolation spaces}\label{SS:interp}

We now list several facts about real interpolation spaces, the proofs of which, along with many more details, can be found in \cite{BL}.

Let $X_0$ and $X_1$ be compatible normed spaces, that is, both $X_0$ and $X_1$ are subspaces of some Hausdorff topological space $Y$. For $x \in X_0 + X_1$ and $t \ge 0$, define the $K$-functional
\[
	K(t,x,X_0,X_1) = \inf\left\{ \nor{x_0}{X_0} + t\nor{x_1}{X_1} : x = x_0 + x_1, \; x_0 \in X_0, \; x_1 \in X_1 \right\}.
\]

\begin{lemma}\label{L:K}
	\begin{enumerate}[(a)]
	\item\label{L:Knorm} For any $t > 0$, the map
	\[
		x \mapsto K(t,x,X_0,X_1)
	\]
	defines a norm on $X_0 + X_1$.
	\item\label{L:Kscales} For any $t > 0$ and $x \in X_0 + X_1$,
	\[
		K(t,x,X_0,X_1) = t K \pars{ \frac{1}{t}, x, X_1,X_0}.
	\]
	\item\label{L:Kmonotone} For any $x \in X_0 + X_1$,
	\[
		t \mapsto K(t,x,X_0,X_1)
	\]
	is convex and nondecreasing, and, for all $s,t > 0$,
	\[
		K(s+t,x,X_0,X_1) \le 2 \pars{ K(s,x,X_0,X_1) + K(t,x,X_0,X_1)}.
	\]
	\item\label{L:Kmonotone2} For any $x \in X_0 + X_1$, the map
		\[
			[0,\oo)^2 \ni (s,t) \mapsto t K\pars{ \frac{s}{t}, x, X_0,X_1} = s K \pars{ \frac{t}{s}, x, X_1,X_0}
		\]
		is nondecreasing in both $s$ and $t$.
	\end{enumerate}
\end{lemma}

For $\alpha \in [0,1]$ and $q \in [1,\oo]$, we define the norm
\[
	\nor{x}{(X_0,X_1)_{\alpha,q}} :=
	\begin{dcases}
		\left[ \int_0^\oo \pars{ \frac{K(t,x,X_0,X_1)}{t^\alpha}}^q \frac{dt}{t} \right]^{1/q} & \text{if } 1 \le q < \oo, \text{ and}\\
		\sup_{t \in (0,\oo)} \frac{K(t,x,X_0,X_1)}{t^\alpha} & \text{if } q = \oo.
	\end{dcases}
\]
The real interpolation space between $X_0$ and $X_1$ of parameters $\alpha \in [0,1]$ and $q \in [1,\oo]$ is given by
\[
	(X_0,X_1)_{\alpha,q} := \left\{ x \in X_0 + X_1 : \nor{x}{(X_0,X_1)_{\alpha,q}}< \oo \right\}.
\]

\begin{lemma}\label{L:interp}
Let $X_0$ and $X_1$ be two compatible normed spaces, and fix $\alpha \in (0,1)$ and $p \in [1,\oo]$.

\begin{enumerate}[(a)]
\item\label{L:interpreflexive} The equality $(X_0,X_1)_{\alpha,p} = (X_1,X_0)_{1-\alpha,p}$ holds.
\item\label{L:interpsum} The norm $\nor{\cdot}{(X_0,X_1)_{\alpha,p}}$ is equivalent, with the proportionality constants depending only on $\alpha$ and $p$, to
	\[
		x \mapsto \pars{ \sum_{n \in \ZZ} K(2^n, x, X_0,X_1)^q 2^{-nq\alpha} }^{1/q}.
	\]
	If $X_1$ embeds continuously into $X_0$, the following norm is also equivalent:
	\[
		x \mapsto \pars{ \sum_{n =0}^\oo K(2^{-n}, \cdot, X_0,X_1)^q 2^{nq\alpha} }^{1/q}
	\]
	and
	\[
		x \mapsto \left[ \int_0^1 \pars{ \frac{K(t,x,X_0,X_1)}{t^\alpha}}^q \frac{dt}{t} \right]^{1/q}.
	\]
	
\item\label{L:interpmonotone} If $1\le p_1 \le p_2 \le \oo$ and $\alpha \in (0,1)$, then
	\[
		(X_0,X_1)_{\alpha,p_1} \subset (X_0,X_1)_{\alpha,p_2} \quad \text{continuously.}
	\]
	If $X_1$ embeds continuously into $X_0$, $p_1,p_2 \in [1,\oo]$, and $\alpha_1 < \alpha_2$, then
	\[
		(X_0,X_1)_{\alpha_2,p_1} \subset (X_0,X_1)_{\alpha_1,p_1} \quad \text{continuously.}
	\]
\item\label{L:interpdense} Assume that $X_1 \subset X_0$ continuously. Then, for all $\alpha \in (0,1)$ and $1 \le p < \oo$, $X_1$ is dense in $(X_0,X_1)_{\alpha,p}$ in the topology of $(X_0,X_1)_{\alpha,p}$. The closure of $X_1$ in the topology of $(X_0,X_1)_{\alpha,\oo}$ are those $x \in X_0 + X_1$ for which
	\[
		\lim_{n \to \oo} K(2^{-n}, x, X_0,X_1) 2^{n\alpha} = 0.
	\]
\end{enumerate}
\end{lemma}

We will also need the following stability property.

\begin{lemma}\label{L:oostability}
	Let $\alpha \in (0,1)$ and assume that $X_1 \subset X_0$ continuously. Suppose that $(x_n)_{n \in \NN} \subset (X_0,X_1)_{\alpha,\oo}$, and, for some $R > 0$,
	\[
		\sup_{n \in \NN} \nor{x_n}{(X_0,X_1)_{\alpha,\oo}} \le R \quad \text{and} \quad \lim_{n \to \oo} \nor{x_n - x}{X_0} = 0.
	\]
	Then $x \in (X_0,X_1)_{\alpha,\oo}$ and $\nor{x}{(X_0,X_1)_{\alpha,\oo}} \le R$.
\end{lemma}

\begin{proof}
	Fix $\eps > 0$. Then, for every $t > 0$ and $n \in \NN$, there exists $y_n(t) \in X_1$ such that
	\[
		\nor{x_n - y_n(t)}{X_0} + t \nor{y_n(t)}{X_1} \le (R + \eps)t^\alpha,
	\]
	and therefore
	\[
		\nor{x - y_n(t)}{X_0} + t \nor{y_n(t)}{X_1} \le (R + \eps)t^\alpha + \nor{x - x_n}{X_0}.
	\]
	Choose $n(t) \in \NN$ such that
	\[
		\nor{x - x_{n(t)}}{X_0} \le \eps t^\alpha.
	\]
	Then
	\[
		\nor{x - y_{n(t)}(t)}{X_0} + t \nor{y_{n(t)}(t)}{X_1} \le (R + 2\eps)t^\alpha,
	\]
	whence $x \in (X_0,X_1)_{\alpha,\oo}$ and $\nor{x}{(X_0,X_1)_{\alpha,\oo}} \le R + 2\eps$. The result follows from the fact that $\eps$ was arbitrary.
\end{proof}

\section{The main results}\label{S:mainresult}

For $\alpha \in [0,1]$ and $p \in [1,\oo]$, we set
\[
	\mathscr H_{\alpha,p} := (\dc(\RR^d), C(\RR^d) \cap L^\oo(\RR^d))_{\alpha,p;\loc} \quad \text{and} \quad \mathscr P_{\alpha,p} := (C_0([0,T],\RR^m), W_0^{1,1}([0,T],\RR^m))_{\alpha,p},
\]
where the former is more precisely defined as
\[
	H \in \mathscr H_{\alpha,p} \quad \Leftrightarrow \quad H \in (\dc(B_L), C(B_L))_{\alpha,p} \text{ for every } L > 0,
\]
and is a metric space with the metric
\[
	(H,\tilde H) \mapsto \sum_{n=1}^\oo 2^{-n} \wedge \nor{H - \tilde H}{(\dc(B_n), C(B_n))_{\alpha,p}}.
\]

In view of Lemma \ref{L:interp}\eqref{L:interpreflexive} and \eqref{L:interpmonotone}, we have the continuous inclusions
\[
	\begin{dcases}
		\mathscr H_{\alpha,p_2} \subset \mathscr H_{\alpha,p_1} \text{ and } \mathscr P_{\alpha, p_1} \subset \mathscr P_{\alpha,p_2} & \text{if } \alpha \in (0,1) \text{ and } 1 \le p_1 \le p_2 \le \oo, \text{ and}\\
		\mathscr H_{\alpha_1,p_1} \subset \mathscr H_{\alpha_2,p_2} \text{ and } \mathscr P_{\alpha_2,p_1} \subset \mathscr P_{\alpha_1,p_2} & \text{if } 0 < \alpha_1 < \alpha_2 < 1 \text{ and } p_1,p_2 \in [1,\oo].
	\end{dcases}
\]

Throughout the statements and proofs below, for ease of notation, we occasionally omit the domains in the notation of the various function spaces when this does not cause confusion.

For $H \in C(\RR^d)$, $W \in C_0([0,T],\RR^m)$, and $u_0 \in UC(\RR^d)$,
\begin{equation}\label{solutionmap}
	(H,W,u_0) \mapsto S(H,W,u_0) \in UC(\RR^d \times [0,T])
\end{equation}
denotes the solution map for the equation \eqref{E:main}; that is, $u = S(H,W,u_0)$ is the viscosity solution of \eqref{E:main} whenever this makes sense. We point out the abuse of notation between here and the introduction, where the operator $S$ does not depend on the initial datum $u_0$.

If $(H,W) \in C(\RR^d) \times W_0^{1,1}([0,T],\RR^m)$ or $(H,W) \in \dc_\loc(\RR^d) \times C_0([0,T],\RR^m)$, then it follows from respectively the classical \cite{CIL} or pathwise \cite{LS1,LS2,Snotes} viscosity solution theory that, with respect to the initial datum, $S(H,W,\cdot)$ preserves boundedness and Lipschitz continuity, commutes with constants, and is contractive and monotone, that is,
\begin{equation}\label{Sproperties}
	\left\{
	\begin{split}
		&\text{if } u_0 \in BUC(\RR^d), \text{ then } S(H,W,u_0) \in BUC(\RR^d \times [0,T]),\\
		&\text{if } \nor{Du_0}{\oo} \le L, \text{ then } \nor{D_x S(H,W,u_0)}{\oo,\RR^d \times [0,T]} \le L,\\
		&\text{if } u_0 \in UC(\RR^d) \text{ and } k \in \RR, \text{ then } S(H,W,u_0 + k) = S(H,W,u_0) + k,\\
		&\text{if } u_0^1,u_0^2 \in UC(\RR^d), \text{ then } \nor{S(H,W,u_0^1) - S(H,W,u_0^2)}{\oo,\RR^d \times [0,T]} \le \nor{u_0^1 - u_0^2}{\oo}, \text{ and}\\
		&\text{if } u_0^1 \le u_0^2, \text{ then } S(H,W,u_0^1) \le S(H,W,u_0^2).
	\end{split}
	\right.
\end{equation}

We now present and prove the main results from the introduction, that is, we show that the solution operator extends in an appropriate sense to $\mathscr H_{\alpha,p} \times \mathscr P_{\alpha,p'}$ for $\alpha \in (0,1)$ and $p \in [1,\oo]$, and continues to satisfy \eqref{Sproperties}.

\begin{theorem}\label{T:main}
	Let $\alpha \in (0,1)$ and $p \in (1,\oo)$. Then \eqref{solutionmap} extends to a continuous map on $\mathscr H_{\alpha,p} \times \mathscr P_{\alpha,p'} \times UC(\RR^d)$ that satisfies \eqref{Sproperties} for every fixed $(H,W) \in \mathscr H_{\alpha,p} \times \mathscr P_{\alpha,p'}$.
\end{theorem}

The next result treats the cases $p = 1$ and $p = \oo$. Note that, in view of Lemma \ref{L:interp}\eqref{L:interpdense}, $\dc_\loc(\RR^d)$ is not dense in $\mathscr H_{\alpha,\oo}$ and $W^{1,1}_0([0,T],\RR^m)$ is not dense in $\mathscr P_{\alpha,\oo}$.

\begin{theorem}\label{T:mainweak}
Let $\alpha \in (0,1)$, $p \in [1,\oo]$, and $(H,W,u_0) \in \mathscr H_{\alpha,p} \times \mathscr P_{\alpha,p'} \times UC(\RR^d)$. Then there exists a unique $S(H,W,u_0) \in UC(\RR^d \times [0,T])$ such that the following hold:
\begin{enumerate}[(a)]
\item The properties of \eqref{Sproperties} are satisfied.
\item If $p < \oo$, $(W_n)_{n=1}^\oo \subset \mathscr P_{\alpha,p'}$, and
\begin{equation}\label{boundedpaths}
	\lim_{n \to \oo} \nor{W_n - W}{\oo,[0,T]}  = 0 \quad \text{and} \quad \sup_{n \in \NN} \nor{W_n}{\mathscr P_{\alpha,p'} } < \oo,
\end{equation}
then
\[
	\lim_{n \to \oo} \nor{ S(H, W_n, u_0) - S(H,W,u_0)}{\oo,\RR^d \times [0,T]} = 0.
\]
\item If $p' < \oo$, $(H_n)_{n=1}^\oo \subset \mathscr H_{\alpha,p}$,
\begin{equation}\label{boundedHs}
	\lim_{n \to \oo} H_n = H \text{ locally uniformly, and } \sup_{n \in \NN} \nor{H_n}{\mathscr H_{\alpha,p} } < \oo,
\end{equation}
then
\[
	\lim_{n \to \oo} \nor{S(H_n, W, u_0) - S(H,W,u_0)}{\oo,\RR^d \times [0,T]} = 0.
\]
\end{enumerate}
\end{theorem}

The proofs of the above two theorems rely on Proposition \ref{P:dcextension} and some other stability estimates that are proved next.

\begin{lemma}\label{L:easystability}
	Assume that $H \in \dc_\loc$, $W \in C_0$, $\zeta \in W^{1,1}_0$, and $u_0 \in C^{0,1}(\RR^d)$ with $\nor{Du_0}{\oo} \le L$. Then 
	\[
		\nor{S(H,W+\zeta,u_0) - S(H,W,u_0)}{\oo,\RR^d \times [0,T]} \le \sup_{|p| \le L} |H(p)| \int_0^T |\dot \zeta(t)|dt.
	\]
\end{lemma}

\begin{proof}
	Assume first that $W \in W^{1,1}_0$. Then the result follows easily from the comparison principle for classical viscosity solutions. The above estimate then holds for arbitrary $W \in C_0$ by density, since the left-hand side is continuous with respect to $W \in C_0$, and the right-hand side does not depend on $W$.
\end{proof}

In the next result, the stability of the solution operator $S$ is measured with respect to the $K$-functional defined in subsection \ref{SS:interp}.

\begin{lemma}\label{L:decomposed}
	There exists a constant $C = C_L > 0$ such that, if $u_0 \in C^{0,1}(\RR^d)$ with $\nor{Du_0}{\oo} \le L$, then the following hold:
	\begin{enumerate}[(a)]
	\item If $H_1,H_2 \in \dc_\loc$ and $W \in C_0$, then
	\begin{align*}
		\sup_{\RR^d \times [0,T]}& \abs{ S(H_1,W,u_0) - S(H_2,W,u_0)} \\
		&\le C \pars{ \nor{H_1}{\dc(B_L)} + \nor{H_2}{\dc(B_L)} } K\pars{ \frac{  \nor{H_1 - H_2}{\oo,B_L}}{  \nor{H_1}{\dc(B_L)} + \nor{H_2}{\dc(B_L)} } , W, C_0, W^{1,1}_0 }.
	\end{align*}
	\item If $H \in C$ and $W_1,W_2 \in W^{1,1}_0$, then
	\begin{align*}
		\sup_{\RR^d \times [0,T]} &\abs{ S(H,W_1,u_0) - S(H,W_2,u_0) }\\
		& \le C \nor{W_1 - W_2}{\oo,B_L} K \pars{ \frac{ \nor{\dot W_1}{L^1([0,T])} + \nor{\dot W_2}{L^1([0,T])} }{ \nor{W_1 - W_2}{\oo,[0,T]}} ,H, \dc_\loc, C}.
	\end{align*}
	\end{enumerate}
\end{lemma}

\begin{proof}
	Throughout the proofs of both parts, the constant $C > 0$ depends only on $L$, and may change from line to line.
	
	(a) We write $u_1 = S(H_1,W,u_0)$ and $u_2 = S(H_2,W,u_0)$. Using arguments as in \cite{Snotes}, we have, for all $x,y \in \RR^d$ and $t > 0$,
	\[
		u_1(x,t) - u_2(y,t) \le \phi(x-y,t),
	\]
	where
	\[
		d\phi = \pars{ H_1(D\phi) - H_2(D\phi)} \cdot dW \quad \text{in } \RR^d \times [0,T] \quad \text{and} \quad \phi(z,0) = L|z| \quad \text{for } z \in \RR^d.
	\]
	Let $Y \in W^{1,1}_0([0,T],\RR^m)$, write $W = X + Y$, and let $\psi$ solve
	\[
		d\psi = \pars{ H_1(D\psi) - H_2(D\psi)} \cdot dX \quad \text{in } \RR^d \times [0,T] \quad \text{and} \quad \psi(z,0) = L|z| \quad \text{for } z \in \RR^d.
	\]
	Then Proposition \ref{P:dcextension} gives
	\[
		\psi(z,t) \le L|z| + C \pars{ \nor{H_1}{\dc(B_L)} + \nor{H_2}{\dc(B_L)} }\nor{X}{\oo,[0,T]},
	\]
	while Lemma \ref{L:easystability} yields
	\[
		\phi(z,t) - \psi(z,t) \le \nor{H_1 - H_2}{\oo,B_L} \nor{\dot Y}{L^1([0,T])}.
	\]
	Combining all terms and using a symmetric argument for $u_2 - u_1$, we find that
	\[
		\nor{u_1 - u_2}{\oo,\RR^d \times [0,T]} \le C \pars{ \nor{H_1}{\dc(B_L)} + \nor{H_2}{\dc(B_L)} }\nor{X}{\oo,[0,T]} + \nor{H_1 - H_2}{\oo,B_L} \nor{\dot Y}{L^1([0,T])}.
	\]
	We conclude by taking the infimum over all such $X$ and $Y$.
	
	(b) We write $u_1 = S(H,W_1,u_0)$ and $u_2 = S(H,W_2,u_0)$. A similar argument as in part (a) gives that, for all $x,y \in \RR^d$ and $t > 0$,
	\[
		u_1(x,t) - u_2(y,t) \le \phi(x-y,t),
	\]
	where
	\[
		\phi_t = H(D\phi)  \pars{ \dot W_1 - \dot W_2} \quad \text{in } \RR^d \times [0,T] \quad \text{and} \quad \phi(z,0) = L|z| \quad \text{for } z \in \RR^d.
	\]
	Let $G \in \dc_\loc$, set $F = H - G$, and let $\psi$ solve
	\[
		\psi_t = G(D\psi) \pars{ \dot W_1 - \dot W_2} \quad \text{in } \RR^d \times [0,T] \quad \text{and} \quad \psi(z,0) = L|z| \quad \text{for } z \in \RR^d.
	\]
	Then Proposition \ref{P:dcextension} gives
	\[
		\psi(z,t) \le L|z| + C \nor{G}{\dc(B_L) }\nor{W_1 - W_2}{\oo,[0,T]},
	\]
	while an application of standard stability estimates from the viscosity theory yields
	\[
		\phi(z,t) - \psi(z,t) \le \nor{F}{\oo,B_L}\pars{  \nor{\dot W_1}{L^1([0,T])} + \nor{\dot W_2}{L^1([0,T])}}.
	\]
	Combining all terms and using a symmetric argument for $u_2 - u_1$, we find that
	\[
		\nor{u_1 - u_2}{\oo,\RR^d \times [0,T]} \le C \pars{ \nor{F}{\oo,B_L}\pars{  \nor{\dot W_1}{L^1([0,T])} + \nor{\dot W_2}{L^1([0,T])}} + \nor{G}{\dc(B_L)}\nor{W_1 - W_2}{\oo,[0,T]}}.
	\]
	We conclude by taking the infimum over all such $F$ and $G$.
\end{proof}	

\begin{proof}[Proof of Theorem \ref{T:main}]
	The properties in \eqref{Sproperties} are stable under uniform convergence, and therefore continue to hold upon extending the solution operator to the appropriate spaces of Hamiltonians or paths. Moreover, in view of the contraction property, by a density argument, it suffices to consider a fixed initial datum $u_0 \in C^{0,1}(\RR^d)$ with $\nor{Du_0}{\oo} \le L$ for some $L > 0$, and, hence, by \eqref{Sproperties}, it suffices to consider the relevant norms of Hamiltonians over the ball $B_L$.
	
Throughout, since $u_0$ is fixed, we suppress its dependence and write
\[
	S(H,W,u_0) = S(H,W).
\]
%Finally, in what follows, we almost always use the $K$-functionals
%\[
%	t \mapsto K(t,\cdot, \dc(B_L),C(B_L)) \quad \text{and} \quad 
%	t \mapsto K(t,\cdot,C_0([0,T],W^{1,1}_0([0,T]))
%\]
%for, respectively, Hamiltonians and paths. Therefore, to avoid cumbersome notation, and when it is clear from context, we will write, for $t > 0$, a Hamiltonian $H: \RR^d \to \RR$, and a path $W: [0,T] \to \RR^m$,
%\[
%	K(t,H,\dc(B_L),C(B_L)) = K(t,H) \quad \text{and} \quad K(t,W,C_0([0,T]),W^{1,1}_0([0,T])) = K(t,W).
%\]

In order to prove the result, it suffices to show that, for any $(H,W) \in \mathscr H_{\alpha,p} \times \mathscr P_{\alpha,p'}$ and $\eps > 0$ fixed, there exists $\delta \in (0,1]$ such that, if $(H_1,W_1), (H_2,W_2) \in \dc_\loc \times W^{1,1}_0$ satisfy, for $j=  1,2$,
	\[
		\nor{H_j - H}{\mathscr H_{\alpha,p}} + \nor{W_j - W}{\mathscr P_{\alpha,p'}} < \delta,
	\]
	then
	\[
		\nor{S(H_1,W_1) - S(H_2,W_2)}{\oo,\RR^d \times [0,T]} < \eps.
	\]
	Indeed, this means that $S(H,W)$ can be uniquely identified as the uniform limit of $S(\tilde H,\tilde W)$ as $(\tilde H,\tilde W) \in \dc_\loc \times W^{1,1}_0$ converges to $(H,W)$ in the $\mathscr H_{\alpha,p} \times \mathscr P_{\alpha,p'}$-norm, and the extended map $S$ is then continuous.
	
	In what follows, we write $u_1 = S(H_1,W_1)$, $u_2 = S(H_2, W_2)$, and $v = S(H_2,W_1)$.
	
	{\it Step 1.} We first estimate $u_1 - v$. By the definition of the $K$-functional, there exist $(H_{1,n}, H_{2,n})_{n \in \NN} \subset \dc(B_L)$ such that
	\[
		\nor{H_{1,n}}{\dc(B_L)} + 2^{n} \nor{H_1 - H_{1,n}}{\oo,B_L} \le 2K(2^n,H_1,\dc,C)
	\]
	and
	\[
		\nor{H_{2,n}}{\dc(B_L)} + 2^{n} \nor{H_2 - H_{2,n}}{\oo,B_L} \le 2K(2^n, H_2,\dc,C).
	\]
	For $n \in \NN$, set $u_{1,n} := S(H_{1,n}, W^1)$ and $v_n = S(H_{2,n},W^1)$. We have, by Lemma \ref{L:K}\eqref{L:Kmonotone},
	\[
		\nor{H_{1,n}}{\dc(B_L)} + \nor{H_{1,n+1}}{\dc(B_L)} \le 2\pars{ K(2^n,H_1,\dc,C) + K(2^{n+1},H_1,\dc,C)} \le 6 K(2^n,H_1,\dc,C)
	\]
	and
	\begin{align*}
		\nor{H_{1,n} - H_{1,n+1}}{\oo,B_L} &\le \nor{H_{1,n} - H_1}{\oo,B_L} + \nor{H_1 - H_{1,n+1}}{\oo,B_L} \\
		&\le 2\pars{ 2^{-n} K(2^n,H_1,\dc,C) + 2^{-n-1} K(2^{n+1},H_1,\dc,C)} \\
		&\le 2^{2-n} K(2^n,H_1,\dc,C).
	\end{align*}
	Then Lemmas \ref{L:K}\eqref{L:Kmonotone2} and \ref{L:decomposed}(a) give, for some $C = C_L > 0$,
	\begin{align*}
		&\nor{u_{1,n} - u_{1,n+1}}{\oo,\RR^d \times [0,T]}\\
		 &\le C \pars{ \nor{H_{1,n}}{\dc(B_L)} + \nor{H_{1,n+1}}{\dc(B_L)} } K\pars{ \frac{  \nor{H_{1,n} - H_{1,n+1}}{\oo,B_L}}{  \nor{H_{1,n}}{\dc(B_L)} + \nor{H_{1,n+1}}{\dc(B_L)} } , W_1, C_0, W^{1,1}_0 }\\
		&\le 6C K(2^n,H_1,\dc,C) K\pars{ 2^{-n}, W_1, C_0, W^{1,1}_0},
	\end{align*}
	and so, in view of H\"older's inequality and Lemma \ref{L:interp}\eqref{L:interpsum}, for some $C_{L,p,\alpha}$,
	\begin{equation}\label{unu}
		\begin{split}
		&\nor{u_{1,n} - u_1}{\oo,\RR^d \times [0,T]} \\
		&\le C \pars{ \sum_{m = n}^{\oo} K(2^m,H_1,\dc,C)^p 2^{-mp\alpha} }^{1/p} \pars{ \sum_{m = n}^{\oo} K(2^{-m},W_1,C_0,W^{1,1}_0)^{p'} 2^{mp'\alpha}}^{1/p'}\\
		&\le  C \pars{ \sum_{m = n}^{\oo} K(2^m,H,\dc,C)^p 2^{-mp\alpha} }^{1/p} \pars{ \sum_{m = n}^{\oo} K(2^{-m},W,C_0,W^{1,1}_0)^{p'} 2^{mp'\alpha}}^{1/p'} \\
		&+ C\pars{ \nor{H}{\mathscr H_{\alpha,p}} + \nor{W}{\mathscr P_{\alpha,p'}} + 1} \delta.
		\end{split}
	\end{equation}
	A similar argument gives
	\begin{equation}\label{vnv}
	\begin{split}
		&\nor{v_{n} - v}{\oo} \\
		&\le C \pars{ \sum_{m = n}^{\oo} K(2^m, H_2,\dc,C)^p 2^{-mp\alpha} }^{1/p} \pars{ \sum_{m = n}^{\oo} K(2^{-m},W_1,C_0,W^{1,1}_0)^{p'} 2^{mp'\alpha}}^{1/p'}\\
		&\le C \pars{ \sum_{m = n}^{\oo} K(2^m,H,\dc,C)^p 2^{-mp\alpha} }^{1/p} \pars{ \sum_{m = n}^{\oo} K(2^{-m},W,C_0,W^{1,1}_0)^{p'} 2^{mp'\alpha}}^{1/p'} \\
		&+ C\pars{ \nor{H}{\mathscr H_{\alpha,p}} + \nor{W}{\mathscr P_{\alpha,p'}} + 1} \delta 
	\end{split}
	\end{equation}
	
	We next estimate $u_{1,n} - v_{n}$. First, Lemma \ref{L:interp}\eqref{L:interpsum} yields, for some $C = C_{p,\alpha} > 0$ and for all $n \in \NN$ and $j = 1,2$,
	\[
		K(2^{n},H_j,\dc,C) \le C \nor{H_j}{\mathscr H_{\alpha,p}} 2^{\alpha n}
		\le C\pars{ \nor{H}{\mathscr H_{\alpha,p}} + 1} 2^{\alpha n},
	\]
	and, hence,
	\[
		\nor{H_{j,n}}{\dc(B_L)} \le 2 K(2^{n}, H_j,\dc,C) \le C\pars{ \nor{H}{\mathscr H_{\alpha,p}} + 1} 2^{\alpha n}
	\]
	and
	\begin{align*}
		\nor{H_{1,n} - H_{2,n}}{\oo,B_L} 
		&\le 2^{1-n} \left[ K(2^{n}, H_1,\dc,C) + K(2^{n}, H_2,\dc,C) \right] + C \delta\\
		&\le C \pars{ \nor{H}{\mathscr H_{\alpha,p}} + 1} 2^{-(1-\alpha) n} + C \delta.
	\end{align*}
	Lemmas \ref{L:K}\eqref{L:Kmonotone}-\eqref{L:Kmonotone2}, \ref{L:interp}\eqref{L:interpsum}, and \ref{L:decomposed}(a) then give, for some $C = C_{L,p,\alpha}$,
	\begin{equation}\label{unvn}
	\begin{split}
		&\nor{u_{1,n} - v_{n}}{\oo,\RR^d \times [0,T]} \\
		&\le C \pars{ \nor{H}{\mathscr H_{\alpha,p}} + 1} 2^{\alpha n} K \pars{ \frac{ \pars{ \nor{H}{\mathscr H_{\alpha,p}} + 1} 2^{-(1-\alpha) n} +  \delta}{ \pars{ \nor{H}{\mathscr H_{\alpha,p}} + 1} 2^{\alpha n}}, W_1, C_0,W^{1,1}_0}
		\\
		&\le C\pars{ \nor{H}{\mathscr H_{\alpha,p}} + 1} 2^{\alpha n}
K\pars{ 2^{-n}, W_1,C_0,W^{1,1}_0} \\
		&+ C \pars{ \nor{H}{\mathscr H_{\alpha,p}} + 1} 2^{\alpha n} K\pars{ \frac{\delta}{\pars{ \nor{H}{ \mathscr H_{\alpha,p} }+ 1} 2^{\alpha n}},W_1,C_0,W^{1,1}_0 }\\
		&\le C\pars{ \nor{H}{\mathscr H_{\alpha,p}} + 1}\pars{ 2^{\alpha n}
K\pars{ 2^{-n}, W, C_0, W^{1,1}_0} + \delta} \\
		&+ C \pars{ \nor{H}{\mathscr H_{\alpha,p}} + 1}^{1-\alpha} \pars{\nor{W}{\mathscr P_{\alpha,p'}}+1} 2^{\alpha(1-\alpha) n} \delta^\alpha.
	\end{split}
	\end{equation}
	Combining \eqref{unu}, \eqref{vnv}, and \eqref{unvn}, we find that, for any $n \in \NN$,
	\begin{equation}\label{u-v}
	\begin{split}
		\nor{u_1-v}{\oo,\RR^d \times [0,T]} 
		&\le C \pars{ \sum_{m = n}^{\oo} K(2^m,H,\dc,C)^p 2^{-mp\alpha} }^{1/p} \pars{ \sum_{m = n}^{\oo} K(2^{-m},W,C_0,W^{1,1}_0)^{p'} 2^{mp'\alpha}}^{1/p'}\\
		&+ C\pars{ \nor{H}{\mathscr H_{\alpha,p}} + 1} 2^{\alpha n}
K\pars{ 2^{-n}, W,C_0, W^{1,1}_0} \\
		&+ C \pars{ \nor{H}{\mathscr H_{\alpha,p}} + \nor{W}{\mathscr P_{\alpha,p'}} + 1}\delta\\
		&+ C \pars{ \nor{H}{\mathscr H_{\alpha,p}} + 1}^{1-\alpha} \nor{W}{\mathscr P_{\alpha,p'}} 2^{\alpha(1-\alpha) n} \delta^\alpha.
	\end{split}
	\end{equation}
	In view of Lemma \ref{L:interp}\eqref{L:interpsum}, the first two terms converge to $0$ as $n \to \oo$. Therefore, choosing first $n$ large and then $\delta$ small, depending only on $H$ and $W$, we achieve
	\[
		\nor{u_1 - v}{\oo} < \frac{\eps}{2}.
	\]
	
	{\it Step 2.} We now estimate $v - u_2$, using similar arguments as in Step 1, but this time invoking Lemma \ref{L:decomposed}(b).
	
	Let $(W_{1,n}, W_{2,n})_{n \in \NN} \subset W^{1,1}_0$ be such that
	\[
		\nor{W_1 - W_{1,n}}{\oo,[0,T]} + 2^{-n} \nor{\dot W_{1,n}}{L^1([0,T])} \le 2K(2^{-n},W_1,C_0,W^{1,1}_0)
	\]
	and
	\[
		\nor{W_2 - W_{2,n}}{\oo,[0,T]} + 2^{-n} \nor{\dot W_{2,n}}{L^1([0,T])} \le 2K(2^{-n},W_2,C_0,W^{1,1}_0).
	\]
	For $n \in \NN$, set $u_{2,n} := S(H_2, W_{2,n})$ and $\tilde v_n = S(H_2,W_{1,n})$. Then Lemma \ref{L:decomposed}(b) gives, for some $C = C_L > 0$,
	\begin{align*}
		\nor{u_{2,n} - u_{2,n+1}}{\oo,\RR^d \times [0,T]} &\le C \nor{W_{2,n} - W_{2,n+1}}{\oo,[0,T]} K\pars{ \frac{\nor{\dot W_{2,n}}{1,[0,T]} + \nor{\dot W_{2,n+1}}{1,[0,T]} }{ \nor{W_{2,n} - W_{2,n+1}}{\oo,[0,T]} } , H_2, \dc_\loc, C }\\
		&\le C K\pars{ 2^{-n}, W_2, C_0, W^{1,1}_0} K(2^n, H_2,\dc,C),
	\end{align*}
	and so, in view of H\"older's inequality and Lemma \ref{L:interp}\eqref{L:interpsum}, for all $n \in \NN$, we have
	\begin{equation}\label{tildeunu}
		\begin{split}
		&\nor{u_{2,n} - u_2}{\oo,\RR^d \times [0,T]}\\
		&\le C\pars{ \sum_{m = n}^{\oo} K(2^{-m},W_2,C_0,W^{1,1}_0)^{p'} 2^{mp'\alpha}}^{1/p'} \pars{ \sum_{m = n}^{\oo} K(2^m, H_2,\dc,C)^p 2^{-mp\alpha} }^{1/p} \\
		&\le C \pars{ \nor{W}{\mathscr P_{\alpha,p'}} + \nor{H}{\mathscr H_{\alpha,p} }  +1}\delta \\
		&+ C\pars{ \sum_{m = n}^{\oo} K(2^{-m},W,C_0,W^{1,1}_0)^{p'} 2^{mp'\alpha}}^{1/p'} \pars{ \sum_{m = n}^{\oo} K(2^m,H,\dc,C)^p 2^{-mp\alpha} }^{1/p}.
		\end{split}
	\end{equation}
	A similar argument gives
	\begin{equation}\label{tildevnv}
	\begin{split}
		&\nor{\tilde v_{n} - v}{\oo,\RR^d \times[ 0,T]}\\
		&\le C \pars{ \nor{W}{\mathscr P_{\alpha,p'}} + \nor{H}{\mathscr H_{\alpha,p} }  +1}\delta \\
		&+ C\pars{ \sum_{m = n}^{\oo} K(2^{-m},W,C_0,W^{1,1}_0)^{p'} 2^{mp'\alpha}}^{1/p'} \pars{ \sum_{m = n}^{\oo} K(2^m,H,\dc,C)^p 2^{-mp\alpha} }^{1/p}.
	\end{split}
	\end{equation}
	
	We next estimate $u_{2,n} - \tilde v_{n}$. Arguing as in Step 1, there exists $C = C_{p,\alpha} > 0$ such that
	\[
		\nor{W_{1,n} - W_{2,n}}{\oo,[0,T]} \le C\delta + C\pars{ \nor{W}{\mathscr P_{\alpha,p'}} + 1} 2^{-n \alpha}
	\]
	and
	\[
		\nor{\dot W_{1,n}}{1,[0,T]} + \nor{\dot W_{2,n}}{1,[0,T]} \le C 2^{n(1-\alpha)} \pars{ \nor{W}{\mathscr P_{\alpha,p'}} + 1}.
	\]
	Lemmas \ref{L:K}\eqref{L:Kscales},\eqref{L:Kmonotone},\eqref{L:Kmonotone2} and \ref{L:decomposed}(b) together give
	\begin{equation}\label{tildeunvn}
	\begin{split}
		&\nor{u_{2,n} - \tilde v_{n}}{\oo,\RR^d \times [0,T]} \\
		&\le C\pars{ \delta + 2^{-n\alpha} \pars{ \nor{W}{\mathscr{P}_{\alpha,p'}} + 1} }
		K \pars{ \frac{2^{n(1-\alpha)} \pars{ \nor{W}{\mathscr P_{\alpha,p'}}+1}}{\delta +  2^{-n\alpha} \pars{ \nor{W}{\mathscr{P}_{\alpha,p'}} + 1}} , H_2,\dc_\loc, C}\\
		&= C	2^{n(1-\alpha)} \pars{ \nor{W}{\mathscr P_{\alpha,p'}}+1} K \pars{ \frac{ \delta +  2^{-n\alpha} \pars{ \nor{W}{\mathscr{P}_{\alpha,p'}} + 1}}{ 2^{n(1-\alpha)} \pars{ \nor{W}{\mathscr P_{\alpha,p'}}+1}}, H_2, C, \dc_\loc}\\
		&\le C\pars{ \nor{W}{\mathscr P_{\alpha,p'}} + 1} 2^{n(1-\alpha)}
K\pars{ 2^{-n}, H_2,C, \dc_\loc} \\
		& \quad + C \pars{ \nor{W}{\mathscr P_{\alpha,p'}} + 1} 2^{n(1-\alpha)} K\pars{ \frac{\delta}{\pars{ \nor{W}{ \mathscr P_{\alpha,p'} }+ 1} 2^{n(1-\alpha)}}, H_2, C,\dc_\loc }\\
		&\le C\pars{ \nor{W}{\mathscr P_{\alpha,p'}} + 1} \pars{ 2^{n(1-\alpha)}
K\pars{ 2^{-n},  H,C, \dc_\loc} + \delta} \\
		& \quad + C \pars{ \nor{W}{\mathscr P_{\alpha,p'}} + 1}^{\alpha} \pars{ \nor{H}{\mathscr H_{\alpha,p}} + 1} 2^{n\alpha(1-\alpha)} \delta^{1-\alpha}.
	\end{split}
	\end{equation}
	Combining \eqref{tildeunu}, \eqref{tildevnv}, and \eqref{tildeunvn}, we find that
	\begin{equation}\label{tildeu-v}
	\begin{split}
		&\nor{u_2-  v}{\oo,\RR^d \times [0,T]} \\
		&\le C\pars{ \sum_{m = n}^{\oo} K(2^{-m},W,C_0,W^{1,1}_0)^{p'} 2^{mp'\alpha}}^{1/p'} \pars{ \sum_{m = n}^{\oo} K(2^m,H,\dc,C)^p 2^{-mp\alpha} }^{1/p}\\
		& \quad +C\pars{ \nor{W}{\mathscr P_{\alpha,p'}} + 1} 2^{n(1-\alpha)}
K\pars{ 2^{-n}, H,C, \dc_\loc} \\
		& \quad + C \pars{ \nor{W}{\mathscr P_{\alpha,p'}} + \nor{W}{\mathscr P_{\alpha,p'}}^\alpha + \nor{H}{\mathscr H_{\alpha,p}} + 1} 2^{n\alpha(1-\alpha)} \delta^{1-\alpha}.
	\end{split}
	\end{equation}
	Lemma \ref{L:interp}\eqref{L:interpreflexive} implies that $H \in (C, \dc)_{1-\alpha,p}$, and so, in view of Lemma \ref{L:interp}\eqref{L:interpsum}, the first two terms converge to $0$ as $n \to \oo$. Therefore, choosing first $n$ large and then $\delta$ small, depending only on $H$ and $W$, we can assure that
	\[
		\nor{u_2 - v}{\oo,\RR^d \times [0,T]} < \frac{\eps}{2}.
	\]
	and, hence,
	\[
		\nor{u_1 - u_2}{\oo,\RR^d \times [0,T]} < \eps.
	\]
\end{proof}

\begin{proof}[Proof of Theorem \ref{T:mainweak}]
	As in the proof of Theorem \ref{T:main}, part (a) is immediate upon proving parts (b) and (c), and so we may assume $u_0$ is fixed and $\nor{Du_0}{\oo} \le L$. We then write $S(H,W) := S(H,W,u_0)$, and all norms of Hamiltonians are taken over $B_L$. We prove only part (b), since the proof of part (c) is similar.
	
	Arguing similarly as in the proof of Theorem \ref{T:main}, the result follows once we show that, for fixed $\eps > 0$ and $R > 0$, there exists $\delta > 0$ such that, if $W_1,W_2 \in W^{1,1}_0$ satisfy
	\[
		\nor{W_1 - W}{\oo,[0,T]} \vee \nor{W_2 - W}{\oo,[0,T]} < \delta \quad \text{and} \quad \nor{W_1}{\mathscr P_{\alpha,p'}} \vee \nor{W_2}{\mathscr P_{\alpha,p'}} \le R,
	\]
	then
	\[
		\nor{S(H,W_1) - S(H,W_2)}{\oo,\RR^d \times [0,T]} < \eps.
	\]
	
	For $n \in \NN$ and $j = 1,2$, choose $W_{j,n} \in W^{1,1}_0$ such that
	\[
		\nor{W_j - W_{j,n}}{\oo,[0,T]} + 2^{-n} \nor{\dot W_{j,n}}{1,[0,T]} \le 2K(2^{-n}, W_j, C_0, W^{1,1}_0),
	\]
	and set
	\[
		u_j := S(H,W_j) \quad \text{and} \quad u_{j,n} := S(H,W_{j,n}).
	\]
	Then, arguing just as for \eqref{tildeunu} and \eqref{tildevnv} from the proof of Theorem \ref{T:main}, we obtain
	\begin{equation}\label{unmun}
		\begin{split}
		\nor{u_{j,n} - u_j}{\oo} &\le C \pars{ \sum_{m = n}^\oo K(2^{-m}, W_j,C_0,W_0^{1,1})^{p'} 2^{mp'\alpha}}^{1/p'} \pars{ \sum_{m=n}^\oo K(2^m, H, \dc, C)^p 2^{-m p \alpha}}^{1/p}\\
		&\le CR \pars{ \sum_{m=n}^\oo K(2^m, H, \dc, C)^p 2^{-m p \alpha}}^{1/p},
		\end{split}
	\end{equation}
	where, in the first line, the $\ell^{p'}$-norm is replaced with a supremum when $p' = \oo$.
	
	Next, by Lemma \ref{L:interp}\eqref{L:interpsum}, for $n \in \NN$,
	\begin{align*}
		\nor{W_{1,n} - W_{2,n}}{\oo} &\le \nor{W_1 - W_2}{\oo,[0,T]} + 2\left[ K(2^{-n}, W_1, C_0, W^{1,1}_0) + K(2^{-n} , W_2, C_0, W^{1,1}_0) \right]\\
		&\le 2\delta + CR 2^{-n\alpha}
	\end{align*}
	and
	\[
		\nor{\dot W_{1,n}}{L^1([0,T])} + \nor{\dot W_{1,n}}{L^1([0,T])} \le C R 2^{n(1-\alpha)}.
	\]
	Then Lemma \ref{L:decomposed}(b) and a similar argument as for \eqref{tildeunvn} give, for $n \in \NN$,
	\begin{equation}\label{unmum}
		\nor{u_{1,n} - u_{2,n}}{\oo,\RR^d \times [0,T]}
		\le CR2^{n(1-\alpha)}K(2^{-n} , H, C, \dc_\loc) + CR^\alpha \nor{H}{\mathscr H_{\alpha,p} }2^{n\alpha(1-\alpha)} \delta^{1-\alpha}.
	\end{equation}
	Combining \eqref{unmun} and \eqref{unmum} yields
	\begin{align*}
		\nor{u_1 - u_2}{\oo,\RR^d \times [0,T]} 
		&\le CR \pars{ \sum_{m=n}^\oo K(2^m, H, \dc, C)^p 2^{-m p \alpha}}^{1/p} \\
		&+ CR2^{n(1-\alpha)}K(2^{-n} , H, C, \dc_\loc) + CR^\alpha \nor{H}{\mathscr H_{\alpha,p} }2^{n\alpha(1-\alpha)} \delta^{1-\alpha}.
	\end{align*}
	Because $p < \oo$, the first two terms can be made less that $\eps/2$ by taking $n$ sufficiently large, and then, for sufficiently small $\delta$, the third term is less than $\eps/2$.
	\end{proof}

\section{The space $\mathscr H_{\alpha,p}$}\label{S:H}

In this section, we give various examples of spaces that embed continuously into the interpolation space given, for $\alpha \in (0,1)$ and $p \in [1,\oo]$, by
\[
	\mathscr H_{\alpha,p} = \pars{ \dc(\RR^d), C(\RR^d) \cap L^\oo(\RR^d)}_{\alpha,p;\loc},
\]
that is, $H \in \mathscr H_{\alpha,p}$ if and only if $H \in (\dc(B_L), C(B_L))_{\alpha,p}$ for all $L > 0$.

We do not know of a completely analytic characterization of the space $\dc$, let alone $\mathscr H_{\alpha,p}$. In what follows, we make use of the many properties of $\dc$ described in Appendix \ref{S:dc}.

In various points in this section, we regularize functions by convolving with a mollifier $\rho \in C^\oo(\RR^d)$ satisfying
\begin{equation}\label{mollifier}
		\rho \ge 0, \quad \supp \rho \subset B_1, \quad \rho(x) = \rho(-x) \text{ for all } x \in \RR^d, \quad \text{and} \quad \int_{\RR^d} \rho = 1.
\end{equation}

\subsection{A first sufficient criterion}

The fact that $C^{1,1} \subset \dc$ leads to a criterion for belonging to $\mathscr H_{\alpha,p}$ in terms of Besov regularity. 

In the lemma below, recall the definition of the second order difference operator $\Delta^2_y$ from subsection \ref{SS:besov} and the $K$-functional from subsection \ref{SS:interp}.

\begin{lemma}\label{L:KC11C}
	There exist universal constants $0 < c < C$ such that, for all $f \in C(\RR^d)$ and $t \in [1,\oo)$,
	\[
		c\pars{ \nor{f}{\oo} + t\sup_{|y| \le t^{-1/2}} \nor{\Delta^2_y f}{\oo} }
		\le K(t,f,C^{1,1}(\RR^d), C(\RR^d))
		\le C\pars{ \nor{f}{\oo} +t \sup_{|y| \le t^{-1/2}} \nor{\Delta^2_y f}{\oo} }.
	\]
\end{lemma}

\begin{proof}
	Let $f_1 \in C^{1,1}(\RR^d)$ and $f_2 \in C(\RR^d)$ be such that $f = f_1 + f_2$. Then, for any $y \in \RR^d$ with $|y| \le t^{-1/2}$ and $x \in \RR^d$,
	\[
		|\Delta^2_y f(x)| \le |\Delta^2_y f_1(x)| + |\Delta^2_y f_2(x)| \le \frac{1}{t} \nor{D^2 f_1}{\oo} + 4 \nor{f_2}{\oo}.
	\]
	Therefore,
	\[
		\nor{f}{\oo} + t\sup_{|y| \le t^{-1/2}} \nor{\Delta^2_y f}{\oo}
		\le \nor{f_1}{C^{1,1}(\RR^d)} + (1+4t)\nor{f_2}{\oo} \le 5 \pars{ \nor{f_1}{C^{1,1}(\RR^d)} + t \nor{f_2}{\oo}}.
	\]
	Taking the infimum over all such $f_1$ and $f_2$ yields the first inequality with $c = \frac{1}{5}$. 
	
	Next, let $\rho$ be as in \eqref{mollifier}, and, for $\delta > 0$ and $x \in \RR^d$, set 
	\[
		\rho_\delta(x) = \frac{1}{\delta^d} \rho\pars{ \frac{x}{\delta}}
	\]
	and $f_\delta = f * \rho_\delta$. Then, for all $x \in \RR^d$, because $\rho$ is even,
	\begin{align*}
		f_\delta(x) - f(x) &= \int_{B_\delta} \pars{ f(x-y) - f(x)} \rho_\delta(y)dy = \int_{B_\delta} \pars{ f(x+y) - f(x)} \rho_\delta(y)dy\\
		& = \frac{1}{2} \int_{B_\delta} \pars{ f(x+y) + f(x-y) - 2f(x)} \rho_\delta(y)dy.
	\end{align*}
	This gives
	\[
		\abs{ f_\delta(x) - f(x)} \le \frac{1}{2} \int_{B_\delta} \abs{ \Delta^2f_y(x)} \rho_\delta(y)dy \le \frac{1}{2} \sup_{|y| \le \delta}\nor{ \Delta^2 f_y}{\oo}.
	\]
	Similarly, for all $x \in \RR^d$,
	\[
		D^2 f_\delta(x) = \frac{1}{2} \int_{B_\delta} \Delta^2 f_y(x) D^2 \rho_\delta(y)dy,
	\]
	and so
	\[
		\abs{ D^2 f_\delta(x)} \le \frac{1}{2\delta^2} \int_{B_1} \abs{ D^2 \rho(y)}dy \sup_{|y| \le \delta} \nor{\Delta^2 f_y}{\oo}.
	\]
	We also have $\nor{f_\delta}{\oo} \le \nor{f}{\oo}$ and
	\[
		\nor{Df_\delta}{\oo} \le 2 \sqrt{ \nor{f_\delta}{\oo} \nor{D^2 f_\delta}{\oo} } \le \nor{f_\delta}{\oo} + \nor{D^2 f_\delta}{\oo}.
	\]
	Now fix $t \ge 1$ and set $\delta = t^{-1/2}$. Then
	\[
		K(t, f, C^{1,1}, C) \le \nor{f_\delta}{\oo} + \nor{Df_\delta}{\oo} + \nor{D^2 f_\delta}{\oo} + t \nor{f - f_\delta}{\oo} \le C\pars{  \nor{f}{\oo} + t \sup_{|y| \le t^{-1/2}} \nor{\Delta^2_y f}{\oo} }.
	\]
\end{proof}

Along with Theorem \ref{T:mainweak} and Propositions \ref{P:variationandHolder}, \ref{P:PapBesov}, and \ref{P:Brownian}, taking $p = 1$ in the result below leads to proofs of Theorems \ref{T:introBesovresult} and \ref{T:introBrownian} from the introduction.
\begin{proposition}\label{P:Hapgeneral}
	For any $\alpha \in (0,1)$ and $p \in [1,\oo]$,
	\[
		B^{2(1-\alpha)}_{\oo,p} \subset \mathscr H_{\alpha,p}.
	\]
\end{proposition}

\begin{proof}
	By Proposition \ref{P:C11dc},
	\[
		\mathscr H_{\alpha,p} \supset (W^{2,\oo}, C \cap L^\oo)_{\alpha,p}.
	\]
	Given $f \in B^{2(1-\alpha)}_{\oo,p}$, we now use \eqref{equivalentbesov} and Lemmas \ref{L:K}\eqref{L:Kscales}, \ref{L:interp}\eqref{L:interpsum}, and \ref{L:KC11C} to write, for some constant $C = C_{p,\alpha}$ that changes from line to line,
	\begin{align*}
		\nor{f}{(W^{2,\oo},C_b)_{\alpha,p}}
		&\le C \left[ \int_1^\oo \frac{K(t,f,C^{1,1},C)^p}{t^{\alpha p+1}} dt \right]^{1/p}\\
		&\le C\nor{f}{\oo} + C \left[ \int_1^\oo \frac{t^p \sup_{|y| \le t^{-1/2}} \nor{\Delta^2_y f}{\oo}^p }{t^{\alpha p+1}} dt \right]^{1/p}\\
		&= C \nor{f}{\oo} + C \left[ \int_0^1 \frac{ \sup_{|y| \le s} \nor{\Delta^2_y f}{\oo}^p}{s^{2(1-\alpha)p + 1}}ds \right]^{1/p} \le C \nor{f}{B^{2(1-\alpha)}_{\oo,p}}.
	\end{align*}
\end{proof}

When $\alpha = 1/2$, Proposition \ref{P:Hapgeneral} implies that $B^1_{\oo, p ; \loc} \subset \mathscr H_{1/2,p}$. If $p > 1$, $B^1_{\oo,p;\loc}$ contains functions that are not Lipschitz. On the other hand, $B^1_{\oo, 1;\loc}$ is contained in $C^1$. This case is of particular interest, since $\mathscr P_{1/2,\oo}$ is the largest path space contained Brownian paths (see Proposition \ref{P:Brownian}).

We next explore ways to obtain weaker criteria by using further properties of the space $\dc$.

\subsection{One spatial dimension}

If $d = 1$, then, by Proposition \ref{P:1ddc},
\[
	(W^{2,1}, C \cap L^\oo)_{\alpha,p;\loc} \subset \mathscr H_{\alpha,p}.
\]
We do not know how to completely characterize this interpolation space. Here, we give an example of a particular subspace. In particular, while Proposition \ref{P:Hapgeneral} establishes that 
\[
	B^{2(1-\alpha)}_p(L^\oo)_\loc \subset \mathscr H_{\alpha,p},
\]
we show that, for $d = 1$, we can replace the underlying $L^\oo$-metric with a suitable Lorentz space (see subsections \ref{SS:LL} and \ref{SS:besov}), at the price of replacing the ``auxiliary'' exponent $p$ above with $1$.

\begin{proposition}\label{P:1dbesov}
	Assume that $\alpha \in (0,1)$, $p \in [1,\oo]$, and $d = 1$. Then
	\[
		B^{2(1-\alpha)}_1(L^{\frac{1}{1-\alpha},p}) \subset \mathscr H_{\alpha,p}.
	\]
\end{proposition}

When $p = \frac{1}{1-\alpha}$, the space on the left-hand side above becomes a Besov space, so that, in particular,
\[
	B^{2(1-\alpha)}_{\frac{1}{1-\alpha},1;\loc} \subset \mathscr H_{\alpha,p} \quad \text{for all } p \ge \frac{1}{1-\alpha}.
\]
By the remark in subsection \ref{SS:besov}, we have, for all $r > \frac{1}{1-\alpha}$,
\[
	B^{2(1-\alpha)}_{r,1;\loc} \subset B^{2(1-\alpha)}_1(L^{\frac{1}{1-\alpha},p})_\loc \subset \mathscr H_{\alpha,p},
\]
which, when $p = 1$, gives Theorem \ref{T:intro1d} from the introduction.

\begin{proof}[Proof of Proposition \ref{P:1dbesov}]
	The embeddings $B^0_{\oo, 1} \subset C$ and $B^2_{1,1} \subset W^{2,1}$ (see \cite{T}) imply that
	\[
		\pars{ B^2_{1,1}, B^0_{\oo,1}}_{\alpha,p} \subset (W^{2,1}, C)_{\alpha,p}.
	\]
	With the method of retracts (see \cite{BL}), this interpolation space can be identified with
	\[
		\pars{ \ell^{2,1}(L^1), \ell^{0,1}(L^\oo)}_{\alpha,p},
	\]
	where $\ell^{s,p}(X)$ denotes the weighted sequence space discussed in subsection \ref{SS:sequence}. This gives, in turn,
	\begin{equation}\label{Cwikelinclusion}
		\ell^{2(1-\alpha), 1}( (L^1,L^\oo)_{\alpha,p}) \subset \pars{ \ell^{2,1}(L^1), \ell^{0,1}(L^\oo)}_{\alpha,p}.
	\end{equation}
	The proof is finished in view of the definition of the Besov-Lorentz space and the fact that $(L^1,L^\oo)_{\alpha,p} = L^{\frac{1}{1-\alpha}, p}$.
\end{proof}

The embedding \eqref{Cwikelinclusion} is a consequence of the more general fact that, if $A_0$ and $A_1$ are compatible normed spaces, $\alpha \in (0,1)$, and $1 \le p \le q \le \oo$, then
\[
	L^p((A_0,A_1)_{\alpha,q}) \subset \pars{ L^p(A_0), L^p(A_1) }_{\alpha,q},
\]
and the reverse holds when $q \le p$, as a consequence of Minkowski's integral inequality. Thus, the above inclusion (and the reverse for $q \ge p$) becomes an equality exactly when $p = q$ (see \cite{BL,LPinterp}). It is a result of Cwikel \cite{C} that, in general, if $p \ne q$, the inclusions are strict.

When $\alpha = 1/2$, Proposition \ref{P:1dbesov} gives
\[
	B^1_1\pars{ L^{2,p}}_\loc \subset \mathscr H_{1/2,p},
\]
and hence, when $d = 1$, $\mathscr H_{1/2,p}$ contains non-Lipschitz functions, even when $p = 1$. 

\subsection{Further results in multiple dimensions}

We now exploit more properties of $\dc$-functions in order to find other criteria for belonging to $\mathscr H_{\alpha,p}$. In particular, the results that follow establish Theorem \ref{T:introstructure} from the introduction.

The first such result exploits the fact that, with a strong structural assumption, $W^{2,q}$-functions are $\dc$ for sufficiently large $q$.

\begin{proposition}\label{P:radialHap}
	Assume that $r > \frac{d}{1-\alpha}$ and $f \in B^{2(1-\alpha)}_1(L^{r,p})_\loc$ is radial. Then $f \in \mathscr H_{\alpha,p}$.
\end{proposition}

\begin{proof}
	We argue similarly as in Proposition \ref{P:1dbesov}. First, since $(1-\alpha) r > d$, Proposition \ref{P:radialfdc} gives
	\[
		(\dc, C)_{\alpha,p} \supset (W^{2,(1-\alpha) r}_{\rad}, C_\rad)_{\alpha,p}.
	\]
	The inclusions $W^{2,(1-\alpha) r}_{\rad} \supset B^2_{(1-\alpha)r, 1; \rad}$ and $C_\rad \supset B^0_{\oo,1;\rad}$ give
	\[
		(\dc, C)_{\alpha,p} \supset \pars{ B^2_{(1-\alpha)r, 1; \rad}, B^0_{\oo,1;\rad}}_{\alpha,p}.
	\]
	Without loss of generality, the partition-of-unity function $\phi$ in \eqref{pou} may be chosen to be radial. As a consequence, the functions $\psi$ and $(\phi_k)_{k \in \NN}$ are radial, and thus so are $f * \psi$ and $(f * \phi_k)_{k \in \NN}$. Therefore, the interpolation space $\pars{ B^2_{(1-\alpha)r, 1; \rad}, B^0_{\oo,1;\rad}}_{\alpha,p}$ can be identified with
	\[
		\pars{ \ell^{2,1}(L^{(1-\alpha)r}_\rad), \ell^{0,1}(L^\oo_\rad) }_{\alpha,p},
	\]
	which, by the same reasoning as in the proof of Proposition \ref{P:1dbesov}, contains $\ell^{2(1-\alpha),1}(L^{r,p}_\rad)$, as desired.
\end{proof}

If $\alpha < 1/2$, or if $\alpha = 1/2$ and $p = 1$, then the Besov regularity specified by Proposition \ref{P:Hapgeneral} is stronger than $C^1$-regularity. On the other hand, $\dc$ contains many non-$C^1$ functions. This indicates that the regularity condition can be relaxed if one interpolates to $\dc$ functions that are weaker that $C^{1,1}$. As an example, we show that the gradients of functions belonging to $\mathscr H_{\alpha,p}$ can be discontinuous across an affine hyperplane.

\begin{proposition}\label{P:Happlane}
	Fix $a \in \RR$ and $v \in S^{d-1}$, and define
	\[
		\Gamma := \left\{ x \in \RR^d : v \cdot x = a\right\}
		\quad \text{and} \quad
		U_\pm := \left\{ x \in \RR^d: \pm (v \cdot x - a) > 0 \right\}.
	\]
	Let $0 < \alpha \le 1/2$ and $1 \le p \le \oo$, and assume that $f \in C^{0,1}(\RR^d)$ is such that there exist $g_+,g_- \in B^{2(1-\alpha)}_{\oo,p}$ such that $f = g_\pm$ on $U_\pm \cup \Gamma$. Then $f \in \mathscr H_{\alpha,p}$.
\end{proposition}

\begin{proof}
	We assume, without loss of generality, that $v = (1,0,\ldots,0)$ and $a = 0$, and we introduce the notation $x = (x_1,x') \in \RR^d$, where $x_1 \in \RR$ and $x' = (x_2,x_3,\ldots, x_d) \in \RR^{d-1}$. Then
	\[
		\Gamma := \left\{ x \in \RR^d : x_1 = 0 \right\}
		\quad \text{and} \quad
		U_\pm := \left\{ x \in \RR^d: \pm x_1 > 0 \right\}.
	\]
	By Proposition \ref{P:Hapgeneral}, $g_- \subset \mathscr H_{\alpha,p}$. Therefore, by considering the function $f - g_-$, we may also assume that $g_- = 0$. This implies that $f(0,x') = 0$ for all $x' \in \RR^{d-1}$. For ease of notation below, we set $g := g_+$.
	
	For $\rho$ as in \eqref{mollifier} and for $\delta > 0$, set $\rho_\delta(x) := \frac{1}{\delta^d} \rho\pars{ \frac{x}{\delta}}$ and
	\[
		f_\delta(x) :=
		\begin{dcases}
			(g * \rho_\delta)(x) & \text{if } x_1 \ge 0,\\
			(g * \rho_\delta)(0,x') & \text{if } x_1 < 0.
		\end{dcases}
	\]
	For $x \in \RR^d$ with $x_1 \ge 0$, the evenness of $\rho$ gives
	\[
		|f_\delta(x) - f(x)| = |(g*\rho_\delta)(x) - g(x)| = \frac{1}{2} \abs{ \int_{B_\delta} \Delta^2 g_y(x) \rho_\delta(y)dy} \le \frac{1}{2} \sup_{|y| \le \delta} \nor{\Delta^2 g_y}{\oo}.
	\]
	Similarly, if $x_1 < 0$,
	\[
		|f_\delta(x) - f(x)| = |(g*\rho_\delta)(0,x')| = |(g*\rho_\delta)(0,x') - g(0,x')| \le \frac{1}{2} \sup_{|y| \le \delta} \nor{\Delta^2 g_y}{\oo},
	\]
	which yields
	\[
		\nor{f_\delta - f}{\oo} \le \frac{1}{2} \sup_{|y| \le \delta} \nor{\Delta^2 g_y}{\oo}.
	\]
	For $x \in \RR^d$, set
	\[
		f^1_\delta(x) := (g*\rho_\delta)(x) \quad \text{and} \quad f^2_\delta(x) := (g*\rho_\delta)(0,x');
	\]
	that is, $f_\delta = f^1_\delta$ in $U_+ \cup \Gamma$ and $f_\delta = f^2_\delta$ in $U_- \cup \Gamma$.
	
	Computations similar to those in the proof of Lemma \ref{L:KC11C} give, for $j = 1,2$ and some constant $C > 0$ independent of $\delta$,
	\[
		\nor{Df^j_\delta}{\oo} \le \nor{Dg}{\oo} \quad \text{and} \quad \nor{D^2 f^j_\delta}{\oo} \le \frac{C}{\delta^2} \sup_{|y| \le \delta} \nor{\Delta^2_y g}{\oo}.
	\]
	Therefore, by Proposition \ref{P:piecewisedc},
	\[
		\nor{f_\delta}{\dc} \le C\pars{ \frac{\nor{\Delta^2_y g}{\oo}}{\delta^2} + \nor{Dg}{\oo}}.
	\]
	Now, for $t \in [1,\oo)$, set $\delta := t^{-1/2}$. Then
	\[
		K(t) := K(t,f,\dc, C) \le \nor{f_\delta}{\dc} + t \nor{f - f_\delta}{\oo} \le C\pars{ t \sup_{|y| \le t^{-1/2}} \nor{\Delta^2_y g}{\oo} +  \nor{Dg}{\oo}},
	\]
	and therefore, for a constant $C = C_{p,\alpha} > 0$ that changes from line to line,
	\begin{align*}
		\int_1^\oo \pars{ \frac{K(t)}{t^\alpha}}^p \frac{dt}{t}
		&\le C \nor{Dg}{\oo} + C\int_1^\oo \pars{ \frac{t \sup_{|y| \le t^{-1/2}} \nor{\Delta^2_y g}{\oo}}{t^\alpha}}^p \frac{dt}{t}\\
		&= C \nor{Dg}{\oo} + \frac{C}{2} \int_0^1 \pars{ \frac{ \sup_{|y| \le s} \nor{\Delta^2_y g}{\oo}}{s^{2(1-\alpha)}}} \frac{ds}{s} \\
		&\le C \pars{ \nor{Dg}{\oo} + [g]_{B^{2(1-\alpha)}_{\oo,p} } }.
	\end{align*}
	The result follows from Lemma \ref{L:interp}\eqref{L:interpsum}.
\end{proof}

We close this section by considering the function given, for $a: S^{d-1} \to \RR$ and $x \in \RR^d$, by
\begin{equation}\label{front}
	f(x) = a\pars{ \frac{x}{|x|}}|x|.
\end{equation}
If $\alpha < 1/2$, or $\alpha = 1/2$ and $p = 1$, then, in general, $f$ fails to have the regularity specified by Proposition \ref{P:Hapgeneral}, even if $a$ is smooth. 

\begin{proposition}\label{P:Hapfront}
	Let $a: S^{d-1} \to \RR$, $0 < \alpha < 1$, and $1 \le p \le \oo$, and set
	\[
		g(x) := a \pars{ \frac{x}{|x|}}.
	\]
	Assume that, for some $\delta_0 \in (0,1/2)$,
	\[
		g \in B^{2(1-\alpha)}_{\oo,p} (\{x : 1 - \delta_0 < |x| < 1 + \delta_0 \}).
	\]
	Then the function $f$ defined in \eqref{front} belongs to $\mathscr H_{\alpha,p}$.
\end{proposition}

\begin{proof}
	For $\rho$ as in \eqref{mollifier} and $\delta > 0$, set $\rho_\delta(x) := \frac{1}{\delta^d} \rho\pars{ \frac{x}{\delta}}$, and, for $x \in S^{d-1}$, define
	\[
		a_\delta(x) := \int_{B_\delta(x)} a\pars{ \frac{y}{|y|}}\rho_\delta(x-y)\;dy
	\quad \text{and} \quad
		f_\delta(x) := a_\delta\pars{ \frac{x}{|x|}}|x|.
	\]
	Set also
	\[
		U := \left\{ x \in \RR^d: 1-\delta_0 < |x| < 1 + \delta_0 \right\}.
	\]
	Then, for all $R > 0$ and $\delta \in (0,\delta_0)$,
	\[
		\nor{f_\delta - f}{\oo,B_R} \le R \nor{a_\delta - a}{S^{d-1}} \le \frac{R}{2} \sup_{|y| \le \delta} \nor{\Delta^2_y f}{\oo,U}
	\]
	and, by Proposition \ref{P:dclevelset}, for some constant $C > 0$ independent of $R$, $a$, and $\delta$,
	\[
		\nor{H_\delta}{\dc,B_R} \le CR \nor{a_\delta}{C^{1,1}} \le \frac{CR}{\delta^2} \sup_{|y| \le \delta} \nor{\Delta^2_y g}{\oo,U}.
	\]
	The result then follows along similar lines as in Lemma \ref{L:KC11C}.
\end{proof}

\section{The space $\mathscr P_{\alpha,p}$}\label{S:P}

We next study, for $\alpha \in (0,1)$ and $p \in [1,\oo]$, the interpolation space
\[
	\mathscr P_{\alpha,p} = (C_0([0,T]), W^{1,1}_0([0,T])_{\alpha,p} \quad \text{for } \alpha \in (0,1), \; p \in [1,\oo].
\]
As in Section \ref{S:H}, we do not know of a simple characterization of this space, and the focus is on finding examples of embeddings. We also identify the parameters for which Brownian paths belong to 
$\mathscr P_{\alpha,p}$ with probability one. 

The hypotheses in Theorem \ref{T:mainweak} deal with sequences of paths converging uniformly while being bounded in $\mathscr P_{\alpha,\oo}$, and therefore, we place special emphasis in this section on identifying such paths. 

\subsection{H\"older and variation spaces}

We first prove a general criterion, using a particular method for measuring the variation of a continuous path. Given $W: [0,T] \to \RR^m$ and a partition
\[
	\mcl P := \left\{ 0 = t_0 < t_1 < \cdots < t_N = T \right\}
\]
of $[0,T]$, define
\[
	\osc(W,\mcl P) := \max_{i=1,2,\ldots,N} \osc(W, [t_i,t_{i+1}]), \quad \# P := N,
\]
and, for $\delta > 0$,
\begin{equation}\label{varcounter}
	N(\delta,W) := \inf\left\{ N \in \NN : \text{there exists $\mcl P$ such that } \osc(W,\mcl P) \le \delta \text{ and } \# P = N \right\}.
\end{equation}
The variation of $W$ can then be quantified in terms of the rate at which $N(\delta,W) \to \oo$ as $\delta \to 0$.

\begin{lemma}\label{L:variation}
	Let $W \in C_0([0,T])$, $\alpha \in (0,1)$, and $p \in [1,\oo]$. Assume that, for some sequence $(\delta_n)_{n=0}^\oo$ satisfying
	\[
		S := \pars{ \sum_{n=0}^\oo 2^{n\alpha p} \delta_n^p }^{1/p} < \oo,
	\]
	there exists $C > 0$ such that, for all $n \in \NN$,
	\[
		N(\delta_n,W) \le C 2^{n}.
	\]
	\begin{enumerate}[(a)]
	\item\label{L:variationcrit} Then $W \in \mathscr P_{\alpha,p}$ and $\nor{W}{\mathscr P_{\alpha,p}} \le (C+1)S$.
	
	\item\label{L:variationapprox} If $p = \oo$ and $\delta_n = 2^{-n\alpha}$ for all $n \in \NN$, then there exists a sequence of piecewise-linear paths $(W_n)_{n=1}^\oo: [0,T] \to \RR^m$ such that
	\[
		\lim_{n\to \oo} \nor{W_n - W}{\oo,[0,T]} = 0 \quad \text{and} \quad \sup_{n \in \NN} \nor{W_n}{\mathscr P_{\alpha,\oo}} < \oo.
	\]
	\end{enumerate}
\end{lemma}

\begin{proof}
	\eqref{L:variationcrit} Fix $n = 0,1,2,\ldots$, let
	\[
		\mcl P_n = \left\{ 0 = t_0^n < t_1^n < \cdots < t_{N_n}^n = T \right\}
	\]
	be a partition of $[0,T]$ such that $\osc(W,\mcl P_n) \le \delta_n$, and let $W_n$ be the piecewise linear interpolation of $W$ over $\mcl P_n$. Then
	\[
		\nor{W - W_n}{\oo,[0,T]} \le \delta_n
	\]
	and
	\[
		\nor{\dot{W_n}}{1,[0,T]} = \sum_{i=1}^{N_n} \abs{ W_n(t_i) - W_n(t_{i-1})} \le N_n \delta_n =  (\# \mcl P_n)\delta_n.
	\]
	Taking the infimum over all such partitions yields 
	\[
		K(2^{-n}, W, C_0, W^{1,1}_0) \le \delta_n + 2^{-n} N(\delta_n,W)\delta_n  \le (C+1) \delta_n,
	\]
	and the result follows from Lemma \ref{L:interp}\eqref{L:interpsum}.
	
	\eqref{L:variationapprox} Let $\mcl P_n$ and $(W_n)_{n=1}^\oo$ be as in part \eqref{L:variationcrit}, so that, in particular, $\# \mcl P_n \le C2^n$. For $k = 1,2,\ldots, N_n$, set $\Delta_k^n := t_k^n - t_{k-1}^n$.
	
	Fix $m > n$, and let $K_n^m \in \NN$ satisfy
	\[
		2^{\alpha(m-n)}  < K_n^m \le 2^{\alpha(m-n)} + 1.
	\]
	We define a refinement of $\mcl P_n$ by setting
	\[
		\mcl P_n^m := \mcl P_n \cup \bigcup_{k=1}^{N_n}\pars{ t_k^n + \frac{j}{K_n^m}}_{j=0}^{K_n^m}.
	\]
	Then
	\[
		\osc(W_n, \mcl P_n^m) = \frac{1}{K_n^m} \osc(W_n, \mcl P_n) \le \frac{2^{-n\alpha}}{2^{\alpha(m-n)}} = 2^{-\alpha m} \quad \text{and} \quad 
		\# \mcl P_n^m = N_n(K_n^m -1) \le 2^{\alpha(m-n)} \# \mcl P_n,
	\]
	and so
	\[
		N(2^{-\alpha m}, W_n) \le 2^{\alpha(m-n)} \# \mcl P_n \le C 2^{\alpha m + (1-\alpha)n} \le C 2^m.
	\]
	The conclusion follows from part \eqref{L:variationcrit}.
\end{proof}

As a corollary of Lemma \ref{L:variation}, we have the following.

\begin{proposition}\label{P:variationandHolder}
For any $p \in [1,\oo]$, the following inclusions are continuous:
\[
	C^{0,1/p}_0([0,T],\RR^m) \subset V_{p,0}([0,T],\RR^m) \subset \mathscr P_{1/p,\oo}.
\]
\end{proposition}

\begin{proof}
	The proof of the first inclusion is standard. For the second, let $W \in V_{p,0}$ and assume that
	\[
		\nor{W}{V_p} = 1.
	\]
	Fix $n \in \NN$ and define a partition $\mcl P := \{0 = t_0 < t_1 < \cdots < t_N := T \}$ by $t_0 = 0$ and
	\[
		t_k := \inf \{ t > t_{k-1} : |W_{t_{k-1}} - W_{t}| = 2^{-n/p} \} \wedge T.
	\]
	Then
	\[
		1 = \nor{W}{V_p}^p \ge \sum_{k=1}^N |W_{t_{k-1}} - W_{t_k}|^p = N2^{-n},
	\]
	and $\osc(W, \mcl P) \le 2^{-n/p}$, whence $N(2^{-n/p}, W) \le 2^n$. The result now follows from Lemma \ref{L:variation}\eqref{L:variationcrit}.
\end{proof}

Proposition \ref{P:variationandHolder} implies that, if $(W,W_n)_{n=1}^\oo$ is bounded in either $C^{0,\alpha}$ or $V_{1/\alpha}$, and, as $n \to \oo$, $W_n$ converges uniformly to $W$, then the hypotheses of Theorem \ref{T:mainweak} are satisfied for $\mathscr P_{\alpha,\oo}$. In either case, $W_n$ can be defined, for example, as the convolution of $W$ with a standard mollifier, or as a piecewise linear interpolation along a partition with vanishing mesh size as $n \to \oo$.

\subsection{Besov regularity}

As in the case of $\mathscr H_{\alpha,p}$, we can use Besov, or Besov-Lorentz, criteria for belonging to $\mathscr P_{\alpha,p}$. We shall also provide a necessary condition, to be used in the next sub-section.

\begin{proposition}\label{P:PapBesov}
	Let $\alpha \in (0,1)$ and $p \in [1,\oo]$. Then
	\[
		B^{\alpha}_1(L^{1/\alpha, p})_0([0,T],\RR)^m \subset \mathscr P_{\alpha,p} \subset B^\alpha_p(L^1)_0([0,T],\RR)^m.
	\]
\end{proposition}

\begin{proof}
	For ease of notation, we consider the case $m = 1$. Let $\mcl E$ be a continuous operator on the spaces
	\[
		\mcl E :C_0([0,T]) \to  C_\odd((-\oo,\oo)), \quad \mcl E: W^{1,1}_0([0,T]) \to W^{1,1}_\odd((-\oo,\oo)),
	\]
	such that, for any $W \in C_0([0,T])$, $\mcl E W = W$ on $[0,T]$ and $\mcl E$ has compact support. The operator $\mcl E$ can be constructed by first extending to $[-2T,2T]$ through odd reflections across the points $0$, $T$, and $-T$, and then multiplying with a smooth, even, nonnegative cutoff function $\eta$ with support in $[-2T,2T]$ such that $\eta \equiv 1$ on $[-T,T]$. If $\mcl R$ is the restriction map
	\[
		\mcl R : C_\odd((-\oo,\oo)) \to C_0([0,T]),
	\]
	then it is clear that $\mcl R \circ \mcl E$ is the identity operator on $C_0([0,T])$. Therefore, through the method of retracts (see \cite{BL}), it suffices to consider the interpolation space $(C_\odd(\RR), W^{1,1}_\odd(\RR))_{\alpha,p}$, which contains (see \cite{T})
	\begin{equation}\label{oddBesovs}
		(B^0_{\oo,1;\odd}(\RR), B^1_{1,1; \odd}(\RR))_{\alpha,p}.
	\end{equation}
	As was argued in Proposition \ref{P:1dbesov}, we may assume without loss of generality that the function $\phi$ defined in \eqref{pou} is even. This implies that, if $f:\RR \to \RR$ is odd, then $f * \psi$ and $f* \phi_k$ are odd for all $k \in \NN$, where $\psi$ and $\phi_k$ are the Schwartz functions defined in subsection \ref{SS:besov}. As a consequence, the method of retracts allows us to relate \eqref{oddBesovs} to the space
	\[
		(\ell^{0,1}(L^\oo_\odd), \ell^{1,1}(L^1_\odd) )_{\alpha,p}.
	\]
	Arguing again as in the proof of Proposition \ref{P:1dbesov}, this contains $\ell^{\alpha,1}(L^{1/\alpha,p}_\odd)$, and we conclude.
	
	For the second inclusion, we note that $\mathscr P_{\alpha,p} \subset (L^1, W^{1,1}_0)_{\alpha,p}$. A similar argument as in Lemma \ref{L:KC11C} gives universal constants $0 < c < C$ that are independent of $t > 0$ and $W \in C_0([0,T])$ such that
	\[
		c\sup_{|h| \le t} \nor{W(\cdot + h) - W}{L^1([0,T])} \le
		K(t,W,L^1_0,W^{1,1}_0) \le C \sup_{|h| \le t} \nor{W(\cdot + h) - W}{L^1([0,T])},
	\]
	which gives the claim in view of \eqref{equivalentbesov2} and Lemma \ref{L:interp}\eqref{L:interpsum}.
\end{proof}

\subsection{Brownian paths}

Assume that
\begin{equation}\label{Brownianmotion}
	W: [0,T] \times \Omega \to \RR \quad \text{is a standard Brownian motion over the probability space } (\Omega,\mbf F, \mbf P).
\end{equation}
The previous results imply that, with probability one, $W$ belongs to $\mathscr P_{\alpha,p}$ whenever $0 < \alpha < 1/2$ and $1 \le p \le \oo$, as can be seen from either the H\"older \cite{SV}, variation \cite{FV}, or Besov \cite{B,CKR,R} regularity of Brownian paths. As a consequence, many types of approximations can be constructed that satisfy the assumptions of Theorem \ref{T:mainweak}.

In the remainder of this section, we show that, in fact, Brownian paths belong to $\mathscr P_{1/2,\oo}$, and fail to belong to $\mathscr P_{1/2,p}$ if $p < \oo$. We also present two particular types of approximations that are bounded in $\mathscr P_{1/2,\oo}$. The first is a piecewise linear interpolation as in Lemma \ref{L:variation}\eqref{L:variationapprox}, and the second is a family of appropriately scaled random walks.

\begin{proposition}\label{P:Brownian}
	Let $W$ be a standard Brownian motion as in \eqref{Brownianmotion}.  
	\begin{enumerate}[(a)]
	\item If $\alpha < 1/2$ and $p \in [1,\oo]$, or if $\alpha = 1/2$ and $p = \oo$, then $W(\cdot,\omega) \in \mathscr P_{\alpha,p}$ for $\mbf P$-almost every $\omega \in \Omega$.
	\item If $\alpha = 1/2$ and $p < \oo$, or if $\alpha > 1/2$ and $p \in [1,\oo]$, then $W(\cdot,\omega) \notin \mathscr P_{\alpha,p}$ for $\mbf P$-almost every $\omega \in \Omega$.
	\end{enumerate}
\end{proposition}

\begin{proof}
	In view of Lemma \ref{L:interp}\eqref{L:interpmonotone}, it suffices to consider $\mathscr P_{1/2,\oo}$ and $\mathscr P_{1/2,p}$ with $p < \oo$ in order to prove, respectively, parts (a) and (b).

	(a) Fix $n \in \NN$, and define $t^n_0 = 0$ and
	\begin{equation}\label{BMstopping}
		t^n_{k+1} = \inf\left\{ t > t^n_k : \osc(W,[t^n_k,t]) = 2^{n/2} \right\} \wedge T \quad \text{for } k = 1,2,\ldots.
	\end{equation}
	Let $N_n \in \NN$ be the first integer for which $t^n_{N_n} = T$. 
	
	Observe that the $(t^n_k)_{k=0}^{N_n}$ are stopping times,
	\[
		\left\{
		\begin{split}
		&\mbf E (t^n_{k+1} - t^n_k) = 2^{-n} \quad \text{for} \quad k = 0,1,2,\ldots, N_n-2,\\
		&\mbf E (t^n_N - t^n_{N-1}) \le 2^{-n}, \quad \text{and}\\
		&\mbf E |t^n_{k+1} - t^n_k - 2^{-n}|^2 \le C 2^{-2n} \quad \text{for some $C > 0$ and all } k = 0,1,2,\ldots,N_n - 1.
		\end{split}
		\right.
	\]
	In view of Markov's inequality, there exists $C > 0$ such that, for any $\lambda > 1$ and for sufficiently large $n$ depending only on $\lambda$,
	\begin{align*}
		\mbf P \pars{ N_n > 2^n \lambda T }
		&= \mbf P \pars{ \sum_{k=0}^{ \floor{2^n \lambda T } - 1} (t^n_{k+1} - t^n_k) < T } \\
		&\le \mbf P \pars{ \sum_{k=0}^{ \floor{2^n \lambda T } - 1} \pars{ t^n_{k+1} - t^n_k - 2^{-n} } < -(\lambda -1)T + 2^{-n} } \le \frac{C}{2^n T},
	\end{align*}
	and, therefore, the Borel-Cantelli lemma yields $n^*: \Omega \to \NN$ such that $n^*(\omega) < \oo$ for $\mbf P$-almost every $\omega \in \Omega$ and $N_{n} \le 2^n \lambda T$ for all $n \ge n^*(\omega)$. As a consequence, 
	\[
		\mbf P \pars{ \sup_{n \in \NN} 2^{-n}N(2^{-n/2},W(\cdot,\omega)) < \oo } = 1,
	\]
	where $N(\delta,W)$ is defined as in \eqref{varcounter}, and so, by Lemma \ref{L:variation}\eqref{L:variationcrit}, for $\mbf P$-almost every $\omega \in \Omega$, $W(\cdot,\omega) \in \mathscr P_{1/2,\oo}$.
	
	(b) Let $p \in [1,\oo)$. Then, by Proposition \ref{P:PapBesov}, $\mathscr P_{1/2,p} \subset B^{1/2}_{1,p}([0,T])$. It is shown in \cite{R} that, with probability one, Brownian paths do not belong to $B^{1/2}_{1,p}$ if $p < \oo$, and the result follows.
\end{proof}

As a consequence of Theorem \ref{T:mainweak} and Proposition \ref{P:Brownian}, for any $H \in \mathscr H_{1/2,1}$ and $u_0 \in UC(\RR^d)$, there exists a well-defined probability measure $u: \Omega \to UC(\RR^d \times [0,T])$ that is understood to be the solution of 
\begin{equation}\label{E:Strat}
	du = H(Du) \circ dW \quad \text{in } \RR^d \times (0,T] \quad \text{and} \quad u(\cdot,0) = u_0 \quad \text{on } \RR^d.
\end{equation}
The use of the Stratonovich notation ``$\circ$'' is motivated by the fact that $u$ is obtained as limits of solutions of
\begin{equation}\label{E:BMapprox}
		u_{n,t} = H(Du_n) \dot W_n(t) \quad \text{in } \RR^d \times (0,T] \quad \text{and} \quad u(\cdot,0) = u_0 \quad \text{on } \RR^d
\end{equation}
for regularizations $(W_n)_{n \in \NN}$ of $W$ that satisfy the hypotheses of Theorem \ref{T:mainweak} for $\mathscr P_{1/2,\oo}$. One such approximation can be constructed along the same lines as in Lemma \ref{L:variation}\eqref{L:variationapprox}. 

\begin{proposition}\label{P:almostsure}
	Fix $H \in \mathscr H_{1/2,1}$ and $u_0 \in UC(\RR^d)$. For $n \in \NN$, let $\mcl P_n$ be the partition of stopping times defined in \eqref{BMstopping}, let $W_n$ be the piecewise linear interpolation of $W$ over $\mcl P_n$, and let $u_n$ and $u$ solve respectively \eqref{E:BMapprox} and \eqref{E:Strat}. Then, with probability one,
	\[
		\lim_{n \to \oo} \nor{u_n - u}{\oo,\RR^d \times [0,T]} = 0.
	\]
\end{proposition}

\begin{proof}
	The result follows from Proposition \ref{P:Brownian}, Theorem \ref{T:mainweak}, and the fact that, as in the proof of Lemma \ref{L:variation}\eqref{L:variationapprox}, we have, with probability one,
	\[
		\lim_{n \to \oo} \nor{W_n - W}{\oo,[0,T]} = 0 \quad \text{and} \quad \sup_{n \in \NN} \nor{W_n}{\mathscr P_{1/2,\oo}} < \oo.
	\]
\end{proof}

We now introduce a family of random walk approximations. Let $\pars{ X_k}_{k \in \NN}: \Omega \to \{-1,1\}$ be a collection of independent Rademacher random variables, define $\zeta: [0,\oo) \to \RR$ by 
\[
	\zeta(0) = 0 \quad \text{and} \quad \dot \zeta(t) = X_k \quad \text{for } t \in [k-1, k) \quad \text{for } k = 1,2,\ldots,
\]
and, for $n \in \NN$, set
\begin{equation}\label{scaledrw}
	W_n(t) = \frac{1}{n} \zeta(n^2 t) \quad \text{for } t \in [0,T].
\end{equation}\

Recall that a sequence of Borel probability measures $(\mu_n)_{n \in \NN}$ on a Polish space $X$ converges weakly, as $n \to \oo$, to a Borel probability measure $\mu$ if
\[
	\lim_{n \to \oo} \int_{X} f(x) \mu_n(dx) = \int_X f(x) \mu(dx) \quad \text{for all bounded and continuous } f: X \to \RR,
\]
and a sequence $(A_n)_{n \in \NN}$ of $X$-valued random variables, not necessarily defined over the same probability space, is said to converge in law, or in distribution, as $n \to \oo$, to another $X$-valued random variable $A$ if, as $n \to \oo$, the law of $A_n$ converges weakly to the law of $A$. It is a classical fact (see Billingsley \cite{Bill}) that, in the topology of $C([0,T])$, as $n \to \oo$, the path $W_n$ defined in \eqref{scaledrw} converges in law to a Brownian motion $W$ as in \eqref{Brownianmotion}.

\begin{proposition} \label{P:randomwalk}
	Assume that $H \in \mathscr H_{1/2,1}$, $u_0 \in UC(\RR^d)$, and $W$ is a standard Brownian motion. For $n \in \NN$, let $W_n$ be the scaled random walk as in \eqref{scaledrw} and let $u_n$ be the viscosity solution of \eqref{E:BMapprox}. Then, as $n \to \oo$, $u_n$ converges locally uniformly and in law to the solution of \eqref{E:Strat}.
\end{proposition}

In view of Proposition \ref{P:dcextension}, if $H \in \dc(\RR^d)$, then the result of Proposition \ref{P:randomwalk} is a consequence of the classical Mapping theorem. That is, if, for another Polish space $Y$, the map $\Phi: X \to Y$ is continuous, then $\Phi_\sharp$ is continuous with respect to weak convergence of measures \cite{Bill}. Here, for a Borel probability measure $\mu$ on $X$ and a Borel measurable map $\Phi: X \to Y$, $\Phi_\sharp \mu$ denotes the Borel probability measure on $Y$ given by
\[
	(\Phi_\sharp \mu)(A) = \mu(\Phi^{-1}(A)) \quad \text{for all Borel sets } A \subset Y.
\]
In order to prove Proposition \ref{P:randomwalk} for $H \in \mathscr H_{1/2,1}$, we need the following generalization of the Mapping theorem.

\begin{lemma}\label{L:mapping}
	For normed spaces $X \subset \tilde X$ and $Y$, a sequence of Borel probability measures $(\mu_n)_{n \in \NN}$ on $\tilde X$, and a map $\Phi: X \to Y$, assume that
	\begin{enumerate}[(a)]
	\item\label{L:mappingmeas} as $n \to \oo$, in the topology of $\tilde X$, $\mu_n$ converges weakly to another probability measure $\mu$ on $\tilde X$,
	\item\label{L:mappingtight} for all $\eps > 0$, there exists $C_\eps > 0$ such that
	\[
		\sup_{n \in \NN} \mu_n \pars{ \left\{ x \in X: \nor{x}{X} > C_\eps\right\} } \le \eps,
	\]
	and
	\item\label{L:mappingweakcts} if $(x_n)_{n \in \NN} \subset X$ and $x \in \tilde X$,
	\[
		\lim_{n \to \oo} \nor{x_n - x}{\tilde X} = 0, \quad \text{and} \quad \sup_{n \in \NN} \nor{x_n}{X} \le R,
	\]
	then $x \in X$, $\nor{x}{X} \le R$, and $\lim_{n \to \oo} \Phi(x_n) = \Phi(x)$.
	\end{enumerate}
	Then, as $n \to \oo$, in the topology of $Y$, $\Phi_\sharp \mu_n$ converges weakly to $\Phi_\sharp \mu$.
\end{lemma}

Observe that property \eqref{L:mappingtight}, which says that the sequence $(\mu_n)_{n \in \NN}$, when restricted to $X$, is tight, does not follow directly from property \eqref{L:mappingmeas}.

\begin{proof}
	By the Portmanteau theorem (see \cite{Bill}), it suffices to show that, for all closed sets $F \subset Y$,
	\[
		\limsup_{n \to \oo} \mu_n\pars{ \Phi^{-1}(F)} \le \mu\pars{ \Phi^{-1}(F)}.
	\]
	Fix $\eps > 0$, let $C_\eps > 0$ be given as in part \eqref{L:mappingtight}, and define
	\[
		A_\eps := \left\{ x \in X: \nor{x}{X} \le C_\eps \right\}.
	\]
	Property \eqref{L:mappingweakcts} implies that $A_\eps \cap \Phi^{-1}(F)$ is a closed subset of $A_\eps$ in the subspace topology inherited from $\tilde X$, and, therefore, there exists a closed set $H \subset \tilde X$ such that
	\[
		A_\eps \cap \Phi^{-1}(F) = A_\eps \cap H.
	\]
	The Portmanteau theorem and property \eqref{L:mappingmeas} give
	\[
		\limsup_{n \to \oo} \mu_n(H) \le \mu(H),
	\]
	and, therefore,
	\begin{align*}
		\limsup_{n \to \oo} \mu_n \pars{ \Phi^{-1}(F)} 
		&\le \liminf_{n \to \oo} \mu_n \pars{ A_\eps \cap H} + \eps\\
		&\le \limsup_{n \to \oo} \mu_n(H) + \eps
		\le \mu(H) + \eps\\
		&\le \mu(A_\eps \cap \Phi^{-1}(F)) + 2\eps
		\le \mu(\Phi^{-1}(F)) + 2\eps.
	\end{align*}
	The proof is finished because $\eps$ was arbitrary.
\end{proof}

The next result gives a uniform estimate on the probability tails of $(W_n)_{n \in \NN}$ in $\mathscr P_{1/2,\oo}$.

\begin{lemma}\label{L:boundedinlaw}
	Let $(W_n)_{n \in \NN}$ be defined by \eqref{scaledrw}. Then, for all $\eps > 0$, there exists $R = R_\eps > 0$ such that
	\[
		\sup_{n \in \NN} \mbf P \pars{ \nor{W_n}{\mathscr P_{1/2,\oo}} > RT} < \eps.
	\]
\end{lemma}

The proof of Lemma \ref{L:boundedinlaw} relies on the construction of useful stopping times, as in the proof of Proposition \ref{P:Brownian}. For $M \in \NN$, define
\begin{equation}\label{rwtimes}
	\tau^M_0 = 0 \quad \text{and} \quad \tau^M_k := \inf \left\{ t \in \NN, \; t > \tau^M_{k-1} : |\zeta(t) - \zeta(\tau^M_{k-1})| = M \right\}\quad \text{for } k = 1,2,\ldots,
\end{equation}
and, for $t > 0$, define
\[
	K^M(t) := \inf\left\{ k \in \NN : \tau^M_k \ge t \right\}.
\]

\begin{lemma}\label{L:stoppingtimes}
	Fix $M \in \NN$ and $t > 0$. Then
	\[
		\mbf P \pars{ K^M(t) > \frac{ \lambda t}{ M^2} } \le \frac{2\lambda (M^2-1)t}{3((\lambda-1)t - M^2)^2} \quad \text{for all } \lambda > 1 + \frac{M^2}{t}.
	\]
\end{lemma}

\begin{proof}
In view of the strong Markov property, the random variables $( \tau^M_k - \tau^M_{k-1})_{k \in \NN}$ are independent and identically distributed, and the optimal stopping theorem gives 
\[
	\mbf E \left[ \tau^M_k - \tau^M_{k-1} \right] = M^{2} \quad \text{and} \quad \Var \left[ \tau^M_k - \tau^M_{k-1} \right] = \frac{2}{3}M^2(M^2 -1) \quad \text{for all } M \in \NN, k \in \NN.
\]
The result then follows from Markov's inequality and the chain of inequalities
\begin{align*}
	\mbf P ( K^M(t) > \lambda tM^{-2})
	&\le \mbf P \pars{ \tau^M_{\floor{\lambda tM^{-2}}} < t}\\
	&= \mbf P \pars{ \sum_{k=1}^{\floor{\lambda tM^{-2}}} (\tau^M_k - \tau^M_{k-1}) < t } \\
	&\le \mbf P \pars{ \sum_{k=1}^{\floor{\lambda tM^{-2}}} (\tau^M_k - \tau^M_{k-1} - M^{2} ) < -(\lambda-1)t +M^2}\\
	&\le \frac{2\lambda (M^2-1)t}{3((\lambda-1)t - M^2)^2}.
\end{align*}
\end{proof}

\begin{proof}[Proof of Lemma \ref{L:boundedinlaw}]
	Fix $n \in \NN$. For two different ranges of $m \in \NN$, we estimate the quantity
	\[
		2^{-m} N(2^{-m/2}, W_n),
	\]
	where $N(\delta,W)$ is defined as in \eqref{varcounter}. Throughout, we assume, without loss of generality, that $T > 1$.
	
	{\it Case 1.} Assume that $m \in \NN$ and $2^{m/2} \le \frac{n}{2}$. Set 
	\[
		M := \left\lfloor n 2^{-m/2} \right\rfloor,
	\]
	and note that
	\[
		2 \vee \frac{1}{2} n 2^{-m/2} \le M \le n2^{-m/2}.
	\]
	For $k = 0,1,2,\ldots$, define
	\[
		\tilde \tau_k := \frac{1}{n^2} \tau^M_k,
	\]
	where $\tau^M_k$ is defined as in \eqref{rwtimes}. Then, for all $k = 1,2,\ldots,$
	\[
		\osc(W_n, [\tilde \tau_{k-1}, \tilde \tau_k]) \le \frac{1}{n} \abs{ \zeta(\tau^M_k) - \zeta(\tau^M_{k-1})} = \frac{M}{n} \le 2^{-m/2}.
	\]
	Therefore, 
	\[
		N(2^{-m/2},W_n) \le \inf\{ k \in \NN: \tilde \tau_k \ge T \} = K^M(n^2T).
	\]
	Fix $R \ge 8$ and set
	\[
		\lambda := \frac{R M^2 2^m}{n^2}.
	\]
	Then
	\[
		\lambda > R\pars{ 1 - 2^{m/2}n^{-2}}^2 \ge 2 \ge 1 + \frac{2^{-m}}{T} \ge 1 + \frac{M^2}{n^2 T},
	\]
	and so, in view of Lemma \ref{L:stoppingtimes}, for some universal constant $C > 0$,
	\begin{align*}
		\mbf P \pars{ N(2^{-m/2}, W_n) > RT 2^{m}}
		&\le \mbf P \pars{ K^M(n^2 T) > RT 2^{m}} \\
		&= \mbf P \pars{ K^M(n^2 T) > \lambda n^2 T M^{-2}} \\
		&\le \frac{2 \lambda (M^2 - 1)n^2 T}{3((\lambda - 1)n^2 T - M^2)^2}\\
		&= \frac{2RM^2(M^2-1)2^m T}{((RM^2 2^m - n^2)T - M^2)^2} \le \frac{C}{RT} 2^{-m}.
	\end{align*}
	We conclude that
	\begin{equation}\label{smallm}
		\mbf P \pars{ \sup \left\{ 2^{-m} N(2^{-m/2},W_n) : m \in \NN, \; 2^{m/2} \le \frac{n}{2} \right\} > RT} \le \frac{C}{RT}.
	\end{equation}
	
	{\it Case 2.} Assume now that $m \in \NN$ and $2^{m/2} \ge \frac{n}{2}$, and define
	\[
		\mcl P := \pars{ \frac{k}{n 2^{m/2}}}_{k = 0,1,2,\ldots}.
	\]
	Then
	\[
		\osc\pars{ W_n, \mcl P} \le \frac{ \nor{\dot W_n}{\oo}}{n 2^{m/2}} \le 2^{-m/2},
	\]
	and so
	\begin{equation}\label{largem}
		2^{-m} N(2^{-m/2}, W_n) \le T n 2^{-m/2} \le 2T \quad \text{whenever } 2^{m/2} \ge \frac{n}{2}.
	\end{equation}
	
	Combining \eqref{smallm} and \eqref{largem} gives, for some universal constant $C > 0$, for all $n \in \NN$, and for all $R \ge 8 \vee \frac{C}{\eps T}$,
	\[
		\mbf P \pars{ \sup_{m \in \NN} 2^{-m} N(2^{-m/2},W_n) > RT} \le \frac{C}{RT} < \eps.
	\]
	The result then follows from Lemma \ref{L:variation}\eqref{L:variationcrit}.
\end{proof}

\begin{proof}[Proof of Proposition \ref{P:randomwalk}]
	Set $\Phi = S(H,\cdot,u_0)$, where $S$ is the solution operator from \eqref{solutionmap}, and define
	\[
		X = \mathscr P_{1/2,\oo}, \quad \tilde X = C([0,T]), \quad \text{and} \quad Y \in C(\RR^d \times [0,T]).
	\]
	For $n \in \NN$, let $\mu_n$ denote the law of $W_n$ on $\tilde X$, and let $\mu$ denote the Wiener measure on $\tilde X$, that is, the law of $W$. Then the hypotheses of Lemma \ref{L:mapping} are satisfied. Indeed, property \eqref{L:mappingmeas} is a restatement of Donsker's invariance principle \cite{Bill}, property \eqref{L:mappingtight} follows from Lemma \ref{L:boundedinlaw}, and, in view of Lemma \ref{L:oostability} and Theorem \ref{T:mainweak}, $\Phi$ is a well-defined map on $\mathscr P_{1/2,\oo}$ that satisfies property \eqref{L:mappingweakcts}.
	
	As a consequence of Lemma \ref{L:mapping}, as $n \to \oo$, $\Phi_\sharp \mu_n$ converges weakly to $\Phi_\sharp \mu$, which is the desired conclusion.
\end{proof}

\section{A remark on the sharpness of the main results}\label{S:sharp}

We present an example to show that the assumptions of the main theorems, in particular, Theorem \ref{T:mainweak}, cannot be relaxed.

The focus is on describing the behavior of solutions of the approximate problem given, for $\beta \in (0,1)$ and a sequence $(W_n)_{n \in \NN} \subset W^{1,1}([0,T])$, by
\begin{equation}\label{E:counterexample}
	u_{n,t} = |Du|^\beta \dot W_n(t) \quad \text{in } \RR^d \times (0,T] \quad \text{and} \quad u(x,0) = |x| \quad \text{on } \RR^d,
\end{equation}
that is, the Hamiltonian of interest is
\begin{equation}\label{counterexample}
	H_\beta(p) = |p|^\beta.
\end{equation}

\begin{theorem}\label{T:counterexample}
	Let $\beta \in (0,1)$. Then $H_\beta \in \mathscr H_{\alpha,1}$ if and only if $\alpha + \beta > 1$. Moreover, if $\alpha \in (0,1)$ and $u_0(x) = |x|$ for $x \in \RR^d$, then the following hold:
	\begin{enumerate}[(a)]
	\item If $\alpha + \beta < 1$, then there exists a sequence of paths $(W_n)_{n \in \NN} \subset W^{1,1}_0([0,T],\RR)$ such that
	\begin{equation}\label{badpaths}
		\lim_{n \to \oo} W_n = 0 \text{ uniformly,} \quad \sup_{n \in \NN} \nor{W_n}{\mathscr P_{\alpha,\oo}} < \oo,
	\end{equation}
	and, as $n \to \oo$, the solution $u_n$ of \eqref{counterexample} converges to $+\oo$.
	
	\item If $\alpha + \beta = 1$, then for any $c_0 > 0$, there exists a sequence of paths $(W_n)_{n \in \NN} \subset W^{1,1}([0,T])$ satisfying \eqref{badpaths} such that, if $u_n$ solves \eqref{E:counterexample}, then
	\[
		\lim_{n \to \oo} \sup_{(x,t) \in \RR^d \times [0,T]} |u_n(x,t) - (|x| \vee c_0 t^\alpha)| = 0.
	\]
	\end{enumerate}
\end{theorem}

We first recall the Hopf formula for solutions of Hamilton-Jacobi equations with convex initial data. For a proof, see \cite{LR}.

\begin{lemma}\label{L:Hopf}
Let $H \in C(\RR^d)$, and assume that $u_0 \in UC(\RR^d)$ is convex. Then the unique viscosity solution of
\[
	u_t = H(Du) \quad \text{in } \RR^d \times (0,\oo) \quad \text{and} \quad u(\cdot,0) = u_0 \quad \text{in } \RR^d
\]
is given by
\begin{equation}\label{Hopf}
	u(x,t) = \sup_{p \in \RR^d} \left\{ p \cdot x - u_0^*(p) + tH(p) \right\} \quad \text{for } (x,t) \in \RR^d \times [0,\oo).
\end{equation}
\end{lemma}

We use this formula to make some computations and bounds for a simple driving path.

\begin{lemma}\label{L:onetooth}
	For $t \in [0,2]$, set $W(t) = 1-|t-1|$, let $\beta \in (0,1)$, and, for $a > 0$, let $u$ solve
	\[
		u_t = \frac{1}{\beta} |Du|^\beta \dot W(t) \quad \text{in } \RR^d \times [0,2] \quad \text{and} \quad u(x,0) = |x| \vee a \quad \text{for } x \in \RR^d.
	\]
	Then, for all $(x,t) \in \RR^d \times [0,2]$,
	\[
		|x| \vee a \le u(x,t) \le \pars{ |x| \vee a } + \frac{1}{\beta} \quad \text{and} \quad u(x,2) = |x| \vee \left[ a + \frac{1-\beta}{\beta} a^{-\beta/(1-\beta)} \right].
	\]
	\end{lemma}

\begin{proof}
	We first compute
	\[
		u^*(p,0) =
		\begin{dcases}
			a \pars{ |p| - 1} & \text{if } |p| \le 1, \text{ and} \\
			+\oo & \text{if } |p| > 1.
		\end{dcases}
	\]
	The contractive property of the equation implies that $\nor{Du}{\oo} \le 1$, and, therefore, for all $t \ge 0$, $u^*(p,t)$ is finite if and only if $|p| \le 1$. 
	
	By Lemma \ref{L:Hopf}, for $|p| \le 1$ and $t \in [0,1]$,
	\[
		u^*(p,t) = \pars{ u^*(\cdot,0) - \frac{t}{\beta} |\cdot|^\beta}^{**}.
	\]
	This implies that, on $[0,1]$, $t \mapsto u^*(\cdot,t)$ is nonincreasing, and therefore, $t \mapsto u(\cdot,t)$ is nondecreasing. Moreover, $u^*(\cdot,1)$ is equal to the lower convex envelope of the radial function 
	\[
		\oline{B_1} \ni p \mapsto a\pars{|p| - 1} - \frac{1}{\beta}|p|^\beta =: \phi(|p|).
	\]
	The function $\phi: [0,1] \to \RR$ is convex and attains a global minimum at $a^{-1/(1-\beta)}$, where it achieves the value
	\[
		- \frac{1-\beta}{\beta} a^{- \beta/(1-\beta)} - a,
	\]
	and, therefore,
	\[
		u^*(p,1)
		:= 
		\begin{dcases}
			- \frac{1-\beta}{\beta} a^{- \beta/(1-\beta)} - a & \text{if } |p| \le a^{-1/(1-\beta)}, \text{ and}\\
			a \pars{|p| - 1} - \frac{1}{\beta}|p|^\beta & \text{if } a^{-1/(1-\beta)} < |p| \le 1,
		\end{dcases}
	\]
	For $|p| \le 1$ and $t \in [0,1]$, we then have the bounds
	\[
		a\pars{ |p| - 1} - \frac{1}{\beta} \le u^*(p,t) \le a \pars{ |p| - 1}.
	\]
	Taking the Legendre transform gives the desired bounds for $u$ on $\RR^d \times [0,1]$.
	
	Now, $u^*(\cdot,2)$ is the lower convex envelope of the radial function
	\[
		p \mapsto u^*(p,1) + \frac{1}{\beta}|p|^\beta := \psi(|p|),
	\]
	where
	\[
		\psi(r) := 
		\begin{dcases}
			- \frac{1-\beta}{\beta} a^{- \beta/(1-\beta)} - a + \frac{1}{\beta}r^\beta & \text{if } 0 \le r \le a^{-1/(1-\beta)}, \text{ and}\\
			a \pars{r - 1} & \text{if } a^{-1/(1-\beta)} < r \le 1.
		\end{dcases}
	\]
	The function $\psi$ is concave and increasing on $[0,1]$, and, therefore, the lower semicontinuous envelope of $p \mapsto \psi(|p|)$ is
	\[
		u^*(p,2) = \pars{ a + \frac{1-\beta}{\beta} a^{-\beta/(1-\beta)} }(|p| - 1).
	\]
	Upon taking the Legendre transform, this gives the desired formula for $u(\cdot,2)$, and, in view of the fact that $t \mapsto u^*(\cdot,t)$ is nondecreasing on $[1,2]$, the claimed bounds for $u$ on $\RR^d \times [1,2]$.
\end{proof}

The following lemma characterizes the long term behavior of a certain recursively-defined sequence that arises in the coming proofs.

\begin{lemma}\label{L:sequence}
Assume that $0 < \beta < 1$, $a_1 \ge \beta^{-(1-\beta)}$, and, for $k \in \NN$,
\begin{equation}\label{recursive}
	a_{k+1} = a_k + \frac{1-\beta}{\beta} a_k^{- \beta/(1-\beta)}.
\end{equation}
Then there exists a constant $C = C_\beta > 0$ such that, for all $k \in \NN$,
\[
	\beta^{-(1-\beta)} k^{1-\beta} \le a_k \le \left[ a_1^{- 1/(1-\beta)} + \beta^{-1} (k-1) + C \log k\right]^{1-\beta}.
\]
\end{lemma}

\begin{proof}
For $x \ge 0$, define
\[
	f(x) := (1+x)^{1-\beta}
\]
and, for $k \in \NN$, set $b_k := a_k^{1/(1-\beta)}$. The concavity of $f$ implies that, for all $x \ge 0$, $f(x) \le 1 + (1-\beta)x$, and so, for all $k \in \NN$,
\[
	a_{k+1} = a_k \pars{ 1 + (1-\beta) \frac{b_k^{-1}}{\beta}}
	\ge a_k \pars{ 1 + \beta^{-1} b_k^{-1}}^{1-\beta}
	= \pars{ b_k + \beta^{-1}}^{1-\beta}.
\]
Therefore, $b_{k+1} \ge b_k + \beta^{-1}$, and the first inequality follows from an induction argument.

Now, for $k \in \NN$, set
\[
	\lambda_k := k \left[ \pars{ 1 + \frac{1-\beta}{k}}^{1/(1-\beta)} - 1 \right],
\]
and observe that, for some $C = C_\beta > 0$,
\[
	\lambda_k \le 1 + \frac{C}{k}.
\]
The concavity of $f$ implies that
\[
	1 + (1-\beta)x \le f(\lambda_k x) \quad \text{for all } 0 \le x \le \frac{1}{k}.
\]
For any $k \in \NN$, the first inequality gives
\[
	0 < \frac{a_k^{-1/(1-\beta)}}{\beta} \le \frac{1}{k},
\]
and, therefore,
\[
	a_{k+1} = a_k \pars{ 1 + (1-\beta) \frac{a_k^{-1/(1-\beta)}}{\beta} } \le a_k \pars{ 1 + \lambda_k \beta^{-1} b_k}^{1-\beta} = \pars{ b_k + \lambda_k \beta^{-1} }^{1-\beta}.
\]
As a consequence, $b_{k+1} \le b_k + \lambda_k \beta^{-1}$, and so
\[
	b_k \le b_1 + \beta^{-1} \sum_{j=2}^k \lambda_j.
\]
The result now follows, since, for some $C = C_\beta > 0$ that changes from line to line,
\[
	b_k \le b_1 + \beta^{-1}(k-1) + C \sum_{j=2}^k \frac{1}{j} \le b_1 + \beta^{-1} (k-1) + C \log k.
\]
\end{proof}

As a consequence of Lemmas \ref{L:onetooth} and \ref{L:sequence}, we have the following result.

\begin{lemma}\label{L:solutionformula}
	Set $W(0) = 0$ and, for $k = 0,1,2,\ldots,$
	\begin{equation}\label{teeth}
		\dot W(t) :=
		\begin{dcases}
			+1 & \text{if } t \in (2k,2k+1), \text{ and}\\
			-1 & \text{if } t \in (2k+1,2k+2),
		\end{dcases}
	\end{equation}
	and let $u$ be the solution of
	\begin{equation}\label{E:unscaledteeth}
		u_t =  \frac{1}{\beta}|Du|^\beta \dot W(t) \quad \text{in } \RR^d \times (0,\oo) \quad \text{and} \quad u(x,0) = |x| \quad \text{on } \RR^d.
	\end{equation}
	Then, for some $C = C_\beta > 0$, and for all $k \in 0,1,2,\ldots$ and $(x,t) \in \RR^d \times [2k,2k+2]$,
	\[
		|x| \vee \pars{ \beta^{-(1-\beta)} k^{1-\beta}} 
		\le u(x,t)
		\le |x| \vee \left[ \beta^{-1/(1-\beta)} + \beta^{-1} (k-1) + C \log k \right]^{1-\beta} + \frac{1}{\beta}.
	\]
\end{lemma}

\begin{proof}
	We first consider $t \in [0,2]$. As in the proof of Lemma \ref{L:onetooth}, $u^*(p,t)$ is finite if and only if $|p| \le 1$.

	For all $t \in [0,1]$,
	\[
		u^*(\cdot,t) = \pars{ -  t\beta^{-1} |\cdot|^\beta}^{**},
	\]
	that is, $u^*(\cdot,t)$ is the lower convex envelope of the radial function $p \mapsto -t \beta^{-1} |p|^\beta$ on $\oline{B_1}$, which yields, for $|p| \le 1$ and $t \in [0,1]$,
	\[
		u^*(p,t) = -t \beta^{-1}.
	\]
	As a consequence, for $(x,t) \in \RR^d \times [0,1]$,
	\[
		u(x,t) = |x| + \frac{t}{\beta},
	\]
	and the claimed bounds are immediate.

	Next, for $t \in [1,2]$, we have $u^*(\cdot,t) = ( -\beta^{-1} + \beta^{-1}(t-1) |\cdot|^\beta)^{**}$, and so, for $|p| \le 1$ and $t \in [1,2]$,
	\[
		u^*(p,t) =  \frac{t-2}{\beta} + \frac{t-1}{\beta} (|p| - 1),
	\]
	which gives, for $(x,t) \in \RR^d \times [1,2]$,
	\[
		u(x,t) = |x| \vee \pars{ \frac{t-1}{\beta} } + \frac{2-t}{\beta}.
	\]
	This yields the desired bounds for $t \in [1,2]$. 
	
	Finally, we have $u(x,2) = |x| \vee \beta^{-1}$, and so, because $\beta^{-1} > \beta^{-(1-\beta)}$, the rest of the proof follows from an inductive argument and Lemmas \ref{L:onetooth} and \ref{L:sequence}.
\end{proof}

\begin{proof}[Proof of Theorem \ref{T:counterexample}]
(a) We first show that, if $\alpha + \beta > 1$, then $H_\beta \in \mathscr H_{\alpha,1}$.

For $\delta > 0$, define $H_{\beta,\delta}(p) := |p|^\beta \vee \delta^\beta$. Then $\nor{H_\beta - H_{\beta,\delta}}{\oo} = \delta^\beta$, and
\[
	H_{\beta,\delta} = H_1 - H_2, \quad \text{where} \quad H_1(p) = \beta\delta^{\beta-1}\pars{ |p| - \delta}_+ \quad \text{and $H_2$ is convex}.
\]
Therefore, for some constant $C > 0$ and any $L > 0$,
\[
	\nor{H_{\beta,\delta}}{\dc(B_L)} \le \nor{H_1}{\oo,B_L} + \nor{H_2}{\oo,B_L} \le CL\delta^{\beta -1}.
\]
Now, for $n \in \NN$, set $\delta = 2^{-n}$. Then
\[
	K(2^n,H_\beta, \dc(B_L),C(B_L)) \le \nor{H_{\beta,\delta}}{\dc(B_L)} + 2^n \nor{H_\beta - H_{\beta,\delta}}{\oo,B_L} \le C L2^{n(1-\beta)}.
\]
Then $H_\beta \in \mathscr H_{\alpha,1}$ as a consequence of Lemma \ref{L:interp}\eqref{L:interpsum}, the fact that $\alpha + \beta > 1$, and
\[
	\sum_{n=1}^\oo 2^{-n\alpha}K(2^n, H_\beta, \dc(B_L),C(B_L)) \le CL \sum_{n=1}^\oo 2^{-n(\alpha + \beta - 1)} < \oo.
\]

Conversely, if $H_\beta \in \mathscr H_{\alpha,1}$, then the conclusions of Theorem \ref{T:mainweak} hold, and, therefore, $\alpha + \beta > 1$ in view of the examples in parts (b) and (c) below.

(b) Let $W$ be as in \eqref{teeth} and let $u$ be the solution of \eqref{E:unscaledteeth}. For $n \in \NN$, define
\[
	u_n(x,t) := 2^{-n\alpha} u(2^{n\alpha}x, 2^n t) \quad \text{and} \quad W_n(t) := \beta^{-1} 2^{-n\alpha} W(2^n t).
\]
Then $u_n$ solves \eqref{E:counterexample}. Note that $\lim_{n \to \oo} W_n = 0$ uniformly on $[0,T]$, and, by Proposition \ref{P:variationandHolder}, for some constant $C = C_{\alpha} > 0$,
\[
	\sup_{n \in \NN} \nor{W_n}{\mathscr P_{\alpha,\oo}} \le C \sup_{n \in \NN} \nor{W_n}{C^\alpha} < \oo.
\]
Now let $(x,t) \in \RR^d \times [0,T]$ and let $k \in \NN$ be such that $k \le 2^{n-1}t < k+1$. If $n$ is sufficiently large, then $k \ge 2$, and so, by Lemma \ref{L:solutionformula}, for some constant $c = c_{\alpha,\beta} > 0$,
\[
	u^n(x,t) = 2^{-n\alpha} u(2^{n\alpha}x, 2^{n} t) \ge \beta^{-(1-\beta)} k^{1-\beta} 2^{-n\alpha} \ge c2^{n(1-\alpha-\beta)} t^{1-\beta} \xrightarrow{n \to \oo} +\oo,
\]
as desired.

(c) As in part (b), let $W$ be as in \eqref{teeth} and let $u$ be the solution of \eqref{E:unscaledteeth} with $\beta = 1-\alpha$. For $n \in \NN$, $(x,t) \in \RR^d \times [0,T]$, and a constant $\lambda > 0$ to be determined, define
\[
	u_n(x,t) := \lambda 2^{-n \alpha} u(\lambda^{-1} 2^{n\alpha} x, 2^n t) \quad \text{and} \quad W_n(t) = \frac{\lambda}{1-\alpha} 2^{-n\alpha} W(2^n t).
\]
Then $u_n$ solves \eqref{E:counterexample} with $\beta = 1-\alpha$, and, once again,
\[
	\lim_{n \to \oo} \nor{W_n}{\oo,[0,T]} = 0 \quad \text{and} \quad \sup_{n \in \NN} \nor{W_n}{\mathscr P_{\alpha,\oo}} < \oo.
\]
For $(x,t) \in \RR^d \times [0,T]$ and $n \in \NN$, let $k = 0,1,2,\ldots$ be such that $k \le 2^{n-1} t \le k+1$. Then, by Lemma \ref{L:solutionformula},
\begin{equation}\label{unlower}
	\begin{split}
	u_n(x,t) &\ge |x| \vee \left[ (1-\alpha)^{-\alpha} \lambda 2^{-n\alpha} k^\alpha \right]\\
	&\ge |x| \vee \left[ (1-\alpha)^{-\alpha} \lambda 2^{-n\alpha} (2^{n-1} t - 1)_+^\alpha \right]\\
	&= |x| \vee \left[ (1-\alpha)^{-\alpha} \lambda \pars{ \frac{t}{2} - 2^{-n} }_+^\alpha \right],
	\end{split}
\end{equation}
and, for some $C = C_\alpha > 0$, 
\begin{equation}\label{unupper}
	\begin{split}
	u_n(x,t) &\le |x| \vee \left[ \lambda 2^{-n\alpha} \pars{ (1-\alpha)^{-1/\alpha} + \frac{k-1}{1-\alpha} + C \log k}_+^\alpha \right] + \frac{\lambda}{2^{n\alpha}(1-\alpha)}\\
	&\le |x| \vee \left[ \lambda \pars{ 2^{-n} (1-\alpha)^{-1/\alpha} + \frac{t}{2(1-\alpha)} + C(\log 2)n2^{-n} + C2^{-n} \log t }_+^\alpha \right] + \frac{\lambda}{2^{n\alpha}(1-\alpha)}.
	\end{split}
\end{equation}
The bounds \eqref{unlower} and \eqref{unupper} together give
\[
	\lim_{n \to \oo} \sup_{(x,t) \in \RR^d \times [0,T]} \abs{ u_n(x,t) - |x| \vee \pars{ \frac{\lambda}{2^\alpha(1-\alpha)^\alpha} t^\alpha} } = 0,
\]
and the proof is finished upon setting $\lambda := 2^\alpha(1-\alpha)^\alpha c_0$.
\end{proof}

\appendix

\section{The extension property for $\dc$-Hamiltonians}\label{S:LS}

We give a stability estimate for pathwise Hamilton-Jacobi equations in the regime where the Hamiltonian belongs to $\dc$. The arguments are essentially the same as those in \cite{LS2, Snotes}.

\begin{proposition}\label{P:dcextension}
Assume that, for each $i = 1,2,\ldots,m$, $H^i \in \dc_\loc$. Then, for all $L > 0$, there exists $C = C_L > 0$ such that, if $\nor{Du_0}{\oo} \le L$, $W_1,W_2 \in C([0,T],\RR^m)$, and, for $j = 1,2$, $u_j$ is the solution of
\begin{equation}\label{E:apppathwise}
	du_j = \sum_{i=1}^m H^i(Du_j) \cdot dW_j^i \quad \text{in } \RR^d \times (0,T] \quad \text{and} \quad u^j(\cdot,0) = u_0 \quad \text{in } \RR^d,
\end{equation}
then
\[
	\sup_{(x,t) \in \RR^d \times [0,T]} \abs{ u_1(x,t) - u_2(x,t)} \le C \sum_{i=1}^m \nor{H^i}{\dc(B_L)} \max_{t \in [0,T]} \abs{W^i_1(s) - W^i_2(s)}.
\]
\end{proposition}

In the next result, we recall that, in order for the equation \eqref{E:apppathwise} to be well-posed for all continuous paths and initial data, the condition $H \in \dc_\loc$ is necessary and sufficient. The proof, which we do not give here, can be found in \cite{Snotes} as Proposition 7.2, and resembles the discussion in Section \ref{S:sharp} of the present paper.

\begin{proposition}\label{P:dcnecessary}
	Assume that $H \in C(\RR^d) \backslash \dc_\loc$. Then there exists $u_0 \in BUC(\RR^d)$ and a sequence of paths $(W_n)_{n=1}^\oo \subset W^{1,1}([0,T],\RR^m)$ such that, as $n \to \oo$, $W_n$ converges uniformly to $0$, and, if $u_n$ is the classical viscosity solution of \eqref{E:apppathwise} corresponding to the path $W_n$, then $u_n$ has no uniform limit as $n \to \oo$.
\end{proposition}

Proposition \ref{P:dcextension} is proved using Theorem 7.2 of \cite{Snotes}, which we restate without proof here as Lemma \ref{L:convexestimate}. The proof of Lemma \ref{L:convexestimate} takes advantage of cancellations that arise in iteratively applying the Hopf-Lax formula for solutions of time homogenous Hamilton-Jacobi equations with convex Hamiltonians. Similar manipulations of solution operators appear in recent works involving the regularity, domain-of-dependence, and long-time-behavior properties of solutions of pathwise Hamilton-Jacobi equations; see, for instance, \cite{GG,GGLS,LSreg}.

For some $M \in \NN$, let $(H_j)_{j=1}^M: \RR^d \to \RR$ be convex functions satisfying
\begin{equation}\label{A:minH}
	\min_{\RR^d} H_j = 0 \quad \text{for each } j = 1,2,\ldots, M.
\end{equation}
Let $(W_j)_{j=1}^M \subset C([0,T])$ and $u_0 \in UC(\RR^d)$, and let $u$ be the pathwise viscosity solution of
\begin{equation}\label{E:convexestimate}
	du = \sum_{j=1}^M H_j(Du) \cdot dW_j \quad \text{in } \RR^d \times (0,T] \quad \text{and} \quad u(\cdot,0) = u_0 \quad \text{in } \RR^d.
\end{equation}
Given a Hamiltonian $H: \RR^d \to \RR$, the maps
\[
	\pars{S_H(t)}_{t \ge 0} : (B)UC(\RR^d) \to (B)UC(\RR^d)
\]
denote the solution operators for the equation
\begin{equation}\label{E:simpleHJ}
	u_t = H(Du) \quad \text{in } \RR^d \times (0,\oo),
\end{equation}
that is, $u(x,t) = S_H(t)\phi(x)$ solves \eqref{E:simpleHJ} with $u(\cdot,0) = \phi$.

Finally, for a real-valued continuous path $\zeta$ and $t \ge 0$, we define
\[
	\zeta^*(t) := \max_{0 \le s \le t} \zeta(s) \quad \text{and} \quad \zeta_*(t) = -\min_{0 \le s \le t} \zeta(s).
\]

\begin{lemma}\label{L:convexestimate}
	Let $u$ be the solution of \eqref{E:convexestimate}. Then, for all $t \ge 0$,
	\[
		\prod_{j=1}^M S_{H_j}\pars{ W_{j,*}(t)} u_0(x) \le u(x,t) \le 
		\prod_{j=1}^M S_{H_j} \pars{ W_j^*(t)} u_0(x).
	\]
\end{lemma}

\begin{proof}[Proof of Proposition \ref{P:dcextension}]
	It suffices to prove the result when $W^1$ and $W^2$ are smooth. The general result follows by a density argument.
	
	The comparison principle and the spatial homogeneity of $H$ give
	\[
		u^1(x,t) - u^2(y,t) \le \Phi(x-y,t),
	\]
	where $\Phi$ solves
	\[
		d\Phi = \sum_{i=1}^m H^i(D\Phi) \cdot d(W^{1,i} - W^{2,i}) \quad \text{in } \RR^d \times (0,T] \quad \text{and} \quad \Phi(x,0) = L|x| \quad \text{in } \RR^d.
	\]
	For $i = 1,2,\ldots, m$, we can write
	\[
		H^i = H^i_1 - H^i_2
	\]
	where $H^i_1$ and $H^i_2$ are convex on $B_L$. For each $i = 1,2,\ldots, m$ and $j = 1,2$, let $v^i_j$ belong to the sub-differential of $H^i_j$ at $p = 0$. Since $H^i_j$ is convex, we have
	\begin{equation}\label{convexLip}
		|v^i_j| \le \frac{1}{L} \sup_{|p| \le L} H^i_j(p) \quad \text{for all } i=1,2,\ldots,m \text{ and } j = 1,2.
	\end{equation}
	
	Define
	\[
		\tilde \Phi(x,t) = \Phi\pars{ x - \sum_{i=1}^m ( v^i_1 - v^i_2) (W^{1,i}(t) - W^{2,i}(t)) } - \sum_{i=1}^m H^i(0) (W^{1,i}(t) - W^{2,i}(t))
	\]
	and, for $i = 1,2,\ldots, m$, $j= 1,2$, and $p \in \RR^d$,
	\[
		\tilde H^i_j(p) := H^i_j(p) - H^i_j(0) - v^i_j \cdot p.
	\]
	Then $\tilde \Phi$ solves
	\[
		\begin{dcases}
		d\tilde \Phi = \sum_{i=1}^m \tilde H^i_1(D \tilde \Phi) \cdot d(W^{1,i} - W^{2,i}) + \sum_{i=1}^m \tilde H^i_2(D \tilde \Phi) \cdot d(W^{2,i} - W^{1,i}) & \text{in } \RR^d \times (0,T],\\
		\tilde\Phi(x,0) = L|x| & \text{in } \RR^d,
		\end{dcases}
	\]
	For each $i = 1,2,\ldots,m$ and $j = 1,2$, $H^i_j$ is convex and satisfies \eqref{A:minH}. Lemma \ref{L:convexestimate} gives
	\[
		\tilde \Phi(x,t)
		\le
		\prod_{i=1}^m S_{\tilde H^i_1}\pars{ (W^{1,i} - W^{2,i})^*(t)} \prod_{j=1}^m S_{\tilde H^i_2} \pars{ (W^{2,j} - W^{1,j})^*(t)} (L |\cdot|)(x).
	\]
	Given $H$ convex satisfying \eqref{A:minH} and $\tau \ge 0$, we estimate, using \eqref{Hopf},
	\[
		S_H(\tau)(L |\cdot|)(x) = \sup_{|p| \le L} \pars{ p \cdot x + \tau H(p)} \le L|x| + \tau \sup_{|p| \le L} H(p).
	\]
	Applying the estimate iteratively, using the comparison principle and the fact that the solution operators all commute with constants, we find that
	\[
		\tilde \Phi(x,t) \le L|x| + \sum_{i=1}^m \sum_{j \in \{1,2\}} \sup_{|p| \le L} |\tilde H^i_j(p)| \nor{W^{1,i} - W^{2,i}}{\oo,[0,t]}.
	\]
	Going back to the definition of $\tilde H$ and using \eqref{convexLip}, we conclude that, for some $C = C_L > 0$,
	\begin{align*}
		u^1(x,t) - u^2(x,t) &\le \Phi(0,t) \\
		&= \tilde \Phi\pars{ x + \sum_{i=1}^m (v^i_1 - v^i_2) (W^{1,i}(t) - W^{i,2}(t))} + \sum_{i=i}^m\pars{ H^i_1(0) - H^i_2(0)} \pars{ W^{1,i}(t) - W^{2,i}(t)}\\
		&\le C \sum_{i=1}^m \pars{ \nor{H^i_1}{\oo,B_L} + \nor{H^i_2}{\oo,B_L} } \nor{W^{i,1}(t) - W^{i,2}(t)}{\oo,[0,T]}.
	\end{align*}
	Taking the infimum over all convex $H^i_1,H^i_2$ satisfying $H^i = H^i_1 - H^i_2$ and applying a symmetric argument for the difference $u^2 - u^1$ gives the result.
\end{proof}

\section{The space of $\dc$-functions}\label{S:dc}

We present various sufficient criteria for a function to belong to $\dc$, that is, the space of functions that are equal to a difference of convex functions. 

The study of $\dc$-functions goes back to Aleksandrov \cite{A1,A2}, Landis \cite{L}, and Hartman \cite{H}. Other analytic properties have been investigated by, among others, Prudnikov \cite{P1,P2} and Zalgaller \cite{Z}. Such functions also play an important role in the study of nonconvex optimization; see, for instance, the survey of Hiriart-Urruty \cite{H-U}.

\subsection{The definition and some basic properties}
We recall, for convenience, the definition of $\dc$-functions, as well as the norm with which we equip the space.

\begin{definition}\label{D:appdc}
Let $U \subset \RR^d$ be an open domain and let $f: U \to \RR$. Then $f \in \dc(U)$ if there exist convex functions $f_1$ and $f_2$ on $U$ such that
\[
	f = f_1 - f_2.
\]
If $U$ is bounded, $\dc(U)$ is equipped with the norm
\[
	\nor{f}{\dc(U)} := \inf\left\{ \nor{f_1}{\oo,U} + \nor{f_2}{\oo,U} : f = f_1 - f_2, \; f_1,f_2 \text{ convex}  \right\}.
\]
A function $f$ is said to belong to $\dc_\loc(U)$ if $f \in \dc(V)$ for all bounded $V \subset U$, or equivalently,
\[
	\nor{f}{\dc_\loc} := \sum_{n=1}^\oo \max(2^{-n}, \nor{f}{\dc(B_n)} ) < \oo.
\]
When $U = \RR^d$, we write $\dc = \dc(\RR^d)$ and $\dc_\loc = \dc_\loc(\RR^d)$.
\end{definition}

The first result, which we state without proof, follows from elementary properties of convex functions.

\begin{lemma}\label{L:dcLip}
If $R > 0$ and $f \in \dc(B_{R})$, then $\nor{Df}{\oo,B_{R/2}} \le 2 R^{-1} \nor{f}{\dc,B_R}$. Moreover, $f$ is twice differentiable almost everywhere in $B_R$, that is, for almost every $x \in B_R$,
\[
	\lim_{h \to 0} \frac{ f(x + h) - f(x) - Df(x) \cdot h - \frac{1}{2} h \cdot D^2 f(x) h}{|h|^2} = 0.
\]
\end{lemma}

The next lemma is used throughout the section in order to prove various criteria for belonging to $\dc$.
\begin{lemma}\label{L:dcchar}
	Let $U \subset \RR^d$ be open and convex. Then $f\in \dc(U)$ if and only if there exists $g \in \dc(U)$ such that
	\begin{equation}\label{dcdom}
		D^2 f \le D^2 g.
	\end{equation}
	Moreover, if $U$ is bounded, then
	\[
		\nor{f}{\dc(U)} \le \nor{g}{\dc(U)} + \nor{g}{\oo,U} + \nor{f}{\oo,U}.
	\]
\end{lemma}

\begin{proof}
	If $f \in \dc(U)$, then \eqref{dcdom} clearly holds with $g = f$. Conversely, suppose that there exists $g \in \dc(U)$ such that \eqref{dcdom} holds, and let $g_1,g_2$ be convex functions on $U$ such that $g = g_1 - g_2$. Then, since $g_1$ and $g_2 + g - f$ are convex,
	\[
		f = g_1 - (g_2 + g-f) \in \dc(U).
	\]
	
	If $U$ is bounded, it follows that
	\[
		\nor{f}{\dc(U)} \le \nor{g_1}{\oo,U} + \nor{g_2 + g - f}{\oo,U} \le \nor{g_1}{\oo,U} + \nor{g_2}{\oo,U} + \nor{g}{\oo,U} + \nor{f}{\oo,U}. 
	\]
	Taking the infimum over convex $g_1,g_2$ on $U$ such that $g = g_1 - g_2$ yields the claim.
\end{proof}

It follows from Hartman \cite{H} that belonging to $\dc$ is a local property. For completeness, we present the proof here.

\begin{lemma}\label{L:dclocal}
Suppose that, for all $x \in U$, there exists $\delta > 0$ such that $f \in \dc(B_\delta(x)\cap U)$. Then $f \in \dc_\loc(U)$.
\end{lemma}

\begin{proof}
	It clearly suffices to consider $U$ compact. Then there exist $x_1,x_2,\ldots, x_n \in U$ and $\delta_1,\delta_2,\ldots, \delta_n > 0$ such that $U \subset \bigcup_{i=1}^n B_{\delta_i}(x_i)$, and, for $i = 1,2,\ldots, n$, there exist convex functions $f_i$ on $B_{2\delta_i}(x_i)$ such that $f + f_i$ is convex. Lemma \ref{L:dcLip} yields that, for each $i = 1,2,\ldots,n$, $f_i$ is uniformly Lipschitz on $B_{\delta_i}(x_i)$, and, therefore, there exists a convex function $\tilde f_i$ on $U$ such that $f_i = \tilde f_i$ on $B_{\delta_i}(x_i)$. The result now follows from the fact that $f + \tilde f_1 + \tilde f_2 + \cdots + \tilde f_n$ is convex on $U$.
\end{proof}

\subsection{Regularity criteria}

Lemma \ref{L:dcLip} suggests that possessing two derivatives, in an appropriate sense, can be a sufficient criterion for belonging to $\dc$. The simplest example of such a result is stated next.

\begin{proposition}\label{P:C11dc}
Semi-convex and semi-concave functions belong to $\dc$. In particular, $C^{1,1} \subset \dc$, and, if $R > 0$ and $f \in C^{1,1}(B_R)$, then
	\[
		\nor{f}{\dc(B_R)} \le  \nor{D^2f}{\oo,B_R}R^2 + \nor{f}{\oo,B_R}.
	\]
\end{proposition}

\begin{proof}
	If $f$ is semi-concave, then, for some constant $C > 0$, $D^2 f \le C \Id$ in the sense of distributions, and, therefore, \eqref{dcdom} holds with $g(x) = \frac{C}{2}|x|^2$. The argument for semi-convex functions is identical.

	If $f \in C^{1,1}(B_R)$, then the constant $C$ above may be taken to be $\nor{D^2f}{\oo,B_R}$, and, hence, Lemma \ref{L:dcchar} gives the bound on $\nor{f}{\dc(B_R)}$.
\end{proof}

Among other things, Proposition \ref{P:C11dc} allows for a useful way to localize throughout the paper.

\begin{lemma}\label{L:localizedc}
	Assume that $f \in \dc_\loc$ and let $\phi \in C^2_c(\RR^d)$. Then $\tilde f := f \cdot \phi \in \dc$.
\end{lemma}

\begin{proof}
	Since $f \in \dc_\loc$, there exist convex functions $f_1$ and $f_2$ on $\RR^d$ such that $f = f_1 - f_2$ on $K := \supp \phi$, and, thus, $f\phi = f_1\phi - f_2 \phi$ on $\RR^d$. For $j = 1,2$, we have, in the distributional sense on $\RR^d$,
	\[
		D^2( f_j \phi) = f_j D^2 \phi + 2 Df_j \otimes D\phi + \phi D^2 f_j \ge - \pars{ \nor{f_j}{\oo,K} \nor{D^2 \phi}{\oo,K} + 2 \nor{Df_j}{\oo,K} \nor{D\phi}{\oo,K}} \Id,
	\]
	and, hence, in view of Lemma \ref{L:dcLip}, $f_1 \phi$ and $f_2 \phi$ are semi-convex. The result now follows from Proposition \ref{P:C11dc}.
\end{proof}

When $d = 1$, the condition that $D^2f$ has a one-sided bound can be substantially weakened, and, in fact, there is an exact characterization.
\begin{proposition}\label{P:1ddc}
	Let $a < b$. Then $f \in \dc([a,b])$ if and only if $f' \in BV([a,b])$. In particular, $W^{2,1}([a,b]) \subset \dc([a,b])$, and, for some constant $C = C(b-a) > 0$,
	\[
		\nor{f}{\dc([a,b])} \le C \nor{f}{W^{2,1}([a,b])}.
	\]
\end{proposition}

\begin{proof}
	Fix $f \in \dc([a,b])$ and write $f = f_1 - f_2$ with $f_1,f_2:[a,b] \to \RR$ convex. Then $f_1''$ and $f_2''$ are positive measures on $[a,b]$, so that $f''$ is a finite signed measure on $[a,b]$. It follows that $f' \in BV$.
	
	Conversely, suppose that $f' \in BV([a,b])$. Then $f''$ is a signed measure $\mu = \mu_+ - \mu_-$, where $\mu_\pm$ are positive measures on $[a,b]$. Since $d = 1$, there exist convex functions $f_+$ and $f_-$ on $[a,b]$ such that, in the distributional sense, $f_\pm'' = \mu_\pm$. It follows that
	\[
		f'' = \pars{ f_+ - f_-}'',
	\]
	and therefore, for some $\alpha, \beta \in \RR$,
	\[
		f(x) = f_+(x) - f_-(x) + \alpha x + \beta.
	\]
	We conclude that $f \in \dc([a,b])$.
	
	Now assume that $f \in W^{2,1}([a,b])$. Then, for any $x_0,x_1,x \in [a,b]$,
	\[
		f(x) = f(x_0) + f(x_1)(x - x_0) + \int_{x_0}^x \int_{x_1}^t (f'')_+(s)dsdt - \int_{x_0}^x \int_{x_1}^t (f'')_-(s)dsdt,
	\]
	with the last two terms defining convex functions. Note that, since both $f$ and $f'$ are absolutely continuous, we can choose $x_0$ and $x_1$ such that
	\[
		f(x_0) = \frac{1}{b-a} \int_a^b f(t)dt \quad \text{and} \quad f(x_1) = \frac{1}{b-a} \int_a^b f'(t)dt.
	\]
	The estimate now follows.
\end{proof}

We note that one direction of the equivalence in Proposition \ref{P:1ddc} holds for all $d \ge 1$, namely, for an open set $U \subset \RR^d$, $f \in \dc(U)$ always implies $Df \in BV(U)$. However, the converse is false in general. Indeed, $\dc$ functions are twice-differentiable almost everywhere, while $BV$ functions may fail to have directional derivatives on a nontrivial set.

On the other hand, it is still an interesting question whether a condition on $D^2 f$ that is weaker than belonging to $L^\oo$ guarantees $f \in \dc$. As was mentioned earlier in the paper, one possible condition reads as follows:
\begin{equation}\label{conjecture}
	\text{for some $p = p_d \ge 1$,} \quad W_\loc^{2,p} \subset \dc_\loc \quad \text{for all } p > p_d.
\end{equation}
However, it turns out that \eqref{conjecture} is false as soon as $d > 1$. To see this, observe that, if $\Gamma \subset \RR^d$ is a line and $f \in \dc(\RR^d)$, then $f|_{\Gamma}$ clearly belongs to $\dc(\Gamma)$. On the other hand, the trace function
\begin{align*}
	T: W^{2,p}(\RR^d) &\to W^{2-1/p,p}(\Gamma) \times W^{1-1/p,p}(\Gamma) \\
	u &\mapsto \pars{ u |_{\Gamma}, \frac{\del u}{\del_\nu}\Big|_{\Gamma} },
\end{align*}
where $\nu \subset S^{d-1}$ is a fixed normal vector to $\Gamma$, is surjective, and general functions in $W^{2-1/p,p}(\Gamma)$ need not belong to $\dc(\Gamma)$, which can be checked by example with Proposition \ref{P:1ddc}. One such example, which was communicated to us by Terence Tao \cite{Texample}, is as follows.

\begin{proposition}\label{P:Texample}
	Assume $d > 1$ and $1 \le p < \oo$. Then there exists $f \in W^{2,p}_\loc(\RR^d) \backslash \dc_\loc$.
\end{proposition}

\begin{proof}
	Throughout the proof, for $x \in \RR^d$, we write $x = (\tilde x, x_d)$ for $\tilde x \in \RR^{d-1}$ and $x_d \in \RR$, that is, $\tilde x = (x_1,x_2,\ldots, x_{d-1})$.
	
	Let $\phi \in C^\oo(\RR^d)$ be such that
	\[
		\supp \phi \subset \left[ - \frac{1}{2}, \frac{1}{2} \right] \quad \text{and} \quad D^2 \phi \ge \Id \quad \text{in } \left[ - \frac{1}{4}, \frac{1}{4} \right].
	\]
	Let $0 < \eps < \frac{d-1}{p}$ be fixed and, for any $N \in \NN$, define
	\[
		\psi_N(x) := N^{-2+\eps} \sum_{j=1}^N \phi(N\tilde x, Nx_d - j) \quad \text{for } x \in \RR^d.
	\]
	We then have, for $k = 0,1,2$,
	\begin{equation}\label{phiW2p}
		\nor{D^k \psi_N}{L^p} = N^{\eps - 2 + k - \frac{d-1}{p}} \nor{D^k \phi}{L^p}.
	\end{equation}
	Now define
	\[
		f(x) := \sum_{m=0}^\oo \psi_{2^m}(x_1 - 2(1-2^{-m}), x_2, \ldots, x_d).
	\]
	Observe that $f$ is compactly supported, the summands $\psi_{2^m}(x_1 + 2(1-2^{-m}), x_2, \ldots, x_d)$ have disjoint supports, and, by  \eqref{phiW2p}, $f \in W^{2,p}$. 
	
	Assume now that there exists convex functions $f_1,f_2: \RR^d \to \RR$ such that $f = f_1 - f_2$. We then have $D^2 f_1 \ge D^2 f$. Fix $m \in \NN$ and define the convex function $\chi_m: \RR \to \RR$ by
	\[
		\chi_m(t) = f_1( 2(1-2^{-m}), 0,\ldots,0,t).
	\]
	Observe that
	\[
		\chi_m''(t) \ge \del_{x_d}^2 f( 2(1-2^{-m}), 0,\ldots,0,t) = \del_{x_d}^2 \psi_{2^m}(0, 0,\ldots,0,t),
	\]
	and so
	\[
		\chi_m'' \ge 2^{m\eps} \quad \text{in } 2^{-m} j + \left[ -2^{-m-2}, 2^{-m-2} \right], \; j = 1,2,\ldots, 2^m.
	\]
	Because $\chi_m$ is convex, we conclude that, for some universal constant $c > 0$,
	\[
		\del_{x_d} f_1( 2(1-2^{-m}), 0,\ldots,0,1) = \chi_m'(1) = \chi_m'(0) + \int_0^{1} \rho''(s)ds \ge \del_{x_d} f_1( 2(1-2^{-m}), 0,\ldots,0,0) + c2^{m\eps}.
	\]
	We conclude that $Df_1$ has unbounded oscillation in $[-2,2]^d$, which contradicts the fact that $f_1$ is convex.
\end{proof}

Proposition \ref{P:Texample} makes it clear that \eqref{conjecture} is false unless stronger assumptions are placed on $f$. The next few lemmata are partial results to that effect.

\begin{proposition}\label{P:radialfdc}
	Suppose that $p > d$ and $f \in W_\loc^{2,p}$ is radial. Then $f \in \dc_\loc$, and, for all $R > 0$ and some constant $C = C_{d,p} > 0$,
	\begin{equation}\label{radialestimate}
		\nor{f}{\dc(B_R)} \le \nor{f}{\oo,B_R} + C \nor{D^2 f}{L^p(B_R)}.
	\end{equation}
\end{proposition}

\begin{proof}
We write $f(x) = \phi(|x|)$ for some $\phi: [0,\oo) \to \RR$. Then, setting $r = |x|$,
\[
	D^2 f(x) = \phi''(r) \hat x \otimes \hat x + \frac{\phi'(r)}{r} \pars{ \Id - \hat x \otimes \hat x},
\]
and, therefore, for all $R > 0$ and some constant $C = C_p > 0$,
\[
	\int_0^R |\phi''(r)|^p r^{d-1}dr + \int_0^R |\phi'(r)|^p r^{d-1-p}dr \le C \nor{D^2 f}{L^p(B_R)}^p.
\]
Note in particular that, for any $0 \le a< b$,
\[
	\int_a^b |\phi''(r)|dr  \le \pars{ \int_a^b |\phi''(r)|^p r^{d-1}dr }^{1/p} \pars{ \int_a^b r^{- \frac{d-1}{p-1}}dr}^{1 - 1/p},
\]
and therefore, since $p > d$, $\phi'' \in L^1_\loc([0,\oo))$ and, for all $R > 0$ and some constant $C = C_{d,p} > 0$,
\begin{equation}\label{gineq}
	\int_0^R |\phi''(r)|dr \le C\nor{D^2 f}{L^p(B_R)}.
\end{equation}
Let $\psi: [0,\oo) \to [0,\oo)$ be given by
\[
	\psi(r) = \int_0^r \int_0^s \phi''(t)_+\;dt\;ds.
\]
Then
\[
	\psi''(r) = (\phi'')_+(r) \ge \phi''(r) \ge 0,
\]
and, since $p > d$, $Df(x) = \phi'(r)\hat x$ is continuous, which implies that $\phi'(0) = 0$, and, hence,
\[
	\psi'(r) = \int_0^r (\phi'')_+(s)ds \ge \int_0^r \phi''(s)ds = \phi'(r) \ge 0.
\]
For $x \in \RR^d$, let $g(x) = \psi(|x|)$. Then
\[
	D^2 g(x) = \psi''(r) \hat x \otimes \hat x + \frac{\psi'(r)}{r} \pars{ \Id - \hat x \otimes \hat x}.
\]
It is immediate that $g$ is convex and $D^2 f \le D^2 g$ on $B_R$, so, in view of Lemma \ref{L:dcchar}, $f \in \dc_\loc$. The estimate \eqref{radialestimate} is then a consequence of \eqref{gineq}.
\end{proof}

The next conclusion is an immediate corollary of Lemma \ref{L:dcchar} and Proposition \ref{P:radialfdc}.

\begin{proposition}\label{P:radialfdcdom}
	Assume that $p > d$, $g \in W^{2,p}$ is radial, and $D^2 f \le D^2 g$ in the sense of distributions. Then $f \in \dc$.
\end{proposition}

We next explore conditions in which $D^2 f$ itself is dominated by a radial function, or a superposition of radial functions.

\begin{proposition}\label{P:D2fradial}
	Suppose there exists $\phi: [0,\oo) \to \RR$ such that $\phi_+ \in L^1_\loc$, $\{r \mapsto r\phi_+(r)\} \in L^\oo_\loc$, and
	\[
		D^2f(x) \le \phi(|x|) \Id.
	\]
	Then $f \in \dc$.
\end{proposition}

\begin{proof}
	Define $\psi: [0,\oo) \to [0,\oo)$ by 
	\[
		\psi(0) = 0 \quad \text{and} \quad
		\psi'(r) = \max_{0 \le s \le r} s \phi_+(s) + \int_0^r \phi_+(s)ds.
	\]
	Then $\psi$ is convex and increasing, whence $g(x) = \psi(|x|)$ is convex. The result follows from Lemma \ref{L:dcchar} and the fact that $D^2 f \le D^2 g$.
\end{proof}

The final result of this subsection involves the Riesz potential, which, for $0 < s < d$ and $f \in \mathscr S(\RR^d)$, is the map $(-\Delta)^{-s} : \mathscr S \to \mathscr S$ given by
\[
	(-\Delta)^{-s} f(x) := c_{s,d} \int_{\RR^d} \frac{f(y)}{|x-y|^{d - s}} dy,
\]
where, with $\Gamma$ denoting the Gamma function,
\begin{equation}\label{cRiesz}
	c_{s,d} := \pi^{d/2} 2^s \frac{\Gamma(s/2)}{\Gamma((d-s)/2)}.
\end{equation}
The operator $(-\Delta)^{-s}$ is the inverse of the fractional Laplacian $(-\Delta)^s$. Moreover, the definition of $(-\Delta)^{-s}$ extends by duality to $\mathscr S'$. For more details, see Stein \cite{Stein}.

For $p \ge 1$, we introuduce the space
\[
	\mcl M_p := \left\{ \mu \in \mathscr S': \mu \text{ is a signed measure on $\RR^d$ and } \int_{\RR^d} |x|^p |\mu|(dx) < \oo \right\}.
\]

\begin{proposition}
	If $d \ge 2$ and $s > \frac{d+1}{2}$, then
	\[
		(-\Delta)^{-s} \mcl M_{2s-d} \subset \dc.
	\]
\end{proposition}

\begin{proof}
	Fix $f \in (-\Delta)^{-s} \mcl M_{2s-d}$, set $r := d - 2s + 2$, and note that $r < 1$ and $2-r = 2s-d$. 
	
	Let $\beta \in C^\oo(\RR^d \otimes \RR^d)$ be such that, for some positive constants $(C_k)_{k \in \NN}$ and for all $X \in \RR^d \otimes \RR^d$,
	\[
		|X| \le \beta(X) \le C_0(1 + |X|) \quad \text{and} \quad \nor{ D^k \beta}{\oo} \le C_k \quad \text{for all } k \in \NN.
	\]
	Define
	\[
		\mu := (-\Delta)^{\frac{d-r}{2}} \beta(D^2 f) \in \mcl M_{2s-d},
	\]
	denote by $\mu_+$ and $\mu_-$ the nonnegative measures in the Hahn decomposition $\mu = \mu_+ - \mu_-$ of $\mu$, set
	\[
		\nu := \frac{1}{1-r} \mu_+ - \mu_-,
	\]
	and define
	\[
		g(x) :=  \frac{ c_{\frac{d-r}{2}, d} }{2-r} \int_{\RR^d} |x-y|^{2-r}\;\nu(dy),
	\]
	where $c_{\frac{d-r}{2}, d}$ is defined as in \eqref{cRiesz}. Then $g \in \dc$, and
	\begin{align*}
		D^2 g(x) &= c_{\frac{d-r}{2}, d} \int_{\RR^d} |x-y|^{-r} \pars{ \Id - r (\widehat{x-y}) \otimes (\widehat {x-y}) }\nu(dy)\\
		&\ge c_{\frac{d-r}{2}, d} \int_{\RR^d} |x-y|^{-r} \pars{ (1-r) \nu_+(dy) - \nu_-(dy) }\Id \\
		&= \left[ (-\Delta)^{- \frac{d-r}{2}} \mu \right](x) \cdot \Id\\
		&= \beta(D^2 f) \cdot \Id \ge D^2f(x).
	\end{align*}
	The result now follows from Lemma \ref{L:dcchar}.
\end{proof}

\subsection{Structural criteria}\label{SS:structure}

Except for the case $d=1$, we are not aware of a simple characterization of $\dc$ functions in terms of the regularity or structure of their gradients. Nevertheless, as is demonstrated by the results that follow, the Lipschitz assumption for the gradient in Proposition \ref{P:C11dc} can be relaxed in various ways, even when $d > 1$. 

Throughout the subsection, we use the $\max$ and $\min$ operations on functions, and therefore the following lemma is useful.

\begin{lemma}\label{L:invariant}
If $f,g \in \dc(U)$, then so are $\min\{f,g\}$ and $\max\{f,g\}$, and
\[
	\nor{ \min\{f,g\}}{\dc(U)} \le 2\pars{ \nor{f}{\dc(U)} + \nor{g}{\dc(U)}} \quad \text{and} \quad
	\nor{\max\{f,g\}}{\dc(U)} \le 2\pars{\nor{f}{\dc(U)} + \nor{g}{\dc(U)}}.
\]
\end{lemma}

\begin{proof}
	Let $f = f_1 - f_2$ and $g = g_1 - g_2$ with $f_j, g_j$ convex for $j = 1,2$. Then
	\[
		\max\{f,g\} + f_2 + g_2 = \max\{ f + f_2 + g_2, g + f_2 + g_2\} = \max\{f_1 + g_2, f_2 + g_1\}
	\]
	is convex as the maximum of convex functions, so that
	\[
		\max\{f,g\} = \pars{ \max\{f,g\} + f_2 + g_2} - f_2 - g_2 \in \dc.
	\]
	The argument for $\min\{f,g\}$ is similar.
\end{proof}

We next prove a result on extending $\dc$ functions past convex sub-domains.

\begin{lemma}\label{L:dcextension}
	Assume that $K \subset \RR^d$ is convex and compact, $f_1,f_2: K \to \RR$ are convex and Lipschitz with the common Lipschitz constant $L > 0$, and $U \supset K$ is bounded and open, and set $f = f_1 - f_2$ on $K$. Then there exists $\tilde f \in \dc(U)$ such that $f = \tilde f$ on $K$, and
	\begin{equation}\label{dcextendbound}
		\nor{\tilde f}{\dc,U} \le \nor{f_1}{\oo,K} + \nor{f_2}{\oo,K} + 2L \dist(K,\del U).
	\end{equation}
	Moreover, given $g \in C^{0,1}(U)$ such that $g \le f$ in $K$, the extension $\tilde f$ can be chosen such that $\tilde f \ge g$ in $U$ and
	\begin{equation}\label{dcextendbound2}
		\nor{\tilde f}{\dc,U} \le \nor{f_1}{\oo,K} + \nor{f_2}{\oo,K} + (2L + \nor{Dg}{\oo})\dist(K,\del U).
	\end{equation}
\end{lemma}

\begin{proof}
For $j = 1,2$ and $x \in U$, define
\[
	\tilde f_j(x) := \sup\left\{ p \cdot x + a : |p| \le L, \; p \cdot y + a \le f(y) \text{ for } y \in K\right\}.
\]
It is immediate that $\tilde f_j$ is convex, $\nor{D \tilde f_j}{\oo,U} \le L$, and $\tilde f_j \le f_j$ on $K$. Moreover, since $\nor{Df_j}{\oo,K} \le L$, we have $f_j = \tilde f_j$ on $K$, and
\[
	\nor{\tilde f_j}{\oo,U} \le \nor{f_j}{\oo,K} + L \dist(K,\del U).
\]
Let $\tilde f := \tilde f_1 - \tilde f_2$. Then $\tilde f \in \dc(U)$, $\tilde f = f$ on $K$, and \eqref{dcextendbound} holds.

For $x \in U$, set
\[
	\rho(x) := \dist(x,K).
\]
Then $\rho = 0$ on $K$, and, as a consequence of the convexity of $K$, $\rho$ is convex. For the second part of the lemma, we may then take the function $\tilde f + \nor{Dg}{\oo}\rho$ to be the extension.
\end{proof}

Using the previous two lemmata, we show that a Lipschitz, piecewise $\dc$ function is also $\dc$.

Given a convex and compact set $U \subset \RR^d$ with nonempty interior, $(K_i)_{i=1}^n$ is called a convex tessellation of $U$ if each $K_i$ is convex and compact with nonempty interior, $K_i \cap K_j = (\del K_i) \cap (\del K_j)$, and $U = \bigcup_{i=1}^n K_i$. Note that, for each $i = 1,2,\ldots, n$, $\del K_i \cap \mathrm{int}(U)$ is necessarily polygonal. A standard example of such a tessellation is a triangulation of simplices.

\begin{proposition}\label{P:piecewisedc}
	Assume that $(K_i)_{i=1}^n$ is a convex tessellation of a convex and compact set $U \subset \RR^d$ with nonempty interior. Let $f \in C^{0,1}(U)$ and assume that, for each $i = 1,2\ldots, n$, there exist convex and Lipschitz functions $f^i_1,f^i_2: K_i \to \RR$ with Lipschitz constant $L > 0$ such that $f = f^i_1 - f^i_2$ on $K_i$. Then $f \in \dc(U)$, and, for some constant $C > 0$ depending only on $n$ and $U$,
	\begin{equation}\label{piecewisedcbound}
		\nor{f}{\dc(U)} \le C\pars{ L +  \sum_{i=1}^n\sum_{j=1,2} \nor{f^i_j}{\oo,K}}.
	\end{equation}
\end{proposition}

\begin{proof}
	It follows from Lemma \ref{L:dcextension} that, for each $i = 1,2,\ldots,n$, there exists $\tilde f^i \in \dc(U)$ such that $f = \tilde f^i$ on $K_i$, $\tilde f^i \ge f$ in $U$, and
	\[
		\nor{\tilde f^i}{\dc(U)} \le \nor{f^i_1}{\oo,K_i} + \nor{f^i_2}{\oo,K_i} + 4L\diam{U}.
	\]
	 It is then clear that
	\[
		f = \min\{ \tilde f^1, \tilde f^2, \ldots, \tilde f^n\}.
	\]
	Therefore, as a consequence of Lemma \ref{L:invariant}, $f \in \dc(U)$ and \eqref{piecewisedcbound} holds.
\end{proof}

Proposition \ref{P:piecewisedc} immediately leads to the following generalization of Proposition \ref{P:C11dc}.

\begin{proposition}\label{P:piecewiseC11dc}
	Assume that $(K_i)_{i=1}^n$ is a convex tessellation of a convex and compact set $U \subset \RR^d$ with nonempty interior, $f \in C^{0,1}(U)$, and, for each $i = 1,2,\ldots, n$, $f \in C^{1,1}(K_i)$. Then $f \in \dc(U)$.
\end{proposition}

We now demonstrate that the Hessians of $\dc$-functions can have certain point singularities.

\begin{proposition}\label{P:dcpointsing}
	Let $U \subset \RR^d$ be open and convex and fix $X := \{x_1,x_2,\ldots,x_n\} \subset U$. Assume that
	\[
		f \in C^{0,1}(U) \cap C^{1,1}(U \backslash X)
	\]
	and there exist $C > 0$ and $\sigma_1,\sigma_2,\ldots, \sigma_n \in (0,1)$ such that, for each $i = 1,2,\ldots, n$,
	\[
		\liminf_{x \to x_i} \pars{ D^2 f(x)  + \frac{C}{|x-x_i|^{\sigma_i}} \Id} \ge -C \Id.
	\]
	Then $f \in \dc(U)$.
\end{proposition}

\begin{proof}
	In view of Lemma \ref{L:dclocal} and Proposition \ref{P:C11dc}, it suffices to assume $X = \{0\}$, $U = B_r(0)$ for some $r > 0$, and, for some $\sigma \in (0,1)$ and $C > 0$ and all $x \in B_r(0) \backslash \{0\}$,
	\[
		D^2 f(x) \ge - C\pars{1 + \frac{1}{|x|^\sigma}} \Id.
	\]
	Set
	\[
		g(x) := f(x) + \frac{C}{(2-\sigma)(1-\sigma)} |x|^{2-\sigma} + \nor{Df}{\oo}|x| + \frac{C}{2} |x|^2.
	\]
	We show that $g$ is convex, from which the result follows because $g - f$ is convex, in view of the fact that $2-\sigma > 1$.
	
	First, note that, for $x \ne 0$,
	\[
		D^2 g(x) \ge D^2 f(x) + C \pars{ \frac{1}{|x|^\sigma} + 1} \ge 0.
	\]
	Therefore, $g$ is locally convex in $B_r(0) \backslash \{0\}$. To show that $g$ is convex on all of $B_r(0)$, it sufficies to check that, for all $x \in B_r(0) \backslash \{0\}$,
	\[
		g(x) + g(-x) \ge 2g(0). 
	\]
	This is easily seen from
	\[
		g(x) + g(-x) - 2g(0) \ge f(x) + f(-x) - 2f(0) + 2\nor{Df}{\oo})|x| \ge 0.
	\]
\end{proof}

A particular example of a function satisfying the hypothesis of Proposition \ref{P:dcpointsing} is
\[
	f(x) = a\pars{ \frac{x}{|x|}}|x|^q
\]
for some $a \in C^{1,1}(S^{d-1})$ and $q > 1$. The case $q \ge 2$ is covered by Proposition \ref{P:C11dc}. When $q < 2$, $f$ fails to belong to $C^{1,1}$, but
\[
	D^2 f (x) = |x|^{q-2} \left[ (D^2 a(\hat x) + q a(\hat x))(\Id - \hat x \otimes \hat x) + q(\hat x \otimes Da(\hat x) + Da(\hat x) \otimes \hat x) + q(q-1)a(\hat x) \hat x \otimes \hat x \right],
\]
which clearly satisfies the hypotheses of Proposition \ref{P:dcpointsing} with $\sigma = 2-q$, and so $f \in \dc_\loc$ in view of Proposition \ref{P:dcpointsing}. 

The case $q = 1$ is of particular interest in the pathwise viscosity solution theory. The following result treats this example.

\begin{proposition}\label{P:dclevelset}
	Let $a \in C^{1,1}(S^{d-1})$ and set
	\[
		f(x) = a\pars{ \frac{x}{|x|}}|x|.
	\]
	Then $f \in \dc$, and, for some constant $C > 0$ and all $R > 0$,
	\[
		\nor{f}{\dc(B_R)} \le CR\nor{a}{C^{1,1}}.
	\]
\end{proposition}

\begin{proof}
	If $x \ne 0$, then the positive $1$-homogeneity of $f$ implies that $x$ is an eigenvector of $D^2 f(x)$ with eigenvalue $0$. More precisely,
	\[
		D^2 f(x) = \frac{1}{|x|} \pars{  a(\hat x) + D^2 a(\hat x)} \pars{\Id - \hat x \otimes \hat x}.
	\]
	Set
	\[
		g(x) := \pars{ \nor{a}{\oo} + \nor{D^2 a}{\oo}} |x|.
	\]
	It follows easily that $g$ is convex, and $D^2 g \ge D^2 f$. Therefore, in view of Lemma \ref{L:dcchar}, $f \in \dc$ and
	\[
		\nor{f}{\dc(B_R)} \le 2 \nor{g}{\oo, B_R} + \nor{f}{\oo,B_R} \le (3\nor{a}{\oo} + 2 \nor{D^2 a}{\oo})R.
	\]
\end{proof}

\bibliography{interpolation}{}

\begin{thebibliography}{10}

\bibitem{A1}
{\sc Aleksandrov, A.~D.}
\newblock On surfaces represented as the difference of convex functions.
\newblock {\em Izvestiya Akad. Nauk Kazah. SSR. {\bf 60,} Ser. Mat. Meh. 3\/}
  (1949), 3--20.

\bibitem{A2}
{\sc Aleksandrov, A.~D.}
\newblock Surfaces represented by the differences of convex functions.
\newblock {\em Doklady Akad. Nauk SSSR (N.S.) 72\/} (1950), 613--616.

\bibitem{BL}
{\sc Bergh, J., and L\"{o}fstr\"{o}m, J.}
\newblock {\em Interpolation spaces. {A}n introduction}.
\newblock Springer-Verlag, Berlin-New York, 1976.
\newblock Grundlehren der Mathematischen Wissenschaften, No. 223.

\bibitem{Bill}
{\sc Billingsley, P.}
\newblock {\em Convergence of probability measures}.
\newblock John Wiley \& Sons, Inc., New York-London-Sydney, 1968.

\bibitem{B}
{\sc Brze\'{z}niak, Z.~a.}
\newblock On {S}obolev and {B}esov spaces regularity of {B}rownian paths.
\newblock {\em Stochastics Stochastics Rep. 56}, 1-2 (1996), 1--15.

\bibitem{CFO}
{\sc Caruana, M., Friz, P.~K., and Oberhauser, H.}
\newblock A (rough) pathwise approach to a class of non-linear stochastic
  partial differential equations.
\newblock {\em Ann. Inst. H. Poincar\'{e} Anal. Non Lin\'{e}aire 28}, 1 (2011),
  27--46.

\bibitem{CKR}
{\sc Ciesielski, Z., Kerkyacharian, G., and Roynette, B.}
\newblock Quelques espaces fonctionnels associ\'{e}s \`a des processus
  gaussiens.
\newblock {\em Studia Math. 107}, 2 (1993), 171--204.

\bibitem{CIL}
{\sc Crandall, M.~G., Ishii, H., and Lions, P.-L.}
\newblock User's guide to viscosity solutions of second order partial
  differential equations.
\newblock {\em Bull. Amer. Math. Soc. (N.S.) 27}, 1 (1992), 1--67.

\bibitem{C}
{\sc Cwikel, M.}
\newblock On {$(L^{po}(A_{o}),\,\ L^{p_{1}}(A_{1}))_{\theta },\,_{q}$}.
\newblock {\em Proc. Amer. Math. Soc. 44\/} (1974), 286--292.

\bibitem{DFO}
{\sc Diehl, J., Friz, P.~K., and Oberhauser, H.}
\newblock Regularity theory for rough partial differential equations and
  parabolic comparison revisited.
\newblock In {\em Stochastic analysis and applications 2014}, vol.~100 of {\em
  Springer Proc. Math. Stat.} Springer, Cham, 2014, pp.~203--238.

\bibitem{FGLS}
{\sc Friz, P.~K., Gassiat, P., Lions, P.-L., and Souganidis, P.~E.}
\newblock Eikonal equations and pathwise solutions to fully non-linear {SPDE}s.
\newblock {\em Stoch. Partial Differ. Equ. Anal. Comput. 5}, 2 (2017),
  256--277.

\bibitem{FH}
{\sc Friz, P.~K., and Hairer, M.}
\newblock {\em A course on rough paths}.
\newblock Universitext. Springer, Cham, 2014.
\newblock With an introduction to regularity structures.

\bibitem{FV}
{\sc Friz, P.~K., and Victoir, N.~B.}
\newblock {\em Multidimensional stochastic processes as rough paths}, vol.~120
  of {\em Cambridge Studies in Advanced Mathematics}.
\newblock Cambridge University Press, Cambridge, 2010.
\newblock Theory and applications.

\bibitem{GG}
{\sc Gassiat, P., and Gess, B.}
\newblock Regularization by noise for stochastic {H}amilton-{J}acobi equations.
\newblock {\em Probab. Theory Related Fields 173}, 3-4 (2019), 1063--1098.

\bibitem{GGLS}
{\sc Gassiat, P., Gess, B., Lions, P.-L., and Souganidis, P.~E.}
\newblock Speed of propagation for {H}amilton-{J}acobi equations with
  multiplicative rough time dependence and convex {H}amiltonians.
\newblock {\em Probab. Theory Related Fields 176}, 1-2 (2020), 421--448.

\bibitem{GIP}
{\sc Gubinelli, M., Imkeller, P., and Perkowski, N.}
\newblock Paracontrolled distributions and singular {PDE}s.
\newblock {\em Forum Math. Pi 3\/} (2015), e6, 75.

\bibitem{Hrfirst}
{\sc Hairer, M.}
\newblock Solving the {KPZ} equation.
\newblock {\em Ann. of Math. (2) 178}, 2 (2013), 559--664.

\bibitem{Hrrs}
{\sc Hairer, M.}
\newblock A theory of regularity structures.
\newblock {\em Invent. Math. 198}, 2 (2014), 269--504.

\bibitem{H}
{\sc Hartman, P.}
\newblock On functions representable as a difference of convex functions.
\newblock {\em Pacific J. Math. 9\/} (1959), 707--713.

\bibitem{H-U}
{\sc Hiriart-Urruty, J.-B.}
\newblock Generalized differentiability, duality and optimization for problems
  dealing with differences of convex functions.
\newblock In {\em Convexity and duality in optimization ({G}roningen, 1984)},
  vol.~256 of {\em Lecture Notes in Econom. and Math. Systems}. Springer,
  Berlin, 1985, pp.~37--70.

\bibitem{Hu}
{\sc Hunt, R.~A.}
\newblock On {$L(p,\,q)$} spaces.
\newblock {\em Enseign. Math. (2) 12\/} (1966), 249--276.

\bibitem{I}
{\sc Ishii, H.}
\newblock Hamilton-{J}acobi equations with discontinuous {H}amiltonians on
  arbitrary open sets.
\newblock {\em Bull. Fac. Sci. Engrg. Chuo Univ. 28\/} (1985), 33--77.

\bibitem{K}
{\sc Kruglyak, N.~Y.}
\newblock Smooth analogues of the {C}alder\'{o}n-{Z}ygmund decomposition,
  quantitative covering theorems and the {$K$}-functional for the couple
  {$(L_q,\dot W^k_p)$}.
\newblock {\em Algebra i Analiz 8}, 4 (1996), 110--160.

\bibitem{L}
{\sc Landis, E.~M.}
\newblock On functions representable as the difference of two convex functions.
\newblock {\em Doklady Akad. Nauk SSSR (N.S.) 80\/} (1951), 9--11.

\bibitem{LPinterp}
{\sc Lions, J.-L., and Peetre, J.}
\newblock Sur une classe d'espaces d'interpolation.
\newblock {\em Inst. Hautes \'{E}tudes Sci. Publ. Math.}, 19 (1964), 5--68.

\bibitem{LP}
{\sc Lions, P.-L., and Perthame, B.}
\newblock Remarks on {H}amilton-{J}acobi equations with measurable
  time-dependent {H}amiltonians.
\newblock {\em Nonlinear Anal. 11}, 5 (1987), 613--621.

\bibitem{LR}
{\sc Lions, P.-L., and Rochet, J.-C.}
\newblock Hopf formula and multitime {H}amilton-{J}acobi equations.
\newblock {\em Proc. Amer. Math. Soc. 96}, 1 (1986), 79--84.

\bibitem{LSpreprint}
{\sc Lions, P.-L., and Souganidis, P.~E.}
\newblock In preparation.

\bibitem{LSreg}
{\sc Lions, P.-L., and Souganidis, P.~E.}
\newblock New regularity results and long time behavior of pathwise
  (stochastic) {H}amilton-{J}acobi equations.
\newblock Preprint: arXiv:1909.05672 [math.AP].

\bibitem{LS1}
{\sc Lions, P.-L., and Souganidis, P.~E.}
\newblock Fully nonlinear stochastic partial differential equations.
\newblock {\em C. R. Acad. Sci. Paris S\'{e}r. I Math. 326}, 9 (1998),
  1085--1092.

\bibitem{LS2}
{\sc Lions, P.-L., and Souganidis, P.~E.}
\newblock Fully nonlinear stochastic partial differential equations: non-smooth
  equations and applications.
\newblock {\em C. R. Acad. Sci. Paris S\'{e}r. I Math. 327}, 8 (1998),
  735--741.

\bibitem{LS4}
{\sc Lions, P.-L., and Souganidis, P.~E.}
\newblock Fully nonlinear stochastic pde with semilinear stochastic dependence.
\newblock {\em C. R. Acad. Sci. Paris S\'{e}r. I Math. 331}, 8 (2000),
  617--624.

\bibitem{LS3}
{\sc Lions, P.-L., and Souganidis, P.~E.}
\newblock Uniqueness of weak solutions of fully nonlinear stochastic partial
  differential equations.
\newblock {\em C. R. Acad. Sci. Paris S\'{e}r. I Math. 331}, 10 (2000),
  783--790.

\bibitem{Ly}
{\sc Lyons, T.~J.}
\newblock Differential equations driven by rough signals.
\newblock {\em Rev. Mat. Iberoamericana 14}, 2 (1998), 215--310.

\bibitem{OW}
{\sc Otto, F., and Weber, H.}
\newblock Quasilinear {SPDE}s via rough paths.
\newblock {\em Arch. Ration. Mech. Anal. 232}, 2 (2019), 873--950.

\bibitem{P1}
{\sc Prudnikov, I.~M.}
\newblock Necessary and sufficient conditions for the representability of a
  positive-homogeneous function of three variables as the difference of convex
  functions.
\newblock {\em Izv. Ross. Akad. Nauk Ser. Mat. 56}, 5 (1992), 1116--1128.

\bibitem{P2}
{\sc Prudnikov, I.~M.}
\newblock On the question of the representability of a function of two
  variables as the difference of convex functions.
\newblock {\em Sibirsk. Mat. Zh. 55}, 6 (2014), 1368--1380.

\bibitem{R}
{\sc Roynette, B.}
\newblock Mouvement brownien et espaces de {B}esov.
\newblock {\em Stochastics Stochastics Rep. 43}, 3-4 (1993), 221--260.

\bibitem{ST}
{\sc Seeger, A., and Trebels, W.}
\newblock Embeddings for spaces of {L}orentz-{S}obolev type.
\newblock {\em Math. Ann. 373}, 3-4 (2019), 1017--1056.

\bibitem{Se}
{\sc Seeger, B.}
\newblock Homogenization of pathwise {H}amilton-{J}acobi equations.
\newblock {\em J. Math. Pures Appl. (9) 110\/} (2018), 1--31.

\bibitem{Snotes}
{\sc Souganidis, P.~E.}
\newblock Pathwise solutions for fully nonlinear first- and second-order
  partial differential equations with multiplicative rough time dependence.
\newblock In {\em Singular random dynamics}, vol.~2253 of {\em Lecture Notes in
  Math.} Springer, Cham, [2019] \copyright 2019, pp.~75--220.

\bibitem{Stein}
{\sc Stein, E.~M.}
\newblock {\em Singular integrals and differentiability properties of
  functions}.
\newblock Princeton Mathematical Series, No. 30. Princeton University Press,
  Princeton, N.J., 1970.

\bibitem{SV}
{\sc Stroock, D.~W., and Varadhan, S. R.~S.}
\newblock {\em Multidimensional diffusion processes}, vol.~233 of {\em
  Grundlehren der Mathematischen Wissenschaften [Fundamental Principles of
  Mathematical Sciences]}.
\newblock Springer-Verlag, Berlin-New York, 1979.

\bibitem{Texample}
{\sc Tao, T.}
\newblock private communication.

\bibitem{T}
{\sc Triebel, H.}
\newblock {\em Theory of function spaces. {II}}, vol.~84 of {\em Monographs in
  Mathematics}.
\newblock Birkh\"{a}user Verlag, Basel, 1992.

\bibitem{Z}
{\sc Zalgaller, V.~A.}
\newblock On the representation of a function of two variables as the
  difference of convex functions.
\newblock {\em Vestnik Keningrad. Univ. Ser. Mat. Meh. Astronom. 18}, 1 (1963),
  44--45.

\end{thebibliography}
\bibliographystyle{acm}

\end{document}